\documentclass[11pt,reqno]{amsart}
\usepackage[T1]{fontenc}
\usepackage{graphicx}
\usepackage{cite}

\usepackage{color}
\definecolor{MyLinkColor}{rgb}{0,0,0.4}
\numberwithin{equation}{section}

\newcommand{\im}{\mathop{\rm Im}\nolimits}

\newcommand{\re}{\mathop{\rm Re}\nolimits}

\newcommand{\sign}{\mathop{\rm sign}\nolimits}
\newcommand{\PV}{\mathop{\rm PV}\nolimits}

\newcommand{\0}{\Omega}
\newcommand{\e}{\varepsilon}

\newcommand{\p}{\partial}
\newcommand{\wt}{\widetilde}
\newcommand{\ov}{\overline}

\newcommand{\A}{\mathcal{A}}
\newcommand{\bA}{\mathbb{A}}

\newcommand{\bB}{\mathbb{B}}
\newcommand{\bM}{\mathbb{M}}

\newcommand{\cO}{\mathcal{O}}

\newcommand{\kF}{\mathcal{F}}

\newcommand{\kH}{\mathcal{H}}

\newcommand{\kL}{\mathcal{L}}

\newcommand{\C}{\mathbb{C}}

\newcommand{\E}{\mathbb{E}}
\newcommand{\F}{\mathbb{F}}
\newcommand{\R}{\mathbb{R}}

\newcommand{\N}{\mathbb{N}}
\newcommand{\Z}{\mathbb{Z}}

\DeclareMathOperator{\supp}{supp}
 
\newtheorem{thm}{Theorem}[section]
\newtheorem{prop}[thm]{Proposition}
\newtheorem{lemma}[thm]{Lemma}
\newtheorem{cor}[thm]{Corollary}
\newtheorem{rem}[thm]{Remark}

\numberwithin{equation}{section}

\title[The Muskat problem in 2D]{The Muskat problem in 2D: equivalence of formulations, well-posedness, and regularity results}

\author{Bogdan--Vasile Matioc}
\address{Fakult\"at f\"ur Mathematik, Universit\"at Regensburg,   93040 Regensburg, Deutschland.}
\email{bogdan.matioc@ur.de}


\usepackage[colorlinks=true,linkcolor=MyLinkColor,citecolor=MyLinkColor]{hyperref} 

\begin{document}

\begin{abstract}
In this paper we  consider the Muskat problem describing the motion of two unbounded  immiscible fluid layers   with equal viscosities in vertical or horizontal two-dimensional geometries.
We first prove that  the mathematical model can be formulated as an evolution problem for the  sharp interface separating the two fluids, which turns out to be, in 
a suitable functional analytic setting, 
quasilinear and of parabolic type.
Based upon these properties, we then establish the local well-posedness of the problem for arbitrary large initial data and show that the solutions become  instantly real-analytic in time and space.
Our method allows us to choose the initial data in the class $H^s,$ $s\in(3/2,2)$, when neglecting surface tension, respectively in $H^s,$ $s\in(2,3),$ when surface tension effects are included.
Besides, we provide new criteria for the global existence of solutions.
\end{abstract}

\maketitle

\tableofcontents

\section{Introduction and the main results}\label{S1}

The Muskat problem  is a classical model proposed in \cite{Mu34} to describe the motion of two immiscible fluids in a porous medium or a Hele-Shaw cell.
We consider here the particular case when the fluids have  equal viscosities and we assume that the flows are two-dimensional.
Furthermore, we consider an unbounded geometry corresponding  to fluid layers that occupy the entire space, the fluid motion being  localized and the fluid system close to the rest state far away from the origin.  
We further assume that the fluids are separated by a sharp interface which flattens out at infinity, evolves in time, and is unknown.
We consider two different scenarios for this unconfined Muskat problem: 
\begin{itemize}
 \item[(a)]   in the absence of surface tension effects at the free boundary, the Hele-Shaw cell is vertical and the fluid located below is more dense;
 \item[(b)] in the presence of surface tension effects,  the Hele-Shaw cell is either vertical or horizontal and we make no restrictions on the  densities of the fluids. 
\end{itemize}

One big advantage of considering this setting is that the equations of motion  can be very  elegantly formulated as a single evolution equation for the interface 
between the fluids.
Indeed, parameterizing this interface  as the graph $[y=f(t,x)]$,  the Muskat problem is equivalent in this setting to an evolution problem 
for the unknown function $f$, cf. Section \ref{S1'}, and it reads
\begin{equation}\label{P}
\left\{
\begin{array}{rlll}
 \p_tf(t,x)\!\!&=&\!\!\displaystyle\frac{\sigma k}{2\pi \mu}f'(t,x)\PV\int_\R\frac{f(t,x)-f(t,x-y)}{ y^2+(f(t,x)-f(t,x-y))^2 }(\kappa(f))'(t,x-y)\, dy\\[2ex]
 &&+\displaystyle\frac{\sigma k}{2\pi \mu}\PV\int_\R\frac{y}{ y^2+(f(t,x)-f(t,x-y))^2 }(\kappa(f))'(t,x-y)\, dy\\[2ex]
 &&+\displaystyle\frac{\Delta_\rho k}{2\pi \mu}\PV\int_\R\frac{y(f'(t,x)-f'(t,x-y))}{ y^2+(f(t,x)-f(t,x-y))^2 }\, dy\qquad \text{for $ t>0$, $x\in\R$},\\[2ex]
  f(0,\cdot)\!\!&=&\!\!f_0.
\end{array}
\right.    
\end{equation} 
For brevity we write   $f'$ for  the spatial derivative  $\p_x f.$ 
We let   $k$ denote the permeability of the homogeneous porous medium, $\mu$ is the viscosity coefficient of the fluids, $\sigma$ is the surface tension coefficient at the free boundary,  and
\[\Delta\rho:=g(\rho_--\rho_+),\]
where $g$ is the Earth's gravity and $\rho_\pm$ is the density of the fluid 
which occupies the domain $\0_\pm(t)$ defined by
\[ \0_-(t):=[y<f(t,x)]\qquad\text{and}\qquad \0_+(t):=[y>f(t,x)].\]
Furthermore,  $\kappa(f(t))$ is the curvature of the graph $[y=f(t,x)] $ and
$\PV$ denotes the principal value which, depending on the regularity of the functions under the integral, is taken at zero and/or at infinity.
Our analysis covers the following scenarios 
\[
\text{(a)} \,\, \sigma=0,\, \Delta_\rho>0 \qquad\text{and}\qquad \text{(b)} \,\, \sigma>0,\, \Delta_\rho\in\R,
\]
meaning that    (a)  corresponds to the stable case when   the denser fluid is located below.

 Due to its physical relevance \cite{Be88}, the Muskat problem has been widely studied in  the last decades in several geometries and physical settings and with various methods. 
 When neglecting surface tension effects the well-posedness of the Muskat problem is in strong relationship with the Rayleigh-Taylor condition, being implied by the latter. 
 The Rayleigh-Taylor condition, which appears first in \cite{ST58}, is a sign restriction on the jump of the pressure gradient in  normal direction at the free boundary.
 For fluids with equal viscosities moving  in a vertical geometry, it reduces    to the simple relation
 \[
 \Delta_\rho>0,
 \]
 cf. e.g. \cite{CG10, EMW18} (see also the equations \eqref{eq:S1}-\eqref{eq:S2}).
The first local existence result has been established in \cite{Y96} by using Newton's iteration method, the analysis in \cite{A04, BCG14, BCS16, GG14, CGSV17,CCG11,CCG13b, CG07, CG10,    CGO14}  is based  on energy estimates and the energy method,
 the authors of \cite{SCH04} use  methods from complex analysis and a version of the Cauchy-Kowalewski theorem,  
   a fixed point argument is employed in \cite{BV14} for nonregular initial data, and the approach 
 in \cite{EMM12a, EM11a,  EMW18} relies on the formulation of the problem as a nonlinear and nonlocal parabolic equation together with an abstract well-posedness result from \cite{DG79}  based on
  continuous maximal regularity.
  Other papers study the qualitative aspects of  solutions to the  Muskat problem for fluids with equal viscosities, such as: global existence of strong and weak solutions \cite{CCGS13, CGSV17, GB14},
  existence of initial data for which solutions turn over \cite{ CCFG13, CCFGL12,CGFL11},   the absence of squirt or splash singularities \cite{CG10,   GS14}.
  
  Compared to the zero surface tension case, the Muskat problem with surface tension is less well studied.
  When allowing for surface tension, the Rayleigh-Taylor condition is no longer needed and the problem is well-posed for general initial data.
 While some of the references require  quite high regularity from the initial data, cf. \cite{A14, FT03, HTY97, To17}, optimal results are established in bounded or periodic geometries under the observation that 
 the Muskat problem with surface tension can be formulated as a quasilinear parabolic evolution problem, cf. \cite{EMW18, PS16}.
  
The stability properties of equilibria which are, depending on the physical scenario,  horizontal lines \cite{BCS16, EEM09c, EMM12a, EM11a},   finger-shaped   \cite{EEM09c, EMM12a, EM11a}, 
circular \cite{FT03}, or a union of disjoint circles/spheres \cite{PS16x} 
  have been also addressed  in the references just mentioned.

In this paper we first rigorously prove in Section \ref{S1'} that the Muskat problem in the classical formulation \eqref{PB} and the system \eqref{P} are equivalent for a certain class of solutions.
Thereafter, the   analysis of \eqref{P} starts from the obvious observation that the right-hand side of the first equation  of \eqref{P} is linear with  respect to the highest order spatial derivative of $f,$ 
that is this particular Muskat problem has a quasilinear structure  (also when neglecting surface tension). 
This property  is not obvious in the  particular geometry considered in \cite{ EMW18, Y96} (when $\sigma=0$).
In a suitable functional analytic setting  we  then prove that \eqref{P} is additionally parabolic for general initial data.
The parabolic character was established previously for bounded geometries \cite{EMM12a, EM11a,   EMW18, PS16, PS16x} (in the absence of surface tension effects only when the Rayleigh-Taylor condition holds),
but for \eqref{P} only for small initial data, cf. \cite{CCGS13, CG07}.   
These two aspects, that is the quasilinearity and the parabolicity, enable us to  use   abstract  results for quasilinear parabolic problems  due to H. Amann  \cite[Section 12]{Am93} to prove, by similar strategies, the 
well-posedness of the Muskat problem with and without surface tension. 

It is worth emphasizing that for this particular Muskat problem the local well-posedness is established, in the zero surface tension case, only for initial data that are twice weakly differentiable 
and which  belong to $  W^2_p(\R)$ for some $p\in(1,\infty],$ cf. \cite{CGSV17}.
Our first main result, i.e. Theorem \ref{MT1}, extends the local well-posedness to general initial data in $H^s(\R)$ with $s\in(3/2,2)$. 
For the unconfined  Muskat problem  with surface tension  the  well-posedness is considered only in \cite{A14, To17} and in both papers 
the authors require that $f_0\in H^s(\mathbb{X})$, with $\mathbb{X}\in\{\R,\mathbb{T}\}$ and $s\geq6$.
In our well-posedness result, i.e.  Theorem \ref{MT1K},  the curvature of the initial data may be even unbounded as we allow for general  initial data  $f_0\in H^s(\R)$ with $s\in(2,3)$.
Additionally, we  also obtain new criteria for  the existence of global solutions to the Muskat problem with and without surface tension and, as a consequence of the parabolic character of the equations, 
we show  that the fluid interfaces become instantly real-analytic.

Our strategy is the following: we formulate, in a suitable functional analytic setting,   \eqref{P} as a quasilinear evolution problem of the form
\footnote{ We write $\dot f$ to denote the derivative $d f/d t$.}
\begin{equation*} 
  \dot f=\Phi_\sigma(f)[f],\quad t>0,\qquad f(0)=f_0,
\end{equation*}
and then we study the properties of the operator $\Phi_\sigma$.
We differentiate between the case  $\sigma=0$, studied in the Sections \ref{S2}-\ref{S4},  when we simply write $\Phi_\sigma=:\Phi,$ and the case $\sigma>0$, as  in the first case  $\Phi(f)$ is a nonlocal operator of order $1$ and in the second case  
$\Phi_\sigma(f)$ has order $3$ (for $f$ appropriately chosen).
At the core of our estimates         lies the following deep  result from harmonic analysis:
given a Lipschitz   function $a:\R\to\R,$ the singular integral operator
\begin{align}\label{FFF}
h\mapsto\left[x\mapsto \PV \int_\R \frac{h(x-y)}{y}\exp\Big(i\frac{a(x)-a(x-y)}{y}\Big)\, dy\right],
\end{align}
belongs to $\kL(L_2(\R))$ and its norm is bounded by $C(1+\|a'\|_\infty)$, cf.   \cite{TM86}, with $C$ denoting a universal constant independent of $a$.
Relying on \eqref{FFF}, we study  the mapping properties of $\Phi_\sigma$ and show, for suitable $f$, that $\Phi_\sigma(f)$ is 
the generator of a strongly continuous and real-analytic semigroup.
The main results of this paper, that is the Theorems \ref{MT1}-\ref{MT2}, are  then obtained by employing  abstract results presented in \cite[Section 12]{Am93}, and which we briefly recall at the end of this section.
The line of approach is close to the one we followed in \cite{EMW18}, however the functional analytic setting and the methods used to  establish  the needed estimates are substantially different.
We expect that our method extends to the general case when $\mu_-\neq\mu_+$ and we believe to obtain, for periodic flows, 
a similar stability behavior of the -- flat and finger shaped -- equilibria as in \cite{EM11a}.\medskip

  Our first main result is the following well-posedness theorem for the Muskat problem without surface tension effects.
  \begin{thm}[Well-posedness: no surface tension]\label{MT1}
Let  $\sigma=0$ and $\Delta_\rho>0$. The problem \eqref{P} possesses for each   $f_0\in H^s(\R)$,  $s\in(3/2,2),$ a unique maximal classical solution 
\[ f:=f(\cdot; f_0)\in C([0,T_+(f_0)),H^s(\R))\cap C((0,T_+(f_0)), H^2(\R))\cap C^1((0,T_+(f_0)), H^1(\R)),\]
with $T_+(f_0)\in(0,\infty]$, and $[(t,f_0)\mapsto f(t;f_0)]$ defines a  semiflow on $H^s(\R)$. Additionally, if   
\[
\sup_{[0,T_+(f_0))\cap[0,T]}\|f(t)\|_{H^s}<\infty\qquad\text{for all $T>0$},
\]
then  $T_+(f_0)=\infty$. 
\end{thm}

The quasilinear character of the  problem is enhanced by the presence of surface tension. For this reason  we may consider, when $\sigma>0$,   
initial data with unbounded curvature. We  show in Theorem \ref{MT2} however, that the curvature becomes instantly real-analytic and bounded.  
\begin{thm}[Well-posedness: with surface tension]\label{MT1K}
Let $\sigma>0$ and $\Delta_\rho\in\R$.
 The problem \eqref{P} possesses for each   $f_0\in H^s(\R)$,  $s\in(2,3),$ a unique maximal classical solution
\[ f:=f(\cdot; f_0)\in C([0,T_+(f_0)),H^s(\R))\cap C((0,T_+(f_0)), H^3(\R))\cap C^1((0,T_+(f_0)), L_2(\R)), \] 
with $T_+(f_0)\in(0,\infty]$ and  $[(t,f_0)\mapsto f(t;f_0)]$ defines a  semiflow on $H^s(\R)$.
Additionally, if   
\[
\sup_{[0,T_+(f_0))\cap[0,T]}\|f(t)\|_{H^s}<\infty\qquad\text{for all $T>0$},
\]
then  $T_+(f_0)=\infty$.

\end{thm}

These results  reflect the fact that the Muskat problem without surface tension is a first order evolution problem, while the Muskat problem with surface tension is of third order.   
The solutions obtained in Theorems \ref{MT1} and \ref{MT1K} $(ii)$ become instantly  real-analytic.
 \begin{thm}\label{MT2}
Let $s\in(3/2,2)$ if $\sigma=0$ and $\Delta_\rho>0$, respectively $s\in(2,3)$ if $\sigma>0$. Given $f_0\in H^s(\R)$,  let  $f=f(\cdot; f_0)$ 
denote the unique maximal solution to \eqref{P} found in Theorem \ref{MT1} and \ref{MT1K}, respectively.
 Then
 \[
 [(t,x)\mapsto f(t,x)]:(0,T_+(f_0))\times\R\to\R
 \]
 is a real-analytic function. In particular, $f(t,\cdot)$ is real-analytic for each $t\in(0,T_+(f_0)).$
 Moreover,  given $k\in\N$, it holds that
 \[ 
 f\in C^\omega ((0,T_+(f_0)), H^k(\R)),
 \]
where $C^\omega$ denotes real-analyticity.
 \end{thm}
 
 As a direct consequence  of Theorems \ref{MT1} and \ref{MT2} and of \cite[Theorem 3.1]{CCGS13} (see also \cite[Remark 6.2]{CGCSS14x})  we obtain a global existence result  for  solutions to the Muskat problem 
 without surface tension that correspond to   initial data  of medium size in $H^s(\R),$ $s\in(3/2,2)$.
  In the following $\kF$ denotes the Fourier transform.
\begin{cor}\label{Cor1}
 There exists a  constant $c_0\geq 1/5$ such that for all $f_0\in H^s(\R),$ $s\in(3/2,2)$, with 
 \[|||f_0|||:=\int_\R|\xi|\,|\mathcal F f_0(\xi)|\, d\xi< c_0\]
 the solution found in Theorem \ref{MT1} exists globally. 
\end{cor}
 \begin{proof}
  The claim follows from the inequality
  \[
  |||f|||=\int_\R|\xi|\,|\mathcal F f(\xi)|\, d\xi\leq \|f\|_{H^s}\int_\R\frac{1}{(1+|\xi|^2)^{s-1}}\, d\xi\leq C\|f\|_{H^s}
  \]
  for $s\in(3/2,2)$ and $f\in H^s(\R)$.
 \end{proof}

\noindent{\bf An abstract setting for quasilinear parabolic evolution equations.}
In Theorem \ref{T:A} we collect abstract results from \cite[Section 12]{Am93}
for a general class of abstract quasilinear parabolic evolution equations, which we use in an essential way in our analysis. 

Given   Banach spaces $\E_0, \E_1$ with dense embedding $\E_1\hookrightarrow \E_0 $, we  define   ${\mathcal H}(\E_1,\E_0)$ as the subset of $\kL(\E_1,\E_0)$  
consisting of negative generators of
strongly continuous  analytic semigroups. 
More precisely, $\mathbb{A}\in {\mathcal H}(\E_1,\E_0)$ if
$-\mathbb{A},$ considered as an unbounded operator in $\E_0$ with  domain $\E_1$, generates a strongly continuous  and analytic  semigroup in $\kL(\E_0).$

\begin{thm}\label{T:A}
Let $\E_0, \E_1$ be Banach spaces with dense embedding $\E_1\hookrightarrow \E_0$ and let $\E_\theta:=[\E_0,\E_1]_\theta$ for  $0<\theta<1$ 
 be endowed with the $\|\,\cdot\,\|_\theta$-norm. 
Let further $0<\beta<\alpha<1$ and assume that  
\begin{equation}\label{THT}
-\Phi\in C^{1-}(\cO_\beta, \mathcal{H}(\E_1,\E_0)),
\end{equation}
where $\cO_\beta$ denotes an open subset of $\E_\beta $ and $C^{1-}$ stands for local Lipschitz continuity.
The following assertions hold for the quasilinear evolution problem 
\begin{equation}
  \dot f=\Phi(f)[f],\quad t>0,\qquad f(0)=f_0.\tag{QP}\\[-0.1ex]
\end{equation}
{\bf\em  Existence:} Given  $f_0\in \cO_\alpha:=\cO_\beta\cap \E_\alpha,$ the problem {\rm (QP)}
possesses a   maximal solution  
\[ f:=f(\cdot; f_0)\in C([0,T_+(f_0)),\cO_\alpha)\cap C((0,T_+(f_0)), \E_1)\cap C^1((0,T_+(f_0)), \E_0) \cap C^{\alpha-\beta}([0,T], \E_\beta)\]
for all $T\in(0,T_+(f_0))$, with $T_+(f_0)\in(0,\infty]$. \medskip

\noindent{\bf\em  Uniqueness:} If $\wt T\in(0,\infty]$, $\eta\in(0,\alpha-\beta]$,  and $\wt f \in  C((0,\wt T), \E_1)\cap C^1((0,\wt T), \E_0)$ satisfies 
\[
\wt f\in  C^{\eta}([0,T], \E_\beta) \qquad\text{for all $T\in(0,\wt T)$}
\]
and solves {\rm (QP)}, then $\wt T\leq T_+(f_0)$ and $\wt f=f$ on $[0,\wt T)$.\medskip

\noindent{\bf\em  Criterion for global existence:} If $f:[0,T]\cap[0,T_+(f_0))\to \cO_\alpha$ is uniformly continuous for all $T>0$, then  
\[
T_+(f_0)=\infty\qquad\text{or\qquad $T_+(f_0)<\infty$ and ${\rm dist}(f(t),\p\cO_\alpha)\to 0$ for $t\to T_+(f_0)$}.
\]
\noindent{\bf\em  Continuous dependence of initial data:}
The  mapping $[(t,f_0)\mapsto f(t;f_0)]$ defines a  semiflow on $\cO_\alpha $ and, if   $\Phi\in C^{\omega}(\cO_\beta, \mathcal{L}(\E_1,\E_0))$, then
\[[(t,f_0)\mapsto f(t;f_0)]:\{(t,f_0)\,:\, f_0\in\cO_\alpha, t\in(0,T_+(f_0))\}\to \E_\alpha\]
is a real-analytic map too.
\end{thm}

As usual, $[\cdot,\cdot]_\theta$ denotes the complex interpolation functor.
We choose for our particular problem $\E_i\in\{H^s(\R)\,:\, 0\leq s\leq 3\},$ $i=1,2$, and in this context we rely on the well-known interpolation property 
\begin{align}\label{IP}
[H^{s_0}(\R),H^{s_1}(\R)]_\theta=H^{(1-\theta)s_0+\theta s_1}(\R),\qquad\theta\in(0,1),\, -\infty< s_0\leq s_1<\infty,
\end{align}
cf. e.g. \cite[Remark 2, Section 2.4.2]{Tr78}.

The proof of Theorem 1.5 uses to a large extent the linear theory developed in \cite[Chapter II]{Am95}.
 The main ideas of the proof of Theorem 1.5  can be found in the references  \cite{AM86, Am88}.
 The uniqueness claim in Theorem \ref{T:A} is slightly stronger compared to the result 
 in \cite[Section 12]{Am93} and it turns out to be quite useful when establishing the uniqueness
  in the Theorems \ref{MT1}-\ref{MT1K}. For this reason  we present in the Appendix \ref{S:B}  the proof of Theorem \ref{T:A}.

In order to use Theorem \ref{T:A} in the study of the Muskat problem \eqref{P}, we have to  write this 
evolution problem in the form {\rm (QP)} and to established then the property \eqref{THT}.
With respect to this goal, we use the estimated provided in \eqref{FFF} and many techniques of nonlinear analysis.

\section{The equations of motion and the equivalent formulation}\label{S1'}

 In this section we  present  the   equations  governing the dynamic of the fluids system  and   we prove, for a certain class of solutions, that the latter are equivalent to the system \eqref{P}.
 The Muskat problem was originally proposed as a model for the encroachment of water into an oil sand, and therefore it is natural to assume that both  fluids are incompressible, of  Newtonian type, and immiscible.
 Since for flows in porous media the conservation of momentum equation can be replaced by Darcy's law, cf. e.g. \cite{Be88},
 the equations governing the dynamic of the fluids are
\begin{subequations}\label{PB}
\begin{equation}\label{eq:S1}
\left\{\begin{array}{rllllll}
{\rm div}\,  v_\pm(t)\!\!&=&\!\!0&\text{in  $\Omega_\pm(t)$} , \\[1ex]
v_\pm(t)\!\!&=&\!\!-\cfrac{k}{\mu}\big(\nabla p_\pm(t)+(0,\rho_\pm g)\big)&\text{in $ \0_\pm(t)$} 
\end{array}
\right.
\end{equation} 
 for $t> 0$, where, using the subscript $\pm$ for the fluid located at $\0_\pm(t)$, 
  $v_\pm(t):=(v_\pm^1(t),v_\pm^2(t))$ denotes the velocity vector and $p_\pm(t)$ the pressure of the fluid $\pm.$ 
These equations are supplemented by  the natural  boundary conditions at the free surface
\begin{equation}\label{eq:S2}
\left\{\begin{array}{rllllll}
p_+(t)-p_-(t)\!\!&=&\!\!\sigma\kappa(f(t))&\text{on $ [y=f(t,x)]$}, \\[1ex]
 \langle v_+(t)| \nu(t)\rangle\!\!&=&\!\!  \langle v_-(t)| \nu(t)\rangle &\text{on $[y=f(t,x)]$},
\end{array}
\right.
\end{equation} 
where  $\nu(t) $ is the unit normal at $[y=f(t,x)]$ pointing into $\0_+(t)$ and   $\langle \, \cdot\,|\,\cdot\,\rangle$  the inner product on $\R^2$.
Furthermore, the  far-field boundary  conditions
 \begin{equation}\label{eq:S3}
 \left\{\begin{array}{llllll}
f(t,x)\to0 &\text{for  $|x|\to\infty$,}\\[1ex]
v_\pm(t,x,y)\to0 &\text{for  $|(x,y)|\to\infty$}, 
\end{array}
\right.
\end{equation} 
state that  the fluid motion is localized, the  fluids being close to the rest state far way from the origin.  
The motion of the interface $[y=f(t,x)]$ is coupled to that of the fluids through  the kinematic boundary condition
 \begin{equation}\label{eq:S4}
 \p_tf(t)\, =\, \langle v_\pm(t)| (-f'(t),1)\rangle \qquad\text{on   $ [y=f(t,x)].$}
\end{equation}  
Finally, the interface at time $t=0$ is  assumed to be known
\begin{equation}\label{eq:S5}
f(0)\, =\, f_0.
\end{equation} 
\end{subequations}

The equations \eqref{PB} are known as the Muskat problem and they determine  completely the dynamic of the system.
We now show that   the Muskat  problem \eqref{PB} is equivalent to  the system \eqref{P} presented in the introduction.
The proof uses  classical results on Cauchy-type integrals defined on regular curves, see e.g. \cite{JKL93}.   
More precisely, we establish the following equivalence result.
\begin{prop}[Equivalence of the two formulations]\label{PE}
Let $\sigma\geq0$ and $T\in(0,\infty].$
 The following are equivalent:
 \begin{itemize}
 \item[$(i)$] the Muskat problem \eqref{PB} for\footnote{The regularity $f(t)\in H^5(\R)$, $t\in(0,T),$ is not optimal, that is the two formulations
 are still equivalent if $f(t)\in H^r(\R),$ $t\in(0,T)$, for   $r<5$   suitably chosen.
 In fact, if $\sigma=0,$ we may take $r=3.$
 However, as stated in   Theorem \ref{MT2}, $f(t)\in H^\infty(\R)$ for all $t\in(0,T)$,  and there is no reason for seeking the optimal range for $r$.}
 \begin{align*}
  \bullet &\quad f\in   C^1((0,T), L_2(\R))\cap C([0,T), L_2(\R)), \, f(t)\in H^5(\R) \,\, \text{for all $t\in(0,T)$},\\[1ex]
  \bullet &\quad v_\pm(t)\in C(\ov{\0_\pm(t)})\cap C^1({\0_\pm(t)}),  \, p_\pm(t)\in C^1(\ov{\0_\pm(t)})\cap C^2({\0_\pm(t)}) \, \, \text{for all $t\in(0,T);$}
 \end{align*} 

\item[$(ii)$] the evolution problem \eqref{P} for 
$$f\in   C^1((0,T), L_2(\R))\cap C([0,T), L_2(\R)), \, f(t)\in H^5(\R) \,\, \text{for all $t\in(0,T)$}.$$
\end{itemize}
\end{prop}
\begin{proof}
We first establish the implication $(i)$ $\Rightarrow$ $(ii)$.  
Assuming that we are given a solution to \eqref{PB} as in $(i)$, we have to show that the first equation of \eqref{P} holds for each   $t\in(0,T)$. 
Therefore, we fix $t\in(0,T)$ and we  do not write in the arguments that follow  the dependence of the physical variables of time $t$ explicitly.
In the following ${\bf 1}_E$ is the characteristic function of the set $E.$
Introducing the global velocity field  $v:=(v^1,v^2):=v_-{\bf 1}_{[y\leq f(x)]}+v_+{\bf 1}_{[y> f(x)]}$, Stokes' theorem together with  \eqref{eq:S1} and \eqref{eq:S2} yields that
the vorticity, which for two-dimensional flows corresponds to  the scalar function $\omega:=\p_xv^2-\p_yv^1,$ is supported on the free boundary, that is 
\[\langle \omega , \varphi\rangle=  \int_\R\ov\omega(x)\varphi(x,f(x))\, dx\qquad\text{for all $\varphi\in C^\infty_0(\R^2)$,}\]
where
$$\ov\omega:=\frac{k}{\mu}[\sigma\kappa(f)-\Delta_\rho f]'.$$
We next prove that the velocity is defined by the Biot-Savart law, that is $v=\wt v$ in $\R^2\setminus [y=f(x)]$, where 
\begin{equation}\label{V1}
 \wt v(x,y):=\frac{1}{2\pi} \int_\R\frac{(-(y-f(s)),x-s)}{ (x-s)^2+(y-f(s))^2 }\ov\omega(s)\, ds\qquad \text{in $\R^2\setminus[y=f(x)]$.}
\end{equation}
To this end we compute  the limits  $\wt v_-(x,f(x))$ and $\wt v_+(x,f(x))$  of $\wt v$ at $(x,f(x))$ when we approach this point
from below the interface $[y=f(x)] $ or from above, respectively.
Using the well-known Plemelj formula, cf. e.g. \cite{JKL93}, we find due to the fact that  $f\in H^4(\R)$ and after changing variables, the following expressions
\begin{align}
\wt v_\pm(x,f(x))=&\frac{1}{2\pi}\PV\int_\R\frac{(-(f(x)-f(x-s)),s)}{ s^2+(f(x)-f(x-s))^2 }\ov\omega(x-s)\, ds \mp\frac{1}{2}\frac{(1,f'(x))\ov\omega(x)}{1+{f'}^2(x)},\qquad x\in\R,\label{V2}
\end{align}
where the principal value needs to be taken only at $0$.
In view of Lemma \ref{L:A2} and of $f\in H^5(\R),$   the restrictions $\wt v_\pm$ of $\wt v$ to $\0_\pm$ satisfy $\wt v_\pm\in C(\ov{\0_\pm})\cap C^1({\0_\pm})$ and moreover
   $\wt v_\pm$ vanish at infinity.
Next, we define  the pressures $\wt p_\pm\in C^1(\ov{\0_\pm})\cap C^2(\0_\pm)$  by the formula
\begin{align}\label{V3}
\wt p_\pm(x,y):=c_\pm-\frac{\mu}{k}\int_0^x\wt v_\pm^1(s,\pm d)\, ds-\frac{\mu}{k}\int_{\pm d}^y\wt v_\pm^2(x,s)\, ds-\rho_\pm gy, \qquad (x,y)\in\ov\0_\pm,
\end{align}
where $d$ is a positive constant satisfying  $d>\|f\|_ \infty$ and $c_\pm\in\R$. 
For a proper choice of the constants $c_\pm$, it is not difficult to see that the pair $(\wt p_\pm,\wt v_\pm)$ satisfies all the equations \eqref{eq:S1}-\eqref{eq:S3}.
Let $V_\pm:=v_\pm-\wt v_\pm$,  $V:=(V^1,V^2):=V_-{\bf 1}_{[y\leq f(x)]}+V_+{\bf 1}_{[y> f(x)]}\in C(\R^2)$, and  
\[
\psi_\pm(x,y):=\int^{y}_{f(x)}V_\pm ^1(x,s)\, ds-\int_0^x  \langle V_\pm(s,f(s))| (-f'(s),1)\rangle \, ds\qquad\text{for $(x,y)\in\ov \0_\pm,$}
\]
be the stream function  associated to $V_\pm$. 
Recalling  \eqref{eq:S1}-\eqref{eq:S3},  we deduce that the function  $\psi:=\psi_-{\bf 1}_{[y\leq f]}+\psi_+{\bf 1}_{[y>f]}$ satisfies $\Delta\psi=0$ in $\mathcal{D}'(\R^2)$. 
Hence, $\psi$ is the real part of a holomorphic function $u:\C\to\C$.
Since $u'$ is also holomorphic and $u'=\p_x\psi-i\p_y\psi=-(V^2,V^1)$ is bounded and vanishes for $|(x,y)|\to\infty$ it follows that $u'=0$, hence $V=0$.
This proves that $v_\pm=\wt v_\pm$.

We now infer  from \eqref{eq:S4}   and \eqref{V2} that the dynamic of the free boundary separating the fluids is described by the evolution  equation 
\begin{equation*} 
\begin{array}{rlll}
 \p_tf(t,x)\!\!&=&\!\!\displaystyle\frac{k}{2\pi \mu}f'(t,x)\PV\int_\R\frac{f(t,x)-f(t,x-s)}{ s^2+(f(t,x)-f(t,x-s))^2 }[\sigma\kappa(f)-\Delta_\rho f]'(t,x-s)\, ds\\[2ex]
 &&+\displaystyle\frac{k}{2\pi \mu}\PV\int_\R\frac{s}{ s^2+(f(t,x)-f(t,x-s))^2 }[\sigma\kappa(f)-\Delta_\rho f]'(t,x-s)\, ds
\end{array}
\end{equation*} 
for $ t>0$ and  $x\in\R$.
This equation  can be further simplified by using the formula
\begin{align*}
 \int\limits_{\delta<|x|<1/\delta}\frac{\p}{\p s}\big(\ln(s^2+(f(x)-f(x-s))^2)\big) \, ds =\ln\frac{1+\delta^{2}(f(x)-f(x-1/\delta))^2}{1+\delta^{2}(f(x)-f(x+1/\delta))^2}\frac{1+\cfrac{(f(x)-f(x+\delta))^2}{\delta^2}}{1+\cfrac{(f(x)-f(x-\delta))^2}{\delta^2}}
\end{align*}
for  $\delta\in(0,1)$ and $x\in\R$.  
Letting $\delta\to0$, we get 
\begin{align*}
 0=&\frac{1}{2}\PV\int_\R\frac{\p}{\p s}\big(\ln(s^2+(f(x)-f(x-s))^2)\big) \, ds\\[1ex]
 =& \PV\int_\R\frac{s}{ s^2+(f(t,x)-f(t,x-s))^2 } \, ds+\PV\int_\R\frac{(f(t,x)-f(t,x-s)f'(t,x-s)}{ s^2+(f(t,x)-f(t,x-s))^2 }\, ds,
\end{align*}
and now the principal value needs to be taken in the first integral at zero and at infinity.
Using this identity, we have shown that the mapping $[t\to f(t)]$ satisfies  the  evolution problem \eqref{P}.

The implication $(ii)$ $\Rightarrow$ $(i)$ is now obvious.
\end{proof}

\section{The Muskat problem without surface tension:   mapping properties }\label{S2}

 
In the Sections \ref{S2} and \ref{S3} we consider the stable case mentioned at (a).
In this regime, after rescaling time, we may rewrite \eqref{P} in the following abstract form
\begin{equation}\label{P1}
  \dot f=\Phi(f)[f],\quad t>0,\qquad f(0)=f_0,
\end{equation}
where  $\Phi(f)$ is the linear operator    formally defined by
\begin{align}\label{opa} 
 \Phi(f)[h](x):=&\PV\int_\R\frac{y(h'(x)-h'(x-y))}{y^2+(f(x)-f(x-y))^2}\, dy.
\end{align}
We show in the next two sections that the mapping $\Phi$ satisfies all the assumptions of Theorem \ref{T:A} if we make the following choices: 
  $\E_0:=H^1(\R)$, $\E_1:=H^2(\R),$ $\E_\alpha=H^s(\R)$ with $s\in(3/2,2)$, and $\cO_\beta:=H^{\ov s}(\R)$ with $\ov s\in(3/2,s).$
The first goal is to prove that
\begin{align}\label{LLL} 
 \Phi\in C^{1-}(H^s(\R),\kL(H^2(\R), H^1(\R)))  
\end{align}
for each $s\in(3/2,2)$.
Because the property \eqref{LLL} holds for all $s\in(3/2,2)$., the parameter $\ov s$ will appear only in the proof of Theorem \ref{MT1}, which we present at the end of Section \ref{S3}.

For the sake of brevity we set
 \begin{align*}
   \delta_{[x,y]}f:=f(x)-f(x-y)  \qquad \text{for  $x,y\in\R,$} 
 \end{align*}
and therewith
\begin{align*} 
 \Phi(f)[h](x)=&\PV\int_\R\frac{\delta_{[x,y]}h'/y}{1+\big(\delta_{[x,y]}f/y\big)^2}\, dy.
\end{align*}

\paragraph{\bf Boundedness of some multilinear singular integral operators} 
We first consider  some multilinear operators  which are related to $\Phi$.\footnote{As usual, the empty product is set to be equal to $1$.}
The estimates in Lemmas \ref{L21} and \ref{L21'} enable us in particular to establish the regularity property \eqref{LLL}.
 Lemma \ref{L21} is reconsidered later on, cf.  Lemma \ref{L:IA3}, in a particular context when showing that  $\Phi$ is in fact real-analytic.  

  \begin{lemma}\label{L21} Given $n, m \in\N$,  $r\in(3/2,2)$, $a_1,  \ldots, a_{n+1}, b_1,\ldots,b_{m}\in H^r(\R)$, and a function $c\in L_2(\R)$ we define
\[
 A_{m,n}(a_1, \ldots, a_{n+1})[b_1, \ldots,b_{m}, c](x):=\PV\int_\R\frac{\prod_{i=1}^{m}\big(\delta_{[x,y]} b_i/y\big )}{\prod_{i=1}^{n+1}\big[1+\big(\delta_{[x,y]} a_i /y\big)^2\big]}\frac{ \delta_{[x,y]}c}{y} \, dy.
\]
Then:
\begin{itemize}
 \item[$(i)$] There exists a constant $C$, depending only on $r,$ $n,$ $m,$ and $\max_{i=1,\ldots, n+1}\|a_i\|_{H^r}$, such that  
\begin{equation}\label{DE}
\big\|A_{m,n}(a_1, \ldots, a_{n+1})[b_1, \ldots,b_{m}, c]\big\|_{2}\leq C\|c\|_2\prod_{i=1}^m\|b_i\|_{H^r}
\end{equation}
for all $  b_1,\ldots,b_m\in H^r(\R) $ and  $c\in L_2(\R)$. 
\item[$(ii)$] $A_{m,n}\in C^{1-}((H^r(\R))^{n+1},\kL_{m+1}((H^r(\R))^m\times L_2(\R), L_2(\R))$.
\end{itemize}
\end{lemma}

\begin{rem}\label{R1}
 We note that \begin{align} 
 \Phi(f)[h]=A_{0,0}(f)[h']\label{N0}
\end{align} 
for all $f\in H^s(\R),$ $s\in(3/2,2)$,  and $h\in H^2(\R)$, and
 \[
 A_{0,0}(0)[c](x)=\PV\int_\R\frac{ \delta_{[x,y]}c }{y} \, dy=-\PV\int_\R\frac{ c(x-y) }{y} \, dy=-\pi Hc(x),
 \]
 where $H$ denotes the Hilbert transform \cite{St93}.
\end{rem}

\begin{rem}\label{R2} 
In the proof of Lemma \ref{L21} we split the operator  $A_{m,n}:=A_{m,n}(a_1, \ldots, a_{n+1})$ into two operators
 \[
 A_{m,n}=A_{m,n}^1-A_{m,n}^2.
 \]
If we keep   $b_1,\ldots, b_{m}$ fixed, then $A_{m,n}^1$ is   a multiplication type operator
 \[
 A_{m,n}^1[b_1, \ldots,b_{m}, c](x):=c(x)\PV\int_\R\frac{1}{y}\frac{\prod_{i=1}^{m}\big(\delta_{[x,y]} b_i/y\big)}{\prod_{i=1}^{n+1}\big[1+\big(\delta_{[x,y]} a_i /y\big)^2\big]} \, dy,
 \]
 while $A_{m,n}^2$ is the singular integral operator 
\[
 A_{m,n}^2[b_1, \ldots,b_{m}, c](x):=\PV\int_\R  K(x,y)c(x-y) \, dy,
 \]
  with the kernel $K$ defined by 
 \[
 K(x,y):=\frac{1}{y}\cfrac{\prod_{i=1}^{m}\big(\delta_{[x,y]} b_i / y\big)}{\prod_{i=1}^{n+1}\big[1+\big( \delta_{[x,y]} a_i/y\big)^2\big]}\qquad \text{for $x\in\R$, $y\neq0$}. 
\]
Our proof shows that both operators $A_{m,n}^i, 1\leq i\leq 2,$ satisfy \eqref{DE}.
While the    boundedness of $A_{m,n}^1$ follows by direct computation, the boundedness of $A_{m,n}^2$ follows from the estimate on the norm of operator defined in \eqref{FFF} and an argument due to 
  Calder\'on as it appears in the proof of  \cite[Theorem 9.7.11]{CM97}.
In fact, the arguments in the proof of Lemma \ref{L21} show that given Lipschitz functions $a_1,\ldots, a_{n+m}:\R\to\R$  the singular integral operator 
\[
B_{n,m}(a_1,\ldots, a_{n+m})[h](x):=\PV\int_\R  \frac{h(x-y)}{y}\cfrac{\prod_{i=1}^{n}\big(\delta_{[x,y]} a_i /y\big)}{\prod_{i=n+1}^{n+m}\big[1+\big(\delta_{[x,y]}  a_i /y\big)^2\big]}\, dy ,
\]
belongs to $\kL(L_2(\R))$ and $\|B_{n,m}(a_1,\ldots, a_{n+m})\|_{\kL(L_2(\R))}\leq C\prod_{i=1}^{n} \|a_i'\|_\infty$ where $C$ is a constant depending only   on $n, m$ and $\max_{i=n+1,\ldots, n+m}\|a_i'\|_{\infty}.$

It is worth to point out that $B_{0,0}=B_{0,1}(0)=\pi H.$
\end{rem}

\begin{proof}[Proof of Lemma \ref{L21}]  
The multilinear operator $A_{m,n}^1$ is bounded provided that the mapping
\[
\Big[x\mapsto \PV\int_\R\frac{1}{y}\frac{\prod_{i=1}^{m}\big(\delta_{[x,y]} b_i/y\big)}{\prod_{i=1}^{n+1}\big[1+\big(\delta_{[x,y]} a_i /y\big)^2\big]} \, dy\Big]
\]
belongs to $L_\infty(\R)$.
To establish this boundedness property we note that
\begin{align*}
\int\limits_{\delta<|y|<{1/\delta}}\frac{1}{y}\frac{\prod_{i=1}^{m}\big(\delta_{[x,y]} b_i/y\big)}{\prod_{i=1}^{n+1}\big[1+\big(\delta_{[x,y]} a_i /y\big)^2\big]} \, dy
 =&\int_\delta^{1/\delta}\frac{1}{y}\frac{\prod_{i=1}^{m}\big(\delta_{[x,y]} b_i/y\big)}{\prod_{i=1}^{n+1}\big[1+\big(\delta_{[x,y]} a_i /y\big)^2\big]}\\[1ex]
 &-\frac{1}{y}\frac{\prod_{i=1}^{m}\big(-\delta_{[x,-y]} b_i/y\big)}{\prod_{i=1}^{n+1}\big[1+\big(\delta_{[x,-y]} a_i /y\big)^2\big]} \, dy\\[1ex]
  =:&\int_\delta^{1/\delta} I(x,y)\, dy, 
  \end{align*}
  for $\delta\in(0,1)$ and $x\in\R$, where   
  \begin{align*}
 I(x,y):=&\frac{1}{y}\frac{1}{\prod_{i=1}^{n+1}\big[1+\big(\delta_{[x,y]} a_i /y\big)^2\big]}\Big(\prod_{i=1}^{m}\big(\delta_{[x,y]} b_i/y \big)-  \prod_{i=1}^{m}\big(-\delta_{[x,-y]}b_i/y\big)\Big)\\[1ex]
 &+ \frac{1}{y}\Big(\prod_{i=1}^{m}\big(-\delta_{[x,-y]} b_i/y \big)\Big) 
 \frac{\prod_{i=1}^{n+1}\big[1+\big(\delta_{[x,-y]} a_i /y\big)^2\big]-\prod_{i=1}^{n+1}\big[1+\big(\delta_{[x,y]} a_i /y\big)^2\big]}{\prod_{i=1}^{n+1}\big[1+\big(\delta_{[x,y]} a_i /y\big)^2\big]\big[1+\big(\delta_{[x,-y]} a_i /y\big)^2\big]}.
\end{align*}
We further have
\begin{align*}
 &\hspace{-1cm}\frac{1}{y} \Big(\prod_{i=1}^{m}\big(\delta_{[x,y]} b_i/y \big)-  \prod_{i=1}^{m}\big(-\delta_{[x,-y]}b_i/y\big)\Big)=\frac{1}{y}\prod_{i=1}^{m}\frac{b_i(x)-b_i(x-y)}{y}-   \frac{1}{y}\prod_{i=1}^{m}\frac{b_i(x+y)-b_i(x)}{y}\\[1ex]
 =&-\sum_{i=1}^{m}\frac{b_i(x+y)-2b_i(x)+b_i(x-y)}{y^2}\Big[\prod_{j=1}^{i-1}\big(\delta_{[x,y]} b_j/y \big)\Big] \Big[\prod_{j=i+1}^{m}\big(-\delta_{[x,-y]} b_{j}/y\big)\Big],
\end{align*}
and similarly
\begin{align*}
 &\hspace{-1cm}\frac{1}{y}\Big( \prod_{i=1}^{n+1}\big[1+\big(\delta_{[x,-y]} a_i /y\big)^2\big]-\prod_{i=1}^{n+1}\big[1+\big(\delta_{[x,y]} a_i /y\big)^2\big]\Big)\\[1ex]
 =&\sum_{i=1}^{n+1}\Big[\prod_{j=1}^{i-1}\big[1+\big(\delta_{[x,-y]} a_j /y\big)^2\big] \Big[\prod_{j=i+1}^{n+1}\big[1+\big(\delta_{[x,y]} a_j /y\big)^2\big]\Big]\\[1ex]
 &\hspace{0.85cm}\times\frac{a_i(x+y)-a_i(x-y)}{y}\frac{a_i(x+y)-2a_i(x)+a_i(x-y)}{y^2}.
\end{align*}
Let us now observe that
\begin{align}\label{E1}
 |I(x,y)|\leq & \frac{2^{m+1}[1+4(n+1)\max_{i=1,\ldots,n+1}\|a_i\|_\infty^2]}{y^{2}} \prod_{i=1}^{m}\|b_i\|_\infty\qquad \text{for  $x\in\R$, $y\geq1$.}
\end{align}
Furthermore, since $r-1/2\in (1,2)$, we find, by taking advantage of $ H^r(\R)\hookrightarrow {\rm BC}^{r-1/2}(\R) $, that
 \begin{align}\label{MM}
  \frac{|f(x+y)-2f(x)+f(x-y)|}{y^{r-1/2}}\leq 4[f']_{r-3/2}\leq C \|f\|_{H^r} \qquad\text{for all $f\in H^r(\R)$, $x\in\R$, $y>0$,}
\end{align}
cf. \cite[Relation (0.2.2)]{L95}. Here $[\,\cdot\,]_{r-3/2}$ denotes the usual H\"older seminorm.  
Using \eqref{MM}, it follows that 
\begin{align}\label{E2}
 |I(x,y)|\leq & Cy^{r-5/2}\left[\sum_{i=1}^{m} \Big(\|b_i\|_{H^r}\prod_{{j=1, j\neq i}}^{m}\|b_j'\|_\infty\Big)+ \Big(\prod_{i=1}^{m}\|b_i'\|_\infty\Big)\sum_{i=1}^{n+1}\|a_i'\|_\infty\|a_i\|_{H^r}\right] 
\end{align}
for $x\in\R$, $y\in(0,1)$.
Combining  \eqref{E1} and \eqref{E2} yields   
\[
 \sup_{x\in\R}\int_0^\infty |I(x,y)|\, dy\leq C \prod_{ i=1}^{m}\|b_i\|_{H^r},
\]
  where $C$ depends only of $r, $ $n,$ $m,$ and $\max_{i=1,\ldots, n+1}\|a_i\|_{H^r}$. The latter estimate shows that \eqref{DE} is satisfied when replacing $A_{m,n}$ by $A_{m,n}^1$.

To deal with $A_{m,n}^2,$ we define the functions $F:\R^{n+m+1}\to\R$ and $A:\R\to\R^{n+m+1}$  by 
\[
F(u_1,\ldots, u_{n+1} , v_1, \ldots, v_{m})=\frac{\prod_{i=1}^{m}v_i}{\prod_{i=1}^{n+1}(1+u_i^2)}\qquad\text{and}\qquad A :=(a_1,\ldots, a_{n+1},b_1,\ldots, b_{m}),
\]
where $b_i\in H^r(\R) $ satisfy $\|b_i'\|_\infty\leq 1, 1\leq i\leq m.$
The function $F$ is smooth and $A$ is Lipschitz continuous  with a Lipschitz constant $L:=  \sqrt{m+(n+1)\max_{i=1,n+1}\|a_i'\|_{\infty}^2}\geq \|A'\|_\infty$. 
We further observe that
\[
K(x,y)=\frac{1}{y}F\Big(\frac{\delta_{[x,y]}A}{y}\Big)
\]
with $|\delta_{[x,y]}A/y|\leq L.$
Let $\wt F$ be a smooth function on $\R^{n+m+1}$ which is $4L$-periodic in each variable and which matches $F$ on $[-L,L]^{n+m+1}$.
Expanding $\wt F$ by its  Fourier series 
\[
\wt F=\sum_{p\in \Z^{n+m+1}}\alpha_p e^{i(\pi/2L)\langle p\,  | \, \cdot\, \rangle },
\]
the associated sequence $(\alpha_p)_p$ is   rapidly decreasing.
Furthermore, we can write the kernel $K$ as follows
\[
K(x,y)=\sum_{p\in \Z^{n+m+1}}\alpha_p K_p(x,y), \qquad x\in\R,\,  y\neq 0,
\]
with 
\[
K_p(x,y):=\frac{1}{y}\exp\Big(i\frac{\pi}{2L}\frac{\delta_{[x,y]}\langle p |   A\rangle}{y}\Big),\qquad x\in\R,\, y\neq 0,\, p\in \Z^{n+m+1}.
\]
The kernels $K_p, p\in\Z^{n+m+1},$ define  operators in $\kL(L_2(\R))$ of the type \eqref{FFF} and  with norms bounded by 
$$C\Big(1+\frac{\pi}{2L}|p|\|A'\|_\infty\Big)\leq C (1+|p|) ,\qquad   p\in \Z^{n+m+1},$$
with a universal constant $C$ independent of $p$.
Hence, the associated series is absolutely convergent in  $\kL(L_2(\R))$, meaning that the operator 
$A_{m,n}^2(a_1, \ldots, a_{n+1})[b_1, \ldots,b_{m}, \cdot]$ belongs to $\kL(L_2(\R))$ and
\[
\|A_{m,n}^2(a_1, \ldots, a_{n+1})[b_1, \ldots,b_{m}, c]\|_{2}\leq C\big(n,m, \max_{i=1,\ldots, n+1}\|a_i'\|_{\infty}\big)\|c\|_2
\]
for all $c\in L_2(\R)$ and for all $b_i\in H^r(\R) $ that satisfy $\|b_i'\|_\infty\leq 1.$
The desired estimate  \eqref{DE} follows now by using the linearity of $A_{m,n}^2$ in each argument.
The claim $(i)$ is now obvious.

Concerning $(ii)$, we note that
\begin{align*}
 &\hspace{-1cm}A_{m,n}(\wt a_{1}, \ldots, \wt a_{n+1})[b_1,\ldots, b_{m},c]-  A_{m,n}(a_1, \ldots, a_{n+1})[b_1,\ldots, b_{m},c]\\[1ex]
 &=\sum_{i=1}^{n+1} A_{m+2,n+1}(\wt a_{1},\ldots, \wt a_{i},a_i,\ldots \ldots, a_{n+1})[a_i+\wt a_{i}, a_i-\wt a_{i},b_1,\ldots, b_{m},c],
\end{align*}
and the desired assertion follows now from $(i)$.
\end{proof}\medskip

 We consider  once more   the operators $A_{m,n}$ defined in Lemma \ref{L31} in the case when $m\geq1$, but defined on a different Hilbert space product where a weaker regularity of $b_m$ is balanced by a higher regularity of the variable $c$.
 The estimates in Lemma \ref{L21'} are slightly related to the ones announced  in \cite[Theorem 4]{CC78}  and, excepting for the later reference, we did not find similar results. 
 
  \begin{lemma}\label{L21'} Let $n\in\N$, $1\leq m\in  \N$, $r\in(3/2,2)$, $\tau\in(5/2-r,1)$, and $a_1,  \ldots, a_{n+1}\in H^r(\R)$ be given.
Then:
\begin{itemize}
 \item[$(i)$] There exists a constant $C$, depending only on $r $  and $\tau$    such that  
\begin{equation*}
\big\|A_{m,n}(a_1, \ldots, a_{n+1})[b_1, \ldots,b_{m}, c]\big\|_{2}\leq C\|c\|_{H^\tau}\|b_m\|_{H^{r-1}}\prod_{i=1}^{m-1}\|b_i'\|_{\infty}
\end{equation*}
for all $ b_1,\ldots,b_{m}\in H^r(\R) $ and all $c\in H^1(\R)$. In particular, $A_{m,n}(a_1, \ldots, a_{n+1})$ extends to a bounded operator 
$$A_{m,n}(a_1, \ldots, a_{n+1})\in\kL_{m+1}((H^r(\R))^{m-1}\times H^{r-1}(\R)\times H^\tau(\R), L_2(\R)).$$
\item[$(ii)$] $A_{m,n}\in C^{1-}((H^r(\R))^{n+1},\kL_{m+1}((H^r(\R))^{m-1}\times H^{r-1}(\R)\times H^\tau(\R), L_2(\R))$.
\end{itemize}
\end{lemma}
\begin{proof}
 The claim $(ii)$ is again a direct consequence of $(i)$, so that we are left to prove the first claim. 
 To this end we write
 \begin{align*}
 A_{m,n}(a_1, \ldots, a_{n+1})[b_1, \ldots,b_{m}, c](x)= \int_\R K(x,y)\, dy,
 \end{align*}
 where
\begin{align*}
 K(x,y):=\frac{\prod_{i=1}^{m-1}\big(\delta_{[x,y]} b_i/y\big )}{\prod_{i=1}^{n+1}\big[1+\big(\delta_{[x,y]} a_i /y\big)^2\big]}\frac{ \delta_{[x,y]}b_m}{y}\frac{ \delta_{[x,y]}c}{y} \qquad\text{for $x\in\R,$ $ y\neq0$.}
\end{align*}
Using Minkowski's integral inequality, we  compute 
\begin{align*}
 \Big(\int_\R\Big|\int_\R K(x,y)\, dy\Big|^2\, dx\Big)^{1/2}\leq &  \int_\R\Big(\int_\R |K(x,y)|^2\, dx\Big)^{1/2}\, dy,
\end{align*}
and exploiting  the fact $H^{r-1}(\R)\hookrightarrow {\rm BC}^{r-3/2}(\R)$, we get
\begin{align*}
 \int_\R |K(x,y)|^2\, dx\leq &\frac{C }{y^{7-2r}} \|b_m\|_{H^{r-1}}^2\Big(\prod_{i=1}^{m-1}\|b_i'\|_{\infty}^2\Big)\int_\R|c-\tau_yc|^2\, dx\\[1ex]
 =& \frac{C }{y^{7-2r}} \|b_m\|_{H^{r-1}}^2\Big(\prod_{i=1}^{m-1}\|b_i'\|_{\infty}^2\Big)\int_\R|\mathcal{F}c(\xi)|^2|e^{iy\xi}-1|^2\, d\xi.
\end{align*}
Since
\[
|e^{iy\xi}-1|^2\leq  C\big[(1+|\xi|^2)^{\tau}y^{2\tau} {\bf 1}_{(-1,1)}(y)+ {\bf 1}_{[|y|\geq1]}(y)\big],\qquad y,\xi\in\R,
\]
it follows that
\[
\int_\R |K(x,y)|^2\, dx\leq C\|c\|_{H^\tau}^2\|b_m\|_{H^{r-1}}^2\Big(\prod_{i=1}^{m-1}\|b_i'\|_{\infty}^2\Big)  \Big[y^{2(r+\tau)-7} {\bf 1}_{(-1,1)}(y)+ \frac{1}{y^{7-2r}} {\bf 1}_{[|y|\geq1]}(y)\Big],
\]
and we conclude that
\begin{align*}
 \int_\R\Big(\int_\R |K(x,y)|^2\, dx\Big)^{1/2}\, dy\leq C\|c\|_{H^\tau} \|b_m\|_{H^{r-1}} \prod_{i=1}^{m-1}\|b_i'\|_{\infty}.
\end{align*}
The claim $(i)$ follows at once.
\end{proof}

\paragraph{\bf Mapping properties}  
We now use the Lemmas \ref{L21} and \ref{L21'} to prove that  the mapping $\Phi $ defined by \eqref{opa} is  well-defined and locally Lipschitz continuous as an operator
from  $H^s(\R)$ into the Banach space $\kL(H^2(\R), H^1(\R))$ for each $s\in(3/2,2).$

\begin{lemma}\label{L31}
Given  $ s\in(3/2,2)$, it holds that
\begin{align*} 
 \Phi\in C^{1-}(H^s(\R),\kL(H^2(\R), H^1(\R))).
\end{align*}
\end{lemma}
\begin{proof}
We first prove that $\Phi(f)\in \kL(H^2(\R), H^1(\R))$ for each $f\in H^s(\R) $.
Remark \ref{R1} and Lemma \ref{L21} (with $r=s$)  yield that $\Phi(f)\in\kL(H^2(\R), L_2(\R)).$
In order to establish that  $\Phi(f)[h]\in H^1(\R)$, we  let   $\{\tau_\e\}_{\e\in\R}$ denote the    $C_0$-group of right translations on $L_2(\R)$, that is $\tau_\e f(x):=f(x-\e)$ for $f\in L_2(\R)$ and $x,$ $\e\in\R$.
Given $\e\in(0,1)$, it holds that
\begin{align*}
 \frac{\tau_\e(\Phi(f)[h])-\Phi(f)[h]}{\e}=& \frac{\tau_\e(A_{0,0}[f][h'])-A_{0,0}(f)[h']}{\e}=  \frac{ A_{0,0}(\tau_\e f)[\tau_\e h'] -A_{0,0}(f)[h']}{\e}\\[1ex]
 =&   A_{0,0}(\tau_\e f)\Big[\frac{\tau_\e h'-h'}{\e}\Big]-A_{2,1}(\tau_\e f,f)\Big[\tau_\e f+f,\frac{\tau_\e f-f}{\e}, h'\Big]
\end{align*}
and the convergences
\[
  \tau_\e f  \underset{\e\to0}{\longrightarrow}  f  \quad\text{in $H^{s }(\R),$}\qquad \frac{ \tau_\e f - f }{\e}\underset{\e\to0}{\longrightarrow} -f' \quad\text{in $H^{s-1}(\R),$ }\qquad \frac{ \tau_\e h - h }{\e}\underset{\e\to0}{\longrightarrow} -h'  \quad\text{in $H^1(\R),$ }
\]
together with the  Lemma \ref{L21} (with $r=s$) and Lemma \ref{L21'} (with $r=s$, $\tau\in(5/2-s,1)$) imply that $\Phi(f)[h]\in H^1(\R)$ and
\begin{align} 
 (\Phi(f)[h])'=A_{0,0}(f)[h'']-2A_{2,1}(f,f)[f,f',h'].\label{N1}
\end{align}
This proves that $\Phi(f)\in \kL(H^2(\R), H^1(\R))$.
Finally, the local Lipschitz continuity of $\Phi$  follows  from the local Lipschitz continuity properties established in the Lemmas \ref{L21} and \ref{L21'}.
\end{proof}

\section{The Muskat problem without surface tension: the generator property}\label{S3}

We now fix $f\in H^s(\R), s\in(3/2,2).$
The goal of this section is to prove that $\Phi(f)$, regarded as an unbounded operator in $H^1(\R)$ with definition domain $H^2(\R)$, is 
the generator of a strongly continuous and analytic semigroup in $\kL(H^1(\R)),$
that is
\[
-\Phi(f)\in \kH(H^2(\R), H^1(\R)). 
\]
In order to establish this property we first approximate locally the operator $\Phi(f)$, in a sense to be made precise in Theorem \ref{T1}, by Fourier multipliers and carry then the desired generator property, 
which we establish for the Fourier multipliers, 
back to the original operator, cf. Theorem \ref{T2}. 
A similar approach was followed in the papers \cite{E94, EMW18, ES95,  ES97} in the context of spaces of continuous functions.
The situation here is different as we consider Sobolev spaces on the line. The method though can be adapted to this setting after exploiting the structure of the  operator $\Phi(f)$, 
especially the fact that the functions $f$ and $f'$  vanish both  at infinity.
As a result  of this decay  property we can use localization families with a finite number of elements, and this fact enables us to introduce for each localization family an equivalent 
norm on the Sobolev spaces $H^k(\R), k\in\N,$ which is suitable for the further analysis, cf. Lemma \ref{L:EN2}.   
We start by choosing  for each $\e\in(0,1)$, a finite $\e$-localization family, that is  a family  
\[\{\pi_j^\e\,:\, -N+1\leq j\leq N\}\subset  C^\infty_0(\R,[0,1]),\]
with $N=N(\e)\in\N $ sufficiently large, such that
\begin{align}
\bullet\,\,\,\, \,\,  & \text{$ \supp \pi_j^\e $ is an interval of length less or equal $\e$ for all $|j|\leq N-1$;}
\label{i}\\[1ex]
\bullet\,\,\,\, \,\, &\text{$ \supp \pi_{N}^\e\subset(-\infty,-x_{N}]\cup [x_N,\infty)$  and $x_{N}\geq\e^{-1}$;}
\label{ii}\\[1ex]
\bullet\,\,\,\, \,\, &\text{ $\supp \pi_j^\e\cap\supp \pi_l^\e=\emptyset$ if $[|j-l|\geq2, \max\{|j|, |l|\}\leq N-1]$ or $[|l|\leq N-2, j=N];$} \label{iii}\\[1ex]
\bullet\,\,\,\, \,\, &\text{ $\sum_{j=-N+1}^N(\pi_j^\e)^2=1;$}\label{iv}\\[1ex]
 \bullet\,\,\,\, \,\, &\text{$\|(\pi_j^\e)^{(k)}\|_\infty\leq C\e^{-k}$ for all $ k\in\N, -N+1\leq j\leq N$.}\label{v} 
\end{align} 
Such $\e$-localization families can be easily constructed.
Additionally, we choose for each $\e\in(0,1)$  a second family   \[\{\chi_j^\e\,:\, -N+1\leq j\leq N\}\subset  C^\infty_0(\R,[0,1])  \] with the following properties
\begin{align}
\bullet\,\,\,\, \,\,  &\text{$\chi_j^\e=1$ on $\supp \pi_j^\e$;}\label{c1}\\[1ex]
\bullet\,\,\,\, \,\,  &\text{$ \supp \chi_j^\e$ is an interval  of length less or equal $ 3\e$ for $|j|\leq N-1$;}\label{c2}\\[1ex]
\bullet\,\,\,\, \,\,  &\text{$\supp\chi_N^\e\subset [|x|\geq x_N-\e]$.} \label{c3}
\end{align} 

Each $\e$-localization family  $\{\pi_j^\e\,:\, -N+1\leq j\leq N\}$     defines   a  norm on $H^k(\R), k\in\N,$    which is equivalent to the standard $H^k$-norm.
\begin{lemma}\label{L:EN2}
Given  $\e\in(0,1)$, let  $\{\pi_j^\e\,:\, -N+1\leq j\leq N\}\subset  C^\infty_0(\R,[0,1])$ be a family with the properties \eqref{i}-\eqref{v}. 
Then, for each $k\in\N$, the mapping  
\[
\Big[h\mapsto \sum_{j=-N+1}^ N\|\pi_j^\e h\|_{H^k}\Big]:H^k(\R)\to[0,\infty)
\]
defines a norm on $H^k(\R) $ which is equivalent to the standard $H^k$-norm. 
\end{lemma}
\begin{proof} The proof is a simple exercise.
\end{proof}\medskip

We now consider the mapping
\[[\tau\mapsto \Phi(\tau f)]:[0,1]\to \kL(H^2(\R), H^1(\R)).\]
As a consequence of Lemma \ref{L31}, this mapping  continuously transforms the operator $\Phi(f)$, for which we want to establish the generator property, into the operator $\Phi(0)= -\pi (-\p_x^2)^{1/2}$.
Indeed, since  the Hilbert transform is a Fourier multiplier with symbol $[\xi\mapsto -i\sign(\xi)],$ we obtain together with  Remark \ref{R1} that
\[
\kF(\Phi(0)[h])(\xi)=-\pi\kF(Hh')(\xi)=i\pi\sign(\xi)\kF( h')(\xi)=-\pi|\xi|(\kF h)(\xi)=-\pi\kF((-\p_x^2)^{1/2}h)(\xi)
\]
for $\xi\in\R$.  
The parameter $\tau$ will allow us to use a continuity argument when showing that the resolvent set of $\Phi(f)$ contains a positive real number, see the proof Theorem \ref{T2}. 
\medskip

Our next goal is to prove that   the operator $\Phi(\tau f)$ can be locally approximated for each $\tau\in[0,1]$ by Fourier multipliers, as stated below.
The estimate \eqref{DE3} with $j=N$ uses to a large extent the fact that $f$ and $f'$ vanish at infinity.

\begin{thm}\label{T1} 
Let  $f\in H^s(\R)$, $s\in(3/2,2),$  and     $\mu>0$ be given.

Then, there exist $\e\in(0,1)$, a finite $\e$-localization family  $\{\pi_j^\e\,:\, -N+1\leq j\leq N\} $ satisfying \eqref{i}-\eqref{v}, a constant $K=K(\e)$, and for each  $ j\in\{-N+1,\ldots,N\}$ and $\tau\in[0,1]$ there 
exist   operators $$\bA_{ j,\tau}\in\kL(H^2(\R), H^{1}(\R))$$
 such that 
 \begin{equation}\label{DE3}
  \|\pi_j^\e\Phi(\tau f)[h]-\bA_{j,\tau}[\pi^\e_j h]\|_{H^1}\leq \mu \|\pi_j^\e h\|_{H^2}+K\|  h\|_{H^{(11-2s)/4}}
 \end{equation}
 for all $ j\in\{-N+1,\ldots, N\}$, $\tau\in[0,1],$ and  $h\in H^2(\R)$. The operators $\bA_{j,\tau}$ are defined  by 
  \begin{align} 
 \bA_{j,\tau }:=&\Big[\PV\int_\R \frac{1}{y}\frac{1}{1+\tau^2\big(\delta_{[x_j^\e,y]}f/y\big)^2}\, dy\Big]\p_x-\frac{\pi}{1+(\tau f'(x_j^\e))^2}(-\p_x^2)^{1/2}, \qquad |j|\leq N-1,\label{FM1}
 \end{align}
 where $x_j^\e$ is a point belonging to  $\supp  \pi_j^\e$, respectively
 \begin{align} 
 \bA_{N,\tau }:=& - \pi (-\p_x^2)^{1/2}.\label{FM2}
 \end{align}
\end{thm}
\begin{proof} 
Let   $\{\pi_j^\e\,:\, -N+1\leq j\leq N\} $ be a $\e$-localization family satisfying  the properties \eqref{i}-\eqref{v} and $\{\chi_j^\e\,:\, -N+1\leq j\leq N\} $ be an associated family satisfying \eqref{c1}-\eqref{c3}, with $\e\in(0,1)$
which will be fixed below.
We first infer from Lemma \ref{LA1}  that for each $ \tau\in[0,1]$
the function 
\[a_{\tau}(x) :=\PV\int_\R \frac{1}{y}\frac{1}{1+\tau^2\big(\delta_{[x,y]}f/y \big)^2}\, dy,\qquad x\in\R,\]
belongs to ${\rm BC}^\alpha(\R)\cap C_0(\R)$, with $\alpha:=s/2-3/4$.
We now  write
\[\bA_{j,\tau}:=\bA^1_{j,\tau}-\bA^2_{j,\tau}\]
where
\[
\bA^1_{j,\tau}:=\alpha_{j,\tau} \p_x,\qquad  \bA^2_{j,\tau}:=\beta_{j,\tau}(-\p_x^2)^{1/2},
\]
and
\begin{align}\label{CO}
\alpha_{j,\tau}
:=\left\{\begin{array}{clll}
  a_\tau(x_j^\e)&,&  |j|\leq N-1,\\[1ex]
  0&,& j=N,
  \end{array}\right.
\qquad    \beta_{j,\tau}:=
\left\{
\begin{array}{clll}
  \cfrac{\pi}{1+(\tau f'(x_j^\e))^2}&,&  |j|\leq N-1,\\[1ex]
  \pi&,& j=N.
  \end{array}
  \right.
\end{align}
Let  now $h\in H^2(\R)$ be arbitrary.
 In the following we shall denote  by $C$ constants which are
independent of $\e$ (and, of course, of $h\in H^2(\R)$, $\tau\in [0,1]$, and $j \in \{-N+1, \ldots, N\}$), while the constants that we denote by $K$ may depend only upon $\e.$ \medskip

\noindent{\em Step 1.\,\,}  We first infer from Lemma \ref{L31}  that  
\begin{align*}
 \|\pi_j^\e\Phi(\tau f)[h]-\bA_{j,\tau}[\pi^\e_j h]\|_{H^1}\leq& \|\pi_j^\e\Phi(\tau f)[h]-\bA_{j,\tau}[\pi^\e_j h]\|_{2}+\|(\pi_j^\e\Phi(\tau f)[h]-\bA_{j,\tau}[\pi^\e_j h])'\|_{2}\\[1ex]
 \leq&(1+\|(\pi_j^\e)'\|_\infty)\|A_{0,0} (\tau f)[h']\|_2+\|\bA_{j,\tau}[\pi^\e_j h]\|_{2}\\[1ex]
 &+2\|A_{2,1}(\tau f,\tau f)[f,f',h']\|_2+\|\pi_j^\e A_{0,0}(\tau f)[h'']-\bA_{j,\tau}[(\pi^\e_j h)']\|_{2}.
\end{align*}
Using Lemma \ref{L21'} (with $r=s$ and $\tau=7/4-s/2$) and Lemma \ref{L21} (with $r=s$), it follows   
\begin{align} \label{m4}
 \|\pi_j^\e\Phi(\tau f)[h]-\bA_{j,\tau}[\pi^\e_j h]\|_{H^1}\leq& K\|h\|_{H^{(11-2s)/4}} +\|\pi_j^\e A_{0,0}(\tau f)[h'']-\bA_{j,\tau}[(\pi^\e_j h)']\|_{2}.
\end{align}
We are left to estimate the $L_2$-norm of the highest order term $\pi_j^\e A_{0,0}(\tau f)[h'']-\bA_{j,\tau}[(\pi^\e_j h)']$, and for this we need several steps.
 \medskip

\noindent{\em Step 2.\,\,} With  the notation introduced in Remark \ref{R2} we have
\begin{align*}
  A_{0,0}(\tau f)[h'']=&a_\tau h''-B_{0,1}(\tau f)[h''],
\end{align*}
and therewith
 \begin{align} 
 \|\pi_j^\e A_{0,0}(\tau f)[h'']-\bA_{j,\tau}[(\pi^\e_j h)'] \|_{2}\leq \|\pi_j^\e a_\tau   h''-\bA^1_{j,\tau}[(\pi^\e_j h)']\|_2+\|\pi_j^\e B_{0,1}(\tau f)[h'']-\bA^2_{j,\tau}[(\pi^\e_j h)']\|_2.\label{m5}
\end{align}
In virtue of Lemma \ref{LA1} (in particular of the estimate \eqref{UET1}) and of $\chi_{j}^\e=1$ on $\supp\pi_j^\e$, we get for $|j|\leq N-1$ 
\begin{align}
 \| \pi_j^\e a_\tau   h''-\bA^1_{j,\tau}[(\pi^\e_j h)'\|_{2}=&\|a_\tau\pi_j^\e  h''-a_\tau(x_j^\e)(\pi_j^\e h)''\|_2\leq \|(a_\tau -a_\tau(x_j^\e))(\pi_j^\e h)''\|_2+K\|h\|_{H^1}\nonumber\\[1ex]
  =&\|(a_\tau  -a_\tau(x_j^\e))\chi_j^\e(\pi_j^\e h)''\|_2+K\|h\|_{H^1}\nonumber\\[1ex]
   =&\|(a_\tau -a_\tau (x_j^\e))\chi_j^\e\|_{\infty}\|(\pi_j^\e h)''\|_2+K\|h\|_{H^1}\nonumber\\[1ex]
   \leq &\frac{\mu}{2}\|\pi_j^\e h\|_{H^2}+K\|h\|_{H^1} , \label{m6}
\end{align}
provided that  $\e$   is sufficiently small.
We have  used here (and also later on without explicit mentioning) the  fact that $|\supp \chi_j^\e|\leq 3\e.$ 
Since $\bA^1_{N,\tau}=0$, we obtain  from \eqref{UET2} for $\e$ sufficiently small   that 
\begin{align}
 \| \pi_N^\e a_\tau   h''-\bA^1_{N,\tau}[(\pi^\e_N h)'\|_{2}=&\| \pi_N^\e a_\tau   h''\|_{2}\leq\| a_\tau \chi_N^\e\|_{\infty}\|(\pi_N^\e h)''\|_2+K\|h\|_{H^1}\nonumber\\[1ex]
   \leq &\frac{\mu}{2}\|\pi_N^\e h\|_{H^2}+K\|h\|_{H^1}.\label{m6'}
\end{align}

\noindent{\em Step 3.\,\,}
 We are left with the term  $\|\pi_j^\e B_{0,1}(\tau f)[h'']-\bA^2_{j,\tau}[(\pi^\e_j h)']\|_2$, and we consider first the case $|j|\leq N-1 $ (see   {\em Step 4} for $j=N$).
Observing  that $\pi(-\p_x^2)^{1/2}=B_{0,1}(0)\circ\p_x,$ it follows that 
 \begin{align*}
  \pi_j^\e B_{0,1}(\tau f)[h'']-\bA^2_{j,\tau}[(\pi^\e_j h)']=T_1[h]-T_2[h],
 \end{align*}
where
\begin{align*}
T_1[h]&:=\pi_j^\e B_{0,1}(\tau f)[h'']-\frac{1}{1+(\tau f'(x_j^\e))^2}B_{0,1}(0)[\pi_j^\e h''],\\[1ex]
 T_2[h]&:=\frac{1}{1+(\tau f'(x_j^\e))^2} B_{0,1}(0)[(\pi_j^\e)''h+2(\pi_j^\e)'h'].
\end{align*}
Since by Remark \ref{R2}
\begin{align}
 \| T_2[h]\|_{2} \leq K\|h\|_{H^1}, \label{m6''}
\end{align}
we are left to estimate $T_1[h],$ which is further decomposed 
\[
T_1[h]=T_{11}[h]-T_{12}[h],
\]
with 
\begin{align*}
T_{11}[h](x):=&\PV\int_\R\Big[\frac{1}{1+\tau^2\big( \delta_{[x,y]} f/y\big)^2}-\frac{1}{1+ \big( \tau f'(x_j^\e)\big)^2}\Big]\frac{(\chi_j^\e\pi_j^\e h'')(x-y)}{y}\,dy,\\[1ex]
T_{12}[h](x):=& \PV\int_\R\frac{ \delta_{[x,y]} \pi_j^\e/y}{1+\tau^2\big( \delta_{[x,y]} f/y\big)^2}   h'' (x-y) \,dy.
\end{align*}
Integrating by parts, we  obtain the following relation
 \begin{align*}
T_{12}[h]=&B_{0,1}(\tau f)[(\pi_j^\e)'h']-B_{1,1}(\pi_j^\e,\tau f)[h']-2\tau^2B_{2,2}(\pi_j^\e,f,\tau f,\tau f)[f'h']\\[1ex]
&+2\tau^2B_{3,2}(\pi_j^\e,f,f,\tau f,\tau f)[h'],
\end{align*}
and Remark \ref{R2} leads us to
\begin{align}
 \| T_{12}[h]\|_{2} \leq K\|h\|_{H^1}. \label{m6'''}
\end{align}
In order to deal with the term $T_{11}[h]$  we let $F_{j}\in C(\R)$ denote the Lipschitz function that satisfies 
\begin{equation}\label{Func}
 F_{j}=f \quad   \text{on $\supp\chi_j^\e$}, \qquad F_{j}'=f'(x_j^\e) \quad   \text{on $\R\setminus\supp\chi_j^\e$,} 
\end{equation}
and we observe that
\begin{align*}
 T_{11}[h](x):=&\tau^2
 \PV\int_\R\frac{\big[\delta_{[x,y]}(f'(x_j^\e){\rm id}_{\R}-f)/y\big]\big[\delta_{[x,y]}(f'(x_j^\e){\rm id}_{\R}+f)/y\big]}{\big[1+\tau^2\big( \delta_{[x,y]} f/y\big)^2\big]\big[1+ \big( \tau f'(x_j^\e)\big)^2\big]}\frac{(\chi_j^\e\pi_j^\e h'')(x-y)}{y}\,dy\\[1ex]
=& \frac{\tau^2}{1+ \big( \tau f'(x_j^\e)\big)^2}\big(T_{111}[h]-T_{112}[h]\big)(x),
\end{align*}
where 
\begin{align*}
 &T_{111}[h]:=\chi_j^\e B_{2,1}(f'(x_j^\e){\rm id}_{\R}-f,f'(x_j^\e){\rm id}_{\R}+f ,\tau f)[\pi_j^\e h''],\\[1ex]
 &T_{112}[h](x):= \PV\int_\R\frac{\big[\delta_{[x,y]}(f'(x_j^\e){\rm id}_{\R}-f)/y\big]\big[\delta_{[x,y]}(f'(x_j^\e){\rm id}_{\R}+f)/y\big]\big(\delta_{[x,y]}\chi_j^\e/y\big)}{ 1+\tau^2\big( \delta_{[x,y]} f/y\big)^2 } (\pi_j^\e h'')(x-y) \,dy. 
\end{align*}
 Integrating by parts as in the case of $T_{12}[h]$, it follows from Remark \ref{R2} that
\begin{align}
 \| T_{112}[h]\|_{2} \leq K\|h\|_{H^1}. \label{m6''''}
\end{align}
On the other hand, \eqref{Func}, Remark \ref{R2} and the H\"older continuity of $f'$ yield
\begin{align}
 \|T_{111}[h]\|_2=&\|\chi_j^\e B_{2,1}(f'(x_j^\e){\rm id}_{\R}-f,f'(x_j^\e){\rm id}_{\R}+f ,\tau f)[\pi_j^\e h'']\|_2\nonumber\\[1ex]
 =&\|\chi_j^\e B_{2,1}(f'(x_j^\e){\rm id}_{\R}-F_j,f'(x_j^\e){\rm id}_{\R}+F_j ,\tau f)[\pi_j^\e h'']\|_2\nonumber\\[1ex]
 \leq & C\|f'(x_j^\e)-F_j'\|_\infty\|\pi_j^\e h''\|_2\nonumber\\[1ex]
 =&  C\|f'(x_j^\e)-f'\|_{L_\infty(\supp \chi_j^\e)}\|\pi_j^\e h''\|_2\nonumber\\[1ex]
 \leq &\frac{\mu}{2}\|\pi_j^\e h\|_{H^2}+K\|h\|_{H^1}.\label{m67}
\end{align}
The desired estimate \eqref{DE3} follows for $|j|\leq N-1$ from \eqref{m4}-\eqref{m6} and \eqref{m6''}, \eqref{m6'''}, \eqref{m6''''}, and \eqref{m67}.\medskip

\noindent{\em Step 4.\,\,}
 We are left  with the term $\|\pi_N^\e B_{0,1}(\tau f)[h'']-\bA^2_{N,\tau}[(\pi^\e_N h)']\|_2$, which we decompose as follows
\begin{align*}
\big(\pi_N^\e B_{0,1}(\tau f)[h'']-\bA^2_{N,\tau}[(\pi^\e_N h)']\big)(x) =&\pi_N^\e(x)\PV\int_\R\frac{ h''(x-y)}{y}\frac{1}{1+\tau^2 \big(\delta_{[x,y]}f/y\big)^2}\, dy\\[1ex]
&- \PV\int_\R\frac{(\pi_N^\e h)''(x-y)}{y}\, dy \\[1ex]
=:&  T_1[h](x)+ T_2[h](x)- T_3[h](x), 
\end{align*}
where
\begin{align*}
&T_1[h] :=-\tau^2B_{2,1}(f,f,\tau f)[\pi_N^\e h''],\qquad  T_3[h]:= B_{0,1}(0)[(\pi_N^\e)'' h+2(\pi_N^\e)' h'],\\[1ex]
 &T_2[h](x):=\PV\int_\R h''(x-y)\frac{ \delta_{[x,y]}\pi_N^\e}{y}\frac{1}{1+\tau^2 \big(\delta_{[x,y]}f/y\big)^2}\, dy .\\[1ex]
\end{align*}
For the difference $ T_2[h]-T_3[h]$ we find, as in the previous step (see \eqref{m6''} and \eqref{m6'''})
that 
\begin{align}
 \| T_2[h]-T_3[h]\|_{2} \leq K\|h\|_{H^1}. \label{m69a}
\end{align}
When dealing with $  T_1[h],$ we  introduce the function $F_N\in W^1_\infty(\R)$ by the formula
\[
F_N(x):=\left\{
\begin{array}{lll}
f(x)&,& |x|\geq x_N-\e, \\[1ex]
\cfrac{x+x_N-\e}{2(x_N-\e)}f(x_N-\e)+\cfrac{x_N-\e-x}{2(x_N-\e)}f(-x_N+\e)&,& |x|\leq x_N-\e.
\end{array}
\right.
\]
The relation  \eqref{ii} implies $\|F_N\|_\infty+\|F_N'\|_\infty\to0 $ for $\e\to0$.
Moreover,  it holds that 
 \begin{align*}
  T_1[h](x)=&-\tau^2\PV\int_\R\frac{(\chi_N^\e\pi_N^\e h'')(x-y)}{y}  \frac{ \big(\delta_{[x,y]}f/y\big)^2}{1+\tau^2 \big(\delta_{[x,y]}f/y\big)^2}=:T_{11}[h](x)-T_{12}[h](x), 
 \end{align*} 
 where
 \begin{align*}
&T_{11}[h](x):=\tau^2\PV\int_\R( \pi_N^\e h'')(x-y)  \frac{ \big(\delta_{[x,y]}f/y\big)^2\big(\delta_{[x,y]}\chi_N^\e/y\big)}{1+\tau^2 \big(\delta_{[x,y]}f/y\big)^2}\, dy,  \\[1ex]
 &T_{12}[h] :=\tau^2\chi_N^\e B_{2,1}(f,f,\tau f)[\pi_N^\e h'']. 
 \end{align*}
Recalling  that $\supp \pi_N^\e\subset \supp \chi_N^\e \subset[|x|\geq x_N-\e]$ and that $f=F_N$ on $\supp \chi_N^\e ,$ it follows with
Remark \ref{R2}  that 
\begin{align}\label{m69b}
\|T_{12}[h]\|_2=&\|\tau^2\chi_N^\e B_{2,1}(F_N,F_N,\tau f)[\pi_N^\e h'']\|_2 \leq\| B_{2,1}(F_N,F_N,\tau f)[\pi_N^\e h'']\|_2\leq C\|F_N'\|_\infty^2\|\pi_N^\e h''\|_2\nonumber\\[1ex]
\leq& \frac{\mu}{2}\|\pi_j^\e h\|_{H^2}+K\| h\|_{H^1} 
\end{align}
for small $\e$.
As $T_{11}[h]$ can be estimated in the same manner as the term $T_{112}[h]$ in the previous step, we obtain together with  \eqref{m69a} and \eqref{m69b} that
\begin{align}
\|\pi_N^\e B_{0,1}(\tau f)[h'']-\bA^2_{N,\tau}[(\pi^\e_N h)']\|_2\leq \frac{\mu}{2}\|\pi_j^\e h\|_{H^2}+K\|h\|_{H^1}\label{m11}
\end{align}
if $\e$ is sufficiently small. 
 The   claim \eqref{DE3} follows for $j=N$ from \eqref{m4}-\eqref{m5}, \eqref{m6'},  and  \eqref{m11}.
\end{proof}\medskip

The operators $\bA_{\tau,j}$  found in Theorem \ref{T1} are generators of strongly continuous analytic semigroups in $\kL(H^1(\R))$ and they satisfy resolvent estimates which are uniform with respect to $x_j^\e\in\R $ and  $\tau\in [0, 1] $,
see Proposition \ref{PP1} below.
To be more precise, in Proposition \ref{PP1} and in the proof of  Theorem \ref{T2}, the Sobolev spaces $H^k(\R), k\in\{1,2\},$ consist  of complex valued functions
and $\bA_{j,\tau}$ are the natural  extensions (complexifications) of the  operators
introduced in Theorem \ref{T1}.

\begin{prop}\label{PP1} 
Let  $f\in H^s(\R)$, $s\in(3/2,2),$ be fixed. Given $x_0\in \R$ and $\tau\in[0,1]$, let
\[
\bA_{x_0,\tau}:=\alpha_{\tau} \p_x-\beta_{\tau}(-\p_x^2)^{1/2},
\]
where
\begin{align*} 
\alpha_{\tau}\in\{0,a_\tau(x_0)\}
\qquad\text{and}\qquad  \beta_{\tau}\in\Big\{\pi,\cfrac{\pi}{1+(\tau f'(x_0))^2}\Big\},
\end{align*}
with $a_\tau$ denoting the function defined in Lemma \ref{LA1}. 
  Then, there exists a constant  $\kappa_0\geq1$  such that  
  \begin{align}\label{13}
&\lambda-\bA_{x_0,\tau}\in{\rm Isom}(H^2(\R), H^1(\R)),\\[1ex]
\label{14}
& \kappa_0\|(\lambda-\bA_{x_0,\tau})[h]\|_{H^1}\geq  |\lambda|\cdot\|h\|_{H^1}+\|h\|_{H^2}
\end{align}
for all $x_0\in\R$, $\tau\in[0,1],$    $\lambda\in\C$ with $\re \lambda\geq 1$, and $h\in H^2(\R)$.
\end{prop}
\begin{proof}  The  constants $\alpha_{\tau},\beta_{\tau} $ defined above satisfy, in view of \eqref{SE1},
\begin{align}\label{Bounds}
 |\alpha_{\tau}|\leq 4\Big(\|f\|_\infty^2+\frac{2\|f'\|_\infty[f']_{s-3/2}}{s-3/2}\Big)\qquad\text{and}\qquad \beta_{\tau}\in\Big[\frac{\pi}{1+\max |f'|^2},\pi\Big].
\end{align}
Furthermore, the operator  $\bA_{x_0,\tau}$ is a Fourier multiplier with symbol 
\[m_\tau(\xi):=-\beta_{\tau}|\xi|+i\alpha_{\tau} \xi, \qquad \xi\in\R.\]
 Given $\re\lambda\geq1,$ it is easy to see that the operator $R(\lambda,\bA_{x_0,\tau})$  defined by
 \[
\kF (R(\lambda,\bA_{x_0,\tau})[h])=\frac{1}{\lambda-m_\tau}\kF h,\qquad h\in H^1(\R),
 \]
 belongs to $\kL(H^1(\R), H^2(\R))$  and that it is the inverse to $\lambda-\bA_{x_0,\tau}$.
 Moreover,  for each $\re\lambda\geq1$ and $h\in H^2(\R)$, we have 
 \begin{align}
  \|(\lambda-\bA_{x_0,\tau})[h]\|_{H^1}^2=&\int_\R(1+|\xi|^2) |\kF((\lambda-\bA_{x_0,\tau})[h])|^2(\xi)\, d \xi=\int_\R(1+|\xi|^2) |\lambda-m_\tau(\xi)|^2|\kF h|^2(\xi)\, d \xi\nonumber\\[1ex]
  \geq&\min\{1,\beta_{\tau}^2\}\int_\R(1+|\xi|^2)^2|\kF h|^2(\xi)\, d \xi=\min\{1,\beta_{\tau}^2\}\|h\|_{H^2}^2.\label{*}
 \end{align}
Appealing to the inequality 
 \begin{align*}
\frac{|\lambda|^2}{|\lambda-m_\tau(\xi)|^2}=&  \frac{(\re\lambda)^2 }{(\re\lambda+\beta_{\tau}|\xi|)^2+(\im\lambda-\alpha_{\tau}\xi)^2} +\frac{(\im\lambda)^2 }{(\re\lambda+\beta_{\tau}|\xi|)^2+(\im\lambda-\alpha_{\tau}\xi)^2}\nonumber\\[1ex]
\leq & 1+\frac{2(\im\lambda-\alpha_{\tau}\xi)^2+2\alpha_{\tau}^2\xi^2}{(\re\lambda+\beta_{\tau}|\xi|)^2+(\im\lambda-\alpha_{\tau}\xi)^2}
\leq 1+ 2\Big[1+\Big(\frac{\alpha_{\tau}}{\beta_{\tau}}\Big)^2\Big]\leq 3\Big[1+\Big(\frac{\alpha_{\tau}}{\beta_{\tau}}\Big)^2\Big] 
 \end{align*}
for $\lambda\in\C$ with $\re\lambda\geq1,$ the   estimate \eqref{14} follows from the relations \eqref{Bounds} and \eqref{*}.
\end{proof}\medskip

We now establish the desired generation result.

\begin{thm}\label{T2}
 Let $f\in H^s(\R)$, $s\in(3/2,2),$ be given. Then 
 \begin{align}\label{AG}
  -\Phi(f)\in \kH(H^2(\R), H^1(\R)).
 \end{align}
\end{thm}
\begin{proof}
 Let $\kappa_0\geq1$ be the constant determined in Proposition \ref{PP1}.
 Setting $\mu:=1/2\kappa_0$, we deduce  from  Theorem \ref{T1} that there exists a constant $\e\in(0,1),$  a $\e$-localization family $\{\pi_j^\e\,:\, -N+1\leq j\leq N\}$ that satisfies \eqref{i}-\eqref{v}, a constant $K=K(\e)$, 
 and for each  $ -N+1\leq j\leq N$ and $\tau\in[0,1]$
  operators $\bA_{j,\tau}\in\kL(H^2(\R), H^{1}(\R))$
 such that 
 \begin{equation}\label{DE4}
  \|\pi_j^\e\Phi(\tau f)[h]-\bA_{j,\tau}[\pi^\e_j h]\|_{H^1}\leq \frac{1}{2\kappa_0}\|\pi_j^\e h\|_{H^2}+K\|  h\|_{H^{(11-2s)/4}}
 \end{equation}
 for all $-N+1\leq j\leq N$, $\tau\in[0,1],$ and  $h\in H^2(\R)$.
In view of Proposition \ref{PP1}, it holds that
  \begin{equation}\label{DE5}
    \kappa_0\|(\lambda-\bA_{j,\tau})[\pi^\e_jh]\|_{H^1}\geq |\lambda|\cdot\|\pi^\e_jh\|_{H^1}+ \|\pi^\e_j h\|_{H^2}
 \end{equation}
 for all $-N+1\leq j\leq N$, $\tau\in[0,1],$ $\lambda\in\C$ with $\re \lambda\geq 1$, and  $h\in H^2(\R)$.
 The relations \eqref{DE4}-\eqref{DE5} lead us to
 \begin{align*}
   \kappa_0\|\pi_j^\e(\lambda-\Phi(\tau f))[h]\|_{H^1}\geq& \kappa_0\|(\lambda-\bA_{j,\tau})[\pi^\e_jh]\|_{H^1}-\kappa_0\|\pi_j^\e\Phi(\tau f)[h]-\bA_{j,\tau}[\pi^\e_j h]\|_{H^1}\\[1ex]
   \geq& |\lambda|\cdot\|\pi^\e_jh\|_{H^1}+ \frac{1}{2}\|\pi^\e_j h\|_{H^2}-\kappa_0K\|  h\|_{H^{(11-2s)/4}}
 \end{align*}
for all $-N+1\leq j\leq N$, $\tau\in[0,1],$ $\lambda\in\C$ with $\re \lambda\geq 1$, and  $h\in H^2(\R)$.
Summing up  over $j\in\{-N+1,\ldots, N\},$  we infer from Lemma \ref{L:EN2} that there exists a constant  $C\geq1$   with the property that
  \begin{align*}
   C\|  h\|_{H^{(11-2s)/4}}+C\|(\lambda-\Phi(\tau f))[h]\|_{H^1}\geq |\lambda|\cdot\|h\|_{H^1}+ \| h\|_{H^2}
 \end{align*}
for all   $\tau\in[0,1],$ $\lambda\in\C$ with $\re \lambda\geq 1$, and  $h\in H^2(\R)$.
Using \eqref{IP} together with  Young's inequality, we may find constants  $\kappa\geq 1$ and $\omega>0$ such  that 
 \begin{align}\label{DED}
   \kappa\|(\lambda-\Phi(\tau f))[h]\|_{H^1}\geq |\lambda|\cdot\|h\|_{H^1}+ \| h\|_{H^2}
 \end{align}
for   all   $\tau\in[0,1],$ $\lambda\in\C$ with $\re \lambda\geq \omega$, and  $h\in H^2(\R)$.
Furthermore, combining the property 
$$(\omega-\Phi(\tau f))\big|_{\tau=0}=\omega-\Phi(0)=\omega+\pi(-\p_x^2)^{1/2}\in {\rm Isom}(H^2(\R), H^1(\R)),$$
with \eqref{DED}, the method of continuity, cf. e.g.   \cite[Theorem 5.2]{GT01},
 yields that 
\begin{align}\label{DED2}
   \omega-\Phi(f)\in {\rm Isom}(H^2(\R), H^1(\R)).
 \end{align}
 The relations \eqref{DED} (with $\tau=1$),  \eqref{DED2}, and \cite[Corollary 2.1.3]{L95} lead us to the desired claim \eqref{AG}.
\end{proof}\medskip

We are now in a position to prove the well-posedness result Theorem \ref{MT1}.
\begin{proof}[Proof of Theorem \ref{MT1}]
 Let $s\in(3/2,2)$ and  $\ov s\in (3/2,s)$ be given.
Combining Lemma \ref{L31} and Theorem \ref{T2}, yields
\[
-\Phi\in C^{1-}(H^{\ov s}(\R),\kH(H^2(\R), H^1(\R))).
\]
   Setting $\alpha:=s-1$ and $\beta:=\ov s-1,$ we have  $0<\beta<\alpha<1$ and \eqref{IP} yields
 \[
  H^{\ov s}(\R)=[H^1(\R), H^2(\R)]_{\beta} \qquad\text{and}\qquad H^s(\R)=[H^1(\R), H^2(\R)]_{\alpha}.
 \]
It  follows now from Theorem \ref{T:A}
that \eqref{P} (or equivalently \eqref{P1}) possesses a maximally defined solution
 \begin{equation*} 
 f:=f(\cdot; f_0)\in C([0,T_+(f_0)),H^s(\R))\cap C((0,T_+(f_0)), H^2(\R))\cap C^1((0,T_+(f_0)), H^1(\R))
 \end{equation*}
with \begin{equation*} 
 f\in C^{s-\ov s}([0,T],H^{\ov s}(\R)) \qquad\text{for all $T<T_+(f_0)$}.
 \end{equation*}

 Concerning  uniqueness, we now show that any classical solution
 \begin{equation*} 
 \wt f\in C([0,\wt  T),H^s(\R))\cap C((0,\wt T), H^2(\R))\cap C^1((0,\wt T)), H^1(\R)),\qquad \wt T\in(0,\infty],
 \end{equation*}  
 satisfies 
  \begin{equation} \label{T:EEE}
 \wt f\in C^\eta([0,T],H^{\ov s}(\R))\qquad \text{for all $T\in(0,\wt T)$,}
 \end{equation}
where $\eta:=(s-\ov s)/s\in(0,s-\ov s).$ This proves then the uniqueness claim of Theorem \ref{MT1}. 
 We pick thus $T\in(0,\wt T)$ arbitrarily. 
 Then it follows directly from Lemma \ref{L21} $(i)$, that 
 \[
 \sup_{(0,T]} \|\p_t \wt f\|_{2}\leq C,
 \]
 hence $ \wt f\in{\rm BC}^1((0,T], L_2(\R))$. 
 Since $\wt f\in C([0,T], H^s(\R)),$ we conclude form \eqref{IP}, the previous bound, and the mean value theorem, that 
 \begin{align*}
  \|\wt f(t)-\wt f(s)\|_{H^{\ov s}}\leq  \|\wt f(t)-\wt f(s)\|_{2}^{1-\ov s/s}\|\wt f(t)-\wt f(s)\|_{H^s}^{\ov s/s}\leq C|t-s|^{\eta},\qquad t,s\in[0,T],
 \end{align*}
which proves \eqref{T:EEE}.

Assume now that    $T_+(f_0)<\infty$ and  
\[
\sup_{[0,T_+(f_0))}\|f(t)\|_{H^s}<\infty,
\]
Arguing as above, we find that 
\begin{align*}
  \|f(t)-f(s)\|_{H^{(s+\ov s)/2}}\leq C|t-s|^{(s-\ov s)/2s},\qquad t,s\in[0,T_+(f_0)).
 \end{align*} 
The criterion for global existence in Theorem \ref{T:A}  applied for $ \alpha:=(s+\ov s-2)/2$ and $ \beta:=\ov s-1$  implies that
the solution can be continued in  on an interval $[0,\tau)$ with $\tau>T_+(f_0)$.  
Moreover, it holds that
\[
 f\in C^{(s-\ov s)/2}([0,T],H^{\ov s}(\R))\qquad \text{for all $T\in(0,\tau)$.}
\]
The uniqueness claim in Theorem \ref{T:A} leads us to a contradiction. 
Hence our assumption was false and $T_+(f_0)=\infty$.
\end{proof}

\section{Instantaneous real-analyticity}\label{S4}
In this section we improve the  regularity of the solutions found in Theorems \ref{MT1} and \ref{MT1K}.
To this end we first show that the mapping $\Phi$ defined by \eqref{opa} is actually real-analytic, cf. Proposition \ref{P14}. 
As $[f\mapsto\Phi(f)]$ is not a Nemytskij type operator, we cannot use  classical results for such operators, as presented e.g. in \cite{RS96}.
Instead, we  directly estimate the rest of the associated Taylor series.
     We conclude the section with the proof of Theorem \ref{MT2} which is obtained, via Proposition \ref{P14},  from the  real-analyticity  property of the semiflow as stated   in Theorem \ref{T:A},
     applied in the context of a nonlinear evolution problem related to \eqref{P}.

\begin{prop}\label{P14}
Given $s\in(3/2,2)$, it holds that
\begin{align}\label{ARe}
 \Phi\in C^\omega(H^s(\R),\kL(H^2(\R), H^1(\R))).
\end{align}
\end{prop}
\begin{proof}
Let $\phi:\R\to\R$ be the map defined by $\phi(x):=(1+x^2)^{-1},$ $ x\in\R.$
Then, given $f_0\in H^s(\R)$, it holds that
\begin{align*}
   \Phi(f_0)[h](x)=\PV\int_\R   \frac{\delta_{[x,y]}h'}{y} \phi\Big(\frac{\delta_{[x,y]}f_0}{y}\Big)\, dy,\qquad h\in H^2(\R).
 \end{align*}
 Given $n\in\N$, we let
 \begin{align*}
  \p^n\Phi(f_0)[f_1, \ldots, f_n][h](x):=&\PV\int_\R \frac{\delta_{[x,y]}h'}{y} \Big(\prod_{i=1}^n\frac{\delta_{[x,y]}f_i}{y}\Big)\phi^{(n)}\Big(\frac{\delta_{[x,y]}f_0}{y}\Big) \, dy\\[1ex]
  =&\sum_{k=0, n+k\in2\N}^n a_k^nA_{n+k,n}( f_0,\ldots,f_0 )[\underset{\text{$k$-times}}{\underbrace{f_0,\ldots,f_0}}, f_1,\ldots,f_n,h'](x), 
 \end{align*}
  for $ f_i\in H^s(\R), 1\leq i\leq n$, $ h\in H^2(\R),$ and $x\in\R$, where $a^ n_k, n\in\N, 0\leq k\leq n$ are defined in Lemma \ref{L:IA1}.
 Arguing as in the proof of Lemma \ref{L31}, it follows from the  Lemmas \ref{L21} and \ref{L21'} that $\p^n\Phi(f_0)\in \kL^n_{sym}(H^s(\R),\kL(H^2(\R), H^1(\R))),$  that is $\p^n\Phi(f_0)  $ is a bounded $n$-linear and symmetric operator.
   
   Moreover, given  $f_0,f\in H^s(\R)$, $n\in\N^*$, and $h\in H^2(\R)$,   Fubini's theorem combined with  Lebesgue's dominated convergence theorem and  the continuity of the mapping 
\begin{align*}
 \Big[\tau\mapsto \PV\int_\R \frac{\delta_{[\cdot,y]}h'}{y} \Big(\frac{\delta_{[\cdot,y]}f}{y}\Big)^{n+1} \phi^{(n+1)}\Big(\frac{\delta_{[\cdot,y]}(f_0+\tau f)}{y}\Big)\, dy\Big]:[0,1]\to H^1(\R),
\end{align*} 
yield that
   \begin{align*}
    &\hspace{-1cm}\Phi(f_0+f)[h](x)-\sum_{k=0}^n\frac{\p^k\Phi(f_0)[f]^k[h](x)}{k!}\\[1ex]
    &=\PV\int_\R\frac{\delta_{[x,y]}h'}{y} \Big(\frac{\delta_{[x,y]}f}{y}\Big)^{n+1}\int_0^1 \frac{(1-\tau)^n}{n!}\phi^{(n+1)}\Big(\frac{\delta_{[x,y]}(f_0+\tau f)}{y}\Big)\, d\tau\, dy\\[1ex]
     &=\int_0^1\frac{(1-\tau)^n}{n!} \PV\int_\R \frac{\delta_{[x,y]}h'}{y} \Big(\frac{\delta_{[x,y]}f}{y}\Big)^{n+1} \phi^{(n+1)}\Big(\frac{\delta_{[x,y]}(f_0+\tau f)}{y}\Big)\, dy\, d\tau,
   \end{align*}
and
   \begin{align}
    &\hspace{-1cm}\Big\|\Phi(f_0+f)[h] -\sum_{k=0}^n\frac{\p^k\Phi(f_0)[f]^k[h] }{k!}\Big\|_{H^1}\nonumber\\[1ex]
    \leq& \frac{1}{n!}\max_{\tau\in[0,1]}\Big\|\PV\int_\R \frac{\delta_{[\cdot,y]}h'}{y} \Big(\frac{\delta_{[\cdot,y]}f}{y}\Big)^{n+1} \phi^{(n+1)}\big(\delta_{[\cdot,y]}f_\tau/y\big) \, dy\Big\|_{H^1},\label{TBE}
   \end{align}
where $f_\tau:=f_0+\tau f, $ $0\leq \tau \leq1.$ In order to estimate the right-hand side of \eqref{TBE} we note that
\begin{align}
   &\hspace{-0.5cm} \Big\|\PV\int_\R \frac{\delta_{[\cdot,y]}h'}{y} \Big(\frac{\delta_{[\cdot,y]}f}{y}\Big)^{n+1} \phi^{(n+1)}\big(\delta_{[\cdot,y]}f_\tau/y\big) \, dy\Big\|_{H^1}\nonumber\\[1ex]
    &\leq\sum_{k=0, n+k+1\in2\N}^{n+1} |a_k^{n+1}|
    \big\|A_{n+k+1,n+1}( f_\tau,\ldots,f_\tau )[\underset{\text{$k$-times}}{\underbrace{f_\tau,\ldots,f_\tau}}, \underset{\text{$n+1$-times}}{\underbrace{f,\ldots,f}},h']\big\|_{H^1}\nonumber
    \end{align}
    \begin{align}
     &\leq\sum_{k=0, n+k+1\in2\N}^{n+1} |a_k^{n+1}|
    \big\|A_{n+k+1,n+1}( f_\tau,\ldots,f_\tau )[\underset{\text{$k$-times}}{\underbrace{f_\tau,\ldots,f_\tau}}, \underset{\text{$n+1$-times}}{\underbrace{f,\ldots,f}},h']\big\|_{2}\nonumber\\[1ex]
     &\hspace{0.5cm}+\sum_{k=0, n+k+1\in2\N}^{n+1} |a_k^{n+1}|
    \big\|A_{n+k+1,n+1}( f_\tau,\ldots,f_\tau )[\underset{\text{$k$-times}}{\underbrace{f_\tau,\ldots,f_\tau}}, \underset{\text{$n+1$-times}}{\underbrace{f,\ldots,f}},h'']\big\|_{2}\nonumber\\[1ex]
    &\hspace{0.5cm}+k\sum_{k=0, n+k+1\in2\N}^{n+1} |a_k^{n+1}|
    \big\|A_{n+k+1,n+1}( f_\tau,\ldots,f_\tau )[\underset{\text{$k-1$-times}}{\underbrace{f_\tau,\ldots,f_\tau}}, \underset{\text{$n+1$-times}}{\underbrace{f,\ldots,f}},f_\tau',h']\big\|_{2}\nonumber\\[1ex]
    &\hspace{0.5cm}-2(n+2)\sum_{k=0, n+k+1\in2\N}^{n+1} |a_k^{n+1}|
    \big\|A_{n+k+3,n+2}(f_\tau,\ldots,f_\tau)[\underset{\text{$k+1$-times}}{\underbrace{f_\tau,\ldots,f_\tau}}, \underset{\text{$n+1$-times}}{\underbrace{f,\ldots,f}},f_\tau',h']\big\|_{2}.\label{53}
   \end{align}
Combining the results of Lemmas \ref{L:IA1}-\ref{L:IA4}, we conclude that there exists an integer $p>0$ and a positive constant  $C$ (depending only on $\|f_0\|_{H^s}$) such that for all  $f\in H^s(\R)$ with $\|f\|_{H^s}\leq 1$  and all $n\geq 3$ we have
\begin{align*}
    \Big\|\Phi(f_0+f)  -\sum_{k=0}^n\frac{\p^k\Phi(f_0)[f]^k  }{k!}\Big\|_{\kL(H^2(\R), H^1(\R))}\leq C^{n+1}n^p\|f\|_{H^s}^{n+1}.
   \end{align*}
   The claim follows.
\end{proof}

The following technical results  are used in the proof of Proposition \ref{P14}.

\begin{lemma}\label{L:IA1}
 Let $\phi:\R\to\R$ be defined by $\phi(x):=(1+x^2)^{-1},$ $ x\in\R.$
 Given $n\in\N$, it holds that
 \begin{align*}
  \phi^{(n)}(x)=\frac{1}{(1+x^2)^{n+1}} \sum_{k=0}^na^n_kx^k,
 \end{align*}
where the coefficients  $a^n_k\in\R$ satisfy $|a^n_k|\leq 4^{n}(n+2)!$ for all $0\leq k\leq n.$
Moreover,  $a_k^n=0$ if $n+k\not\in 2\N.$
\end{lemma}
\begin{proof}
 The claim for $n\in\{0,1,2,3\}$ is obvious.
 Assume that the claim holds for some integer $n\geq3$. Since
  \begin{align*}
  (1+x^2)^{n+2}\phi^{(n+1)}(x)= (1+x^2) \sum_{k=1}^n k a^n_kx^{k-1}- 2(n+1)x \sum_{k=0}^na^n_kx^k,
 \end{align*}
 the coefficient $a^{n+1}_{k}, 0\leq k\leq n+1$, of $x^k$ satisfies
 \begin{align*}
  |a^{n+1}_{n+1}|\leq &n|a_n^n|+2(n+1)|a_n^n|\leq 4(n+1)|a_n^n|\leq 4^{n+1}(n+3)!,\\[1ex]
  |a^{n+1}_{n}|\leq &(n-1)|a_{n-1}^n|+2(n+1)|a_{n-1}^n|=0,
 \end{align*}
and for $n-1\geq k\geq 2$ we have
\begin{align*}
  |a^{n+1}_{k}|\leq &(k+1)|a_{k+1}^n|+(k-1)|a_{k-1}^n|+2(n+1)|a_{k-1}^n|\leq 4^{n+1}(n+3)!,
 \end{align*}
 while
  \begin{align*}
  |a^{n+1}_{1}|\leq &2|a_2^n|+2(n+1)|a_0^n|\leq  4^{n+1}(n+3)!,\\[1ex]
  |a^{n+1}_{0}|\leq & |a_{1}^n|\leq 4^{n+1}(n+3)!.
 \end{align*}
 The conclusion is now obvious.
\end{proof}

In the next lemma we estimate the first two terms that appear on the right-hand side of  \eqref{53}.

\begin{lemma}\label{L:IA3}
  Let $n,k\in\N$ satisfy $n\geq 3 $ and $0\leq k\leq n+1$,  and let $s\in(3/2,2).$ 
Given $f,f_\tau\in H^s(\R)$, it holds
 \begin{align}\label{esy2}
 \|A_{n+k+1,n+1}( f_\tau,\ldots,f_\tau )[\underset{\text{$k$-times}}{\underbrace{f_\tau,\ldots,f_\tau}}, \underset{\text{$n+1$-times}}{\underbrace{f,\ldots,f}},\,\cdot\,]\|_{\kL(L_2(\R))}\leq
 C^n n^{4}\max\{1, \|f_\tau\|^{4}_{H^s}\} \|f\|^{n+1}_{H^s}, 
 \end{align}
 with a   constant $C\geq 1$ independent of $n, k, f, $ and $ f_\tau$.
\end{lemma}
\begin{proof}
Similarly as in the proof of Lemma \ref{L21} we write   
\[
A_{n+k+1,n+1}( f_\tau,\ldots,f_\tau )[ f_\tau,\ldots,f_\tau ,  f,\ldots,f ,\cdot]=M-S,
\]
where $M$ is the multiplication operator
\[
M[h](x):=h(x)\PV\int_\R\frac{1}{y}\Big(\frac{\delta_{[x,y]}f}{y}\Big)^{n+1}  \frac{\big(\delta_{[x,y]}f_\tau/y\big) ^{k} }{ \big[1+\big(\delta_{[x,y]}f_\tau/y\big)^2\big]^{n+2}}\, dy
\]
and $S$ is the singular integral operator
\[
S[h](x):=\PV\int_\R \Big(\frac{\delta_{[x,y]}f}{y}\Big)^{n+1}  \frac{\big(\delta_{[x,y]}f_\tau/y\big) ^{k} }{ \big[1+\big(\delta_{[x,y]}f_\tau/y\big)^2\big]^{n+2}}\frac{h(x-y)}{y}\, dy 
\]
for   $h\in L_2(\R)$.
Arguing as in proof of Lemma \ref{L21}, it follows that
\begin{align}\label{esy2a}
 \|M\|_{\kL(L_2(\R))}\leq nC^n\max\{1, \|f_\tau\|_{H^s}\} \|f\|^{n+1}_{H^s} 
 \end{align}
 with a   constant $C\geq1$ independent of $n, k, f, $ and $ f_\tau$.
 
 In order to deal with the operator $S$ we consider the functions $F:\R^{2}\to\R$ and $A:\R\to \R^{2}$ defined by
  \[
 F(x_1,x_2):=\cfrac{x_1^{n+1} x_2^{k}}{(1+x_2^2)^{n+2}},\qquad A:=(A_{1}, A_2):=(f,f_\tau).
 \]
 The function $F$  is   smooth,  $A$  is Lipschitz continuous, and we set
 \[
 a_j:=\|A_j'\|_\infty, \qquad 1\leq j\leq 2.
 \]
Since  $S$ is the singular integral operator with kernel
 \[
 K(x,y):=\frac{1}{y}F\Big(\frac{\delta_{[x,y]}A}{y}\Big),\qquad x\in\R,\, y\neq0,
 \]
 and $|\delta_{[x,y]}A_j/y|\leq a_j  $ for $1\leq j\leq 2$, it is natural to  introduce a smooth periodic function $\wt F$ on $\R^2$, which is  $4a_j$-periodic in the variable $x_j$, $1\leq j\leq 2$, and which matches $F$ on $\prod_{j=1}^2[-a_j,a_j]$.
 More precisely, we choose $\varphi\in C^\infty_0(\R,[0,1])$   with $\varphi=1 $ on $[|x|\leq1]$ and  $\varphi=0 $ on $[|x|\geq 2]$
 and we define $\wt F$ to be the periodic extension of 
 \[
  \Big[(x_1,x_2)\mapsto  F(x_1,x_2)\prod_{j=1}^{2}\varphi\big(\frac{x_j}{a_j}\big)\Big]:Q\to\R,
 \]
 where $Q:= \prod_{j=1}^2[-2a_j,2a_j]$.
We now expend $\wt F$ by its   Fourier series 
\[
\wt F(x_1,x_2)=\sum_{p\in \Z^{2}}\alpha_p \exp\Big(i\sum_{j=1}^2\frac{p_jx_j}{T_j}\Big),
\]
where
\[
T_j:=\frac{2a_j}{\pi},\qquad \alpha_p:=\frac{1}{4^2 a_1a_2}\int_{Q} \wt F(x_1,x_2)\exp\Big(-i\sum_{j=1}^2\frac{p_jx_j}{T_j}\Big)\, d(x_1,x_2), \qquad p\in\Z^2,
\]
and observe that
\[
K(x,y)=\frac{1}{y}\wt F\Big(\frac{\delta_{[x,y]}A}{y}\Big)=\sum_{p\in \Z^{2}} \alpha_p K_p(x,y), \qquad x\in\R,\,  y\neq 0,
\]
with 
\[
K_p(x,y):=\frac{1}{y}\exp\Big(i\frac{\delta_{[x,y]}\big(\sum_{j=1}^2\frac{p_j}{T_j}A_j\big)}{y}\Big) ,\qquad x\in\R,\, y\neq0,\, p\in \Z^{2}.
\]
The kernels $K_p $ define   operators in $\kL(L_2(\R))$ of type \eqref{FFF} and the norms of these operators can be estimated from above by 
$$C\Big(1+  \Big\|\sum_{j=1}^2\frac{p_j}{T_j}A_j'\Big\|_\infty\Big)\leq C(1+|p|) ,\qquad   p\in \Z^{2}.$$
Since  $\sum_{p\in\Z^2} (1+|p|^3)^{-1}<\infty$ we get 
\[
\|S\|_{\kL(L_2(\R))}\leq C\sum_{p\in \Z^{2}} |\alpha_p| (1+|p|)\leq C\sup_{p\in\Z^2}\big[(1+|p|^{4})|\alpha_p|\big],
\]
We estimate next the quantity $\sup_{p\in\Z^2}(1+|p|^{4})|\alpha_p|$.
To this end we write  
\begin{align*}
 \alpha_p=\frac{1}{4^2}\prod_{j=1}^2  \frac{I_j}{a_j},
\end{align*}
where
\begin{align*}
 & I_1:=\int_{-2a_1}^{2a_1} x_1^{n+1}\varphi\big(\frac{x_1}{a_1}\big)e^{-i \frac{p_1x_1}{T_1}}\, dx_1,\qquad I_2:=\int_{-2a_2}^{2a_2} \frac{x_2^{k}}{(1+x_2^2)^{n+2}}\varphi\big(\frac{x_2}{a_2}\big)e^{-i \frac{p_2x_2}{T_2}}\, dx_2.
 \end{align*}
Since $\varphi=0$ in $[|x|\geq 2] $ and $n\geq3$, integration by parts leads us, in the case when $p_1\neq0, $ to
\begin{align}\label{PP11}
   |I_1|\leq\Big(\frac{T_1}{|p_1|}\Big)^{4}\int_{-2a_1}^{2a_1} \Big|\Big( x_1^{n+1}\varphi\big(\frac{x_1}{a_1}\big)\Big)^{(4)}\Big|\, dx_1  \leq C  \frac{2^n  n^{4} a_1^{n+2}}{p_1 ^{4}} ,  
    \end{align}
    and similarly, since $x_2\leq  1+x_2^2 ,$ we find for $p_2\neq0$ that
    \begin{align}\label{PP12}
   |I_2|\leq& C   \frac{n^{4}\max\{a_2,a_2^{5}\}}{p_2^{4}}.    
 \end{align}
 The estimates 
 \begin{align}\label{PP14}
  |I_1|\leq C     2^n  a_1^{n+2},\qquad  |I_2|\leq C a_2, 
 \end{align}
are valid for all $p\in\Z^2$.
Combining \eqref{PP11}-\eqref{PP14}, we arrive at  
 \[
 \sup_{p\in\Z^2}(1+|p|^4)|\alpha_p |\leq C 2^n n^{4}\max\{1, a_2^4\} a_1^{n+1}, 
 \]
 which leads us to
 \[
 \|S\|_{\kL(L_2(\R))}\leq C 2^n n^{4}\max\{1, \|f_\tau'\|^{4}_\infty\} \|f'\|^{n+1}_\infty\leq   n^4 C^n\max\{1, \|f_\tau\|_{H^s}^4\} \|f\|^{n+1}_{H^s}.
 \]
This inequality together with  \eqref{esy2a} proves the desired claim.
\end{proof}\medskip

In the next lemma we estimate the last  two terms on the right-hand side of \eqref{53} in the proof of Proposition \ref{P14}.
\begin{lemma}\label{L:IA4}
  Let $n,k\in\N$ satisfy $n\geq1$ and $0\leq k\leq n+1$. Let further $l\in\{0,1\} $ and  $s\in(3/2,2).$ 
Given $f,f_\tau \in H^s(\R)$, it holds
 \begin{align}\label{esy2b}
 \|A_{n+k+1+2l,n+1+l}
 ( f_\tau,\ldots,f_\tau )[\underset{\text{$k-1+2l$-times}}{\underbrace{f_\tau,\ldots,f_\tau}}, \underset{\text{$n+1$-times}}{\underbrace{f,\ldots,f}},f_\tau',\,\cdot\,]\|_{\kL(H^1(\R),L_2(\R))}\leq C^n \|f_\tau\|_{H^s} \|f\|^{n+1}_{H^s}, 
 \end{align}
 with a   constant $C\geq 1$ independent of $n, k, f, $ and $ f_\tau$.
\end{lemma}
\begin{proof}
 The proof is similar to that of Lemma \ref{L21'}.
\end{proof}

We are now in a position to prove Theorem \ref{MT2} when $\sigma=0$, where we use a parameter trick  which  appears, in other forms, also in  \cite{An90, ES96, PSS15}. 
We present a new idea  which uses only  the abstract result Theorem \ref{T:A} in the context of an  evolution problem related to \eqref{P}, and not explicitly the maximal regularity property as in \cite{An90, ES96, PSS15}. 
The proof when $\sigma>0$ is  almost identical and is also discussed below, but it relies on some properties established in Section \ref{S6}.

\begin{proof}[Proof of Theorem \ref{MT2}] 
Assume first that $\sigma=0$.
We  then pick $f_0\in H^s(\R), $ $s\in (3/2,2),$ and we let $f=f(\cdot;f_0):[0,T_+(f_0))\to H^s(\R)$ denote the unique maximal solution to \eqref{P}, whose existence is guaranteed by Theorem \ref{MT1}. 
We further  choose $\lambda_1,\lambda_2\in(0,\infty)$  and   we define
\[
f_{\lambda_1,\lambda_2}(t,x):=f(\lambda_1 t,x+\lambda_2 t), \qquad   x\in\R, \, 0\leq t<T_+:=T_+(f_0)/\lambda_1.
\]
Classical arguments show that
\[f_{\lambda_1,\lambda_2}\in C([0,T_+),H^s(\R))\cap C((0,T_+), H^2(\R))\cap C^1((0,T_+), H^1(\R)).\]
We next  introduce the  function $u :=(u_1,u_2,u_3):[0,T_+)\to\R^2\times H^s(\R)$, where 
\[(u_1,u_2)(t):=(\lambda_1,\lambda_2),\qquad u_3(t):=f_{\lambda_1,\lambda_2}(t),\qquad 0\leq t<T_+,\] 
and we note that $u$ solves the quasilinear evolution problem
\begin{align}\label{QC}
\dot u= \Psi(u)[u],\quad t>0,\qquad u(0)=(\lambda_1,\lambda_2,f_0),
\end{align}
with $ \Psi:(0,\infty)^2\times H^s(\R)\to\kL(\R^2\times H^2(\R), \R^2\times H^1(\R))$ denoting the operator defined by
\begin{align}\label{QO}
  \Psi((v_1,v_2,v_3))[(u_1,u_2,u_3)]:=(0,0,v_1\Phi(v_3)[u_3]+v_2\p_x u_3).
\end{align}
 Proposition \ref{P14}  immediately yields 
\begin{align*}
 \Psi\in C^\omega((0,\infty)^2\times H^s(\R),\kL(\R^2\times H^2(\R), \R^2\times H^1(\R)))\qquad\text{for all $s\in(3/2,2)$.} 
\end{align*}

Given $v:=(v_1,v_2,v_3)\in (0,\infty)^2\times H^s(\R),$ the operator $\Psi(v)$ can be represented as a matrix
\[
\Psi(v)=\begin{pmatrix}
         0&0\\[1ex]
         0&v_1\Phi(v_3) +v_2\p_x
        \end{pmatrix}:\R^2\times H^2(\R)\to \R^2\times H^1(\R),
\]
and we infer from \cite[Corollary I.1.6.3]{Am95} that $-\Psi(v)\in\kH(\R^2\times H^2(\R), \R^2\times H^1(\R)) $ if and only if 
\begin{align}\label{GPq}
-(v_1\Phi(v_3)+v_2\p_x)\in\kH( H^2(\R), H^1(\R)).
\end{align}
We note that $v_2\p_x$ is a first order Fourier multiplier  and its symbol is purely imaginary.
Therefore, obvious modifications of the arguments presented in the proofs of Theorem \ref{T1} and   Proposition \ref{PP1} enable us to conclude that the property \eqref{GPq} is 
satisfied for each $(v_1,v_2,v_3)\in (0,\infty)^2\times H^s(\R)$ and   $s\in(3/2,2)$.  
Setting $\F_0:=\R^2\times H^1(\R)  $ and $\F_1:=\R^2\times H^2(\R),$ it holds that
\[
[\F_0,\F_1]_\theta=\R^2\times H^{1+\theta}(\R),\qquad\theta\in(0,1),
\]
and we may now apply  Theorem \ref{T:A} in the context of the quasilinear parabolic problem \eqref{QC} to conclude (similarly as in the proof of Theorem \ref{MT1}), 
for each $u_0=(\lambda_1,\lambda_2,f_0)\in (0,\infty)^2\times H^{s}(\R), $ $s\in(3/2,2)$, the existence of  a unique maximal solution  
\[ u:=u(\cdot; u_0)\in C([0,T_+(u_0)),(0,\infty)^2\times H^{s}(\R))\cap C((0,T_+(u_0)), \F_1)\cap C^1((0,T_+(u_0)), \F_0).\]
Additionally, the set
\[
\0:=\{(\lambda_1,\lambda_2,f_0,t)\,:\, t\in(0,T_+((\lambda_1,\lambda_2,f_0)))\}
\]
is open in $(0,\infty)^2\times H^s(\R)\times(0,\infty)$ and 
\[
[(\lambda_1,\lambda_2,f_0,t)\mapsto u(t;(\lambda_1,\lambda_2,f_0))]:\0\to \R^2\times H^s(\R)
\]
is a real-analytic map.

So, if we fix $f_0\in H^s(\R),$ then we may conclude that  
\[
\frac{T_+(f_0)}{\lambda_1}=T_+((\lambda_1,\lambda_2,f_0)) \qquad\text{for all $(\lambda_1,\lambda_2)\in(0,\infty)^2.$}
\]
As we want to prove that $f=f(\cdot;f_0)$ is real-analytic in $(0,T_+(f_0))\times\R$, it suffices to  establish the real-analyticity property in a small ball around $(t_0,x_0)$ for each $x_0\in\R$ and $t_0\in(0,T_+(f_0))$.
Let thus $(t_0,x_0)\in (0,T_+(f_0))\times\R$ be arbitrary.
For  $(\lambda_1,\lambda_2)\in\bB ((1,1),\e)\subset(0,\infty)^2$, with $\e$ chosen suitably small,  we have that
\[
t_0<T_+((\lambda_1,\lambda_2,f_0)) \qquad\text{for all $(\lambda_1,\lambda_2)\in\bB ((1,1),\e),$}
\]
and therewith
\[
\bB ((1,1),\e)\times\{f_0\}\times\{t_0\}\subset \0. 
\]
Moreover, since $u_3(\,\cdot\,; u_0)=f_{\lambda_1,\lambda_2}$, the restriction  
\begin{align}\label{IQ}
[(\lambda_1,\lambda_2)\mapsto f_{\lambda_1,\lambda_2}(t_0)]:\bB((1,1),\e)\to H^s(\R)
\end{align}
is  a real-analytic map.
Since $[h\mapsto h(x_0-t_0)]:H^s(\R)\to \R$ is a linear operator,   the composition
\begin{align}\label{11}
[(\lambda_1,\lambda_2)\mapsto f(\lambda_1t_0,x_0-t_0+\lambda_2t_0)]:\bB ((1,1),\e)\to \R
\end{align}
is real-analytic too.
  Furthermore, for $\delta>0$  small,   the mapping $\varphi:\bB ((t_0,x_0),\delta)\to\bB ((1,1),\e)$ with
  \begin{align}\label{12}
  \varphi(t,x):= \Big(\frac{t}{t_0},\frac{x-x_0+t_0}{t_0}\Big)
  \end{align}
is well-defined and real-analytic, and therefore the composition of the functions defined by \eqref{11} and \eqref{12}, that is the mapping 
$$  \big[(t,x)\mapsto  f(t,x)\big]:\bB ((t_0,x_0),\delta)\to\R,$$
is also real-analytic. This proves the first claim. 

Finally, the property   $f\in C^\omega ((0,T_+(f_0)), H^k(\R))$, for arbitrary $k\in\N$, is an immediate consequence of \eqref{IQ}. 

The arguments presented above carry over to the case when $\sigma>0$ (with the obvious modifications).
If $\sigma>0$, the operator $v_2\p_x$ appearing in \eqref{GPq} can be regarded as being a lower order perturbation and therefore the generator property of $\Psi(v)$
 follows in this case directly from the corresponding property of the original operator, cf. Theorem \ref{TK2}.
\end{proof}

\section{The Muskat problem with surface tension and gravity effects}\label{S6}

We now consider  surface tension forces acting  at the interface between the fluids, that is we take $\sigma>0$. 
The motion of the fluids may  also be influenced  by gravity, but we make no restrictions on  $\Delta_\rho$ which is now an arbitrary real number.
If we model flows in a vertical  Hele-Shaw cell  this means in particular that the lower fluid may be less dense than the fluid above.
Since $\Delta_\rho$ can be zero, \eqref{P} is also a model for fluid motions in a horizontal Hele-Shaw cell as for these flows the effects due to gravity are usually neglected, that is $g=0$.  
Again, we rescale the time appropriately and rewrite \eqref{P} as the following system 
\begin{equation}\label{Pest}
\left\{
\begin{array}{rlll}
 \p_tf(t,x)\!\!&=&\!\!\displaystyle f'(t,x)\PV\int_\R\frac{f(t,x)-f(t,x-y)}{ y^2+(f(t,x)-f(t,x-y))^2 }(\kappa(f))'(t,x-y)\, dy\\[2ex]
 &&+\displaystyle \PV\int_\R\frac{y}{ y^2+(f(t,x)-f(t,x-y))^2 }(\kappa(f))'(t,x-y)\, dy\\[2ex]
 &&+\displaystyle\Theta\PV\int_\R\frac{y(f'(t,x)-f'(t,x-y))}{ y^2+(f(t,x)-f(t,x-y))^2 }\, dy\qquad \text{for $ t>0$, $x\in\R$},\\[2ex]
  f(0,\cdot)\!\!&=&\!\!f_0,
\end{array}
\right.
\end{equation} 
with 
$$\Theta:=\frac{\Delta\rho}{\sigma}\in\R.$$
Since
\[
(\kappa(f))'=\frac{f'''}{(1+f'^2)^{3/2}}-3\frac{f'{f''}^2}{(1+f'^2)^{5/2}},
\]
we observe that the first equation of \eqref{Pest} is again  quasilinear, but now this property is due to the fact that  $(\kappa(f))'$ is an affine function in the variable $f'''$.

To be more precise we set
\begin{align}
\Phi_\sigma(f)[h](x):=&\displaystyle f'(x)\PV\int_\R\frac{\delta_{[x,y]}f}{ y^2+(\delta_{[x,y]}f)^2 }\Big(\frac{h'''}{(1+f'^2)^{3/2}}-3\frac{f'f''h''}{(1+f'^2)^{5/2}}\Big)(x-y)\, dy\nonumber\\[1ex]
 &+ \displaystyle\PV\int_\R\frac{y}{ y^2+(\delta_{[x,y]}f)^2 }\Big(\frac{h'''}{(1+f'^2)^{3/2}}-3\frac{f'f''h''}{(1+f'^2)^{5/2}}\Big)(x-y)\, dy\nonumber\\[1ex]
&+\Theta\PV\int_\R\frac{y(\delta_{[x,y]}h')}{ y^2+(\delta_{[x,y]}f)^2 }\, dy,\label{Def1}
\end{align}
and we recast the problem \eqref{Pest}  in the following compact form
\begin{equation}\label{Pest1}
  \dot f= \Phi_\sigma(f)[f],\quad t>0,\qquad f(0)=f_0.
\end{equation}
We emphasize that there are also other ways to write \eqref{Pest} as a quasilinear problem. For example the terms containing  only $f^{(l)}, $ $0\leq l\leq 2$, can be 
viewed as a nonlinear function $[f\mapsto F(f)]$ which would appear as an additive term to the right-hand side of \eqref{Pest1}, with $\Phi_\sigma(f)[h]$ modified accordingly.
However, the formulation \eqref{Def1}-\eqref{Pest1} appears to us as being optimal as it allows us to consider the largest set of initial data among  all  formulations. 
To be more precise,   the operator introduced by \eqref{Def1}  satisfies, with the notation in Lemma \ref{L21} and Remark \ref{R2}, 
the following relation
\begin{align*} 
  \Phi_\sigma(f)[h]=&f'B_{1,1}(f,f)\Big[\frac{h'''}{(1+f'^2)^{3/2}}-3\frac{f'f''h''}{(1+f'^2)^{5/2}}\Big]+B_{0,1}(f)\Big[\frac{h'''}{(1+f'^2)^{3/2}}-3\frac{f'f''h''}{(1+f'^2)^{5/2}}\Big]\nonumber\\[1ex]
  &+\Theta A_{0,0}(f)[h'],  
\end{align*}
and we  now claim, based on the results in Section \ref{S4}, that 
\begin{align}\label{reg}
  \Phi_\sigma\in C^{\omega}  (H^2(\R),\kL(H^3(\R), L_2(\R))) .
\end{align}
Indeed, arguing as  in Section \ref{S4}, it follows that
\begin{equation}\label{reg12}
\begin{aligned}
&\left[f\mapsto \Big[h\mapsto\PV\int_\R\frac{\delta_{[\cdot,y]}f}{ y^2+(\delta_{[\cdot,y]}f)^2 }h(\cdot-y)\, dy\Big]\right]\in C^\omega(H^2(\R),\kL(L_2(\R))), \\[1ex]
& \left[f\mapsto \Big[h\mapsto \PV\int_\R\frac{y}{ y^2+(\delta_{[\cdot,y]}f)^2 }h(\cdot-y)\, dy\Big]\right]\in C^\omega(H^2(\R),\kL(L_2(\R))),\\[1ex]
& \left[f\mapsto \Big[h\mapsto\PV\int_\R\frac{y(\delta_{[\cdot,y]}h')}{ y^2+(\delta_{[\cdot,y]}f)^2 }\, dy\Big]\right]\in C^\omega(H^2(\R),\kL(H^3(\R),L_2(\R))).
\end{aligned}
\end{equation}
Moreover,  classical arguments, see e.g. \cite[Theorem 5.5.3/4]{RS96}, yield  that
\begin{align}\label{reg22}
&\left[f\mapsto \Big[h\mapsto\frac{h'''}{(1+f'^2)^{3/2}}-3\frac{f'f''h''}{(1+f'^2)^{5/2}}\Big]\right]\in C^\omega(H^2(\R),\kL(H^3(\R),L_2(\R))).
\end{align}
The  relations \eqref{reg12}-\eqref{reg22}  immediately imply \eqref{reg}.

In the following, we prove that  $\Phi_\sigma(f)$ is, for each $f\in H^2(\R)$,    the generator of a strongly continuous and analytic semigroup in $\kL( L_2(\R))$, that is
\begin{align*}
  -\Phi_\sigma(f)\in \kH  (H^3(\R), L_2(\R)).
\end{align*}
To this end we write
\begin{align} 
  \Phi_\sigma =&\Phi_{\sigma,1}+\Phi_{\sigma,2}, \label{MP2}
\end{align}
where
\begin{align}
 \Phi_{\sigma,1}(f)[h]:=&f'B_{1,1}(f,f)\Big[\frac{h'''}{(1+f'^2)^{3/2}}\Big]+B_{0,1}(f)\Big[\frac{h'''}{(1+f'^2)^{3/2}} \Big],\label{FT1}\\[1ex]
   \Phi_{\sigma,2}(f)[h]:=&-3f'B_{1,1}(f,f)\Big[\frac{f'f''h''}{(1+f'^2)^{5/2}}\Big]-3B_{0,1}(f)\Big[\frac{f'f''h''}{(1+f'^2)^{5/2}}\Big]+\Theta A_{0,0}(f)[h']
\end{align}
for $f\in H^2(\R)$,  $h\in H^3(\R)$. 
Since $\Phi_{\sigma,2}(f)\in\kL(H^{8/3}(\R), L_2(\R)) $ and $[L_2(\R), H^3(\R)]_{8/9}=H^{8/3}(\R),$ we can view $\Phi_{\sigma,2}(f)$ as being a lower order perturbation, cf. 
\cite[Proposition 2.4.1]{L95}, and we only need to establish the generator property for the leading order term  $\Phi_{\sigma,1}(f).$ 
Similarly as in Section \ref{S3}, we consider    a continuous mapping
\[
 [\tau\mapsto \Phi_{\sigma,1}(\tau f)]:[0,1]\to\kL(H^3(\R), L_2(\R)),
\]
which  transforms the operator $\Phi_{\sigma,1}(f)$ into the operator
\[
\Phi_{\sigma,1}(0)=B_{0,1}(0)\circ\p_x^3=-\pi (\p_x^4)^{3/4},
\]
where $(\p_x^4)^{3/4}$ is the Fourier multiplier with symbol $[\xi\mapsto |\xi|^3].$
We now establish the following result.

\begin{thm}\label{TK1} 
Let  $f\in H^2(\R)$   and     $\mu>0$ be given.

Then, there exist $\e\in(0,1)$, a finite $\e$-localization family  $\{\pi_j^\e\,:\, -N+1\leq j\leq N\} $ satisfying \eqref{i}-\eqref{v}, a constant $K=K(\e)$, and for each  $ j\in\{-N+1,\ldots,N\}$ and $\tau\in[0,1]$ there 
exist   operators $$\bA_{ j,\tau}\in\kL(H^3(\R), L_2(\R))$$
 such that 
 \begin{equation}\label{DEK1}
  \|\pi_j^\e \Phi_{\sigma,1}(\tau f)[h]-\bA_{j,\tau}[\pi^\e_j h]\|_{2}\leq \mu \|\pi_j^\e h\|_{H^3}+K\|  h\|_{H^2}
 \end{equation}
 for all $ j\in\{-N+1,\ldots,N\}$, $\tau\in[0,1],$ and  $h\in H^3(\R)$. The operators $\bA_{j,\tau}$ are defined  by 
  \begin{align} 
 \bA_{j,\tau }:=&- \frac{ \pi}{(1+\tau^2 f'^2(x_j^\e) )^{3/2}}   (\p_x^4 )^{3/4}, \qquad |j|\leq N-1,\label{FMK1}
 \end{align}
 where $x_j^\e$ is a point belonging to  $\supp  \pi_j^\e$, respectively
 \begin{align} 
 \bA_{N,\tau }:=& - \pi (\p_x^4)^{3/4}.\label{FMK2}
 \end{align}
\end{thm}
\begin{proof} 
 Let   $\{\pi_j^\e\,:\, -N+1\leq j\leq N\} $ be a $\e$-localization family satisfying  the properties \eqref{i}-\eqref{v} and $\{\chi_j^\e\,:\, -N+1\leq j\leq N\} $ be an associated family satisfying \eqref{c1}-\eqref{c3}, with $\e\in(0,1)$
which will be fixed below.

To deal with both terms of $\Phi_{\sigma,1}(\tau f)$, cf. \eqref{FT1}, at once, we consider the  operator
\[
K_{a}(\tau f)[h]:=f_{a,\tau}'B_{1,1}(f_{a,\tau},\tau f)\Big[\frac{h'''}{(1+\tau^2 f'^2)^{3/2}}\Big],
\]
where, for $a\in\{0,1\}$, we set
\[
f_{a,\tau}:=(1-a)\tau f+a{\, \rm id}_\R.
\] 
For $a=0$ we recover the first term in the definition of $\Phi_{\sigma,1}(\tau f)[h]$,
while for $a=1$  the expression matches the second one.

In the following $h\in H^3(\R)$ is  arbitrary.
Again, constants which are
independent of $\e$ (and, of course, of $h\in H^3(\R)$, $\tau\in [0,1]$, $a\in\{0,1\}$, and $j \in \{-N+1, \ldots, N\}$) are denoted by $C$, while the constants that we denote by $K$ may depend only upon $\e.$  
We further let 
\[
\bA^a_{j,\tau} :=-\pi \frac{  f_{a,\tau}'^2(x_j^\e)}{(1+\tau^2 f'^2(x_j^\e) )^{5/2}}   (\p_x^4 )^{3/4} \qquad\text{for $|j|\leq N-1$},
\]
respectively
\[
\bA^a_{N,\tau} :=-\pi  a^2   (\p_x^4 )^{3/4}.
\]
We analyze the cases $j=N$ and $|j|\leq N-1$ separately.\medskip

\noindent{\em The case $|j|\leq N-1$.\,\,} For $|j|\leq N-1$     we  write
\begin{align}\label{mm1}
 \pi_j^\e K_{a}(\tau f)[h]-\bA^a_{j,\tau}[\pi_j^\e h]:=T_1[h]+T_2[h]+T_3[h],
\end{align}
where 
\begin{align*}
T_1[h]:=& \pi_j^\e K_{a}(\tau f)[h]- f_{a,\tau}'(x_j^\e)B_{1,1}(f_{a,\tau},\tau f)\Big[\frac{\pi_j^\e h'''}{(1+\tau^2 f'^2)^{3/2}}\Big], \\[1ex]
T_2[h]:=&f_{a,\tau}'(x_j^\e)B_{1,1}(f_{a,\tau},\tau f)\Big[\frac{\pi_j^\e h'''}{(1+\tau^2 f'^2)^{3/2}}\Big]-\frac{f_{a,\tau}'(x_j^\e)}{ (1+\tau^2 f'^2(x_j^\e))^{3/2}}B_{1,1}(f_{a,\tau},\tau f)[ \pi_j^\e h''' ],\\[1ex]
T_3[h]:=&\frac{f_{a,\tau}'(x_j^\e)}{ (1+\tau^2 f'^2(x_j^\e))^{3/2}}B_{1,1}(f_{a,\tau},\tau f)[ \pi_j^\e h''' ] - \bA^a_{j,\tau}[\pi_j^\e h].
\end{align*}

We consider first the term $T_1[h].$   The identity $\chi_j^\e\pi_j^\e=1$ on $\supp \pi_j^\e$   and integration by parts lead us to the following relation
\begin{align*} 
T_1[h] =&\chi_j^\e (f_{a,\tau}' -f_{a,\tau}'(x_j^\e))B_{1,1}(f_{a,\tau}, \tau f)\Big[\frac{\pi_j^\e h'''}{(1+\tau^2 f'^2)^{3/2}}\Big]  \\[1ex]
&+(1-\chi_j^\e )(f_{a,\tau}'(x_j^\e)-f_{a,\tau}' ) B_{1,1}(f_{a,\tau},\tau f)\Big[\frac{   {\pi_j^\e}'h''}{(1+\tau^2f'^2)^{3/2}}  -3\tau^2 \frac{\pi_j^\e f'f''h''}{(1+\tau^2f'^2)^{5/2}}\Big] \\[1ex]
&+(f_{a,\tau}'(x_j^\e)-f_{a,\tau}' )\Big\{ B_{1,1}(\chi_j^\e,\tau f)\Big[\frac{\pi_j^\e f_{a,\tau}' h''}{(1+\tau^2 f'^2)^{3/2}}\Big]-2B_{2,1}(f_{a,\tau},\chi_j^\e,\tau f)\Big[\frac{\pi_j^\e h''}{(1+\tau^2 f'^2)^{3/2}}\Big]\Big\}\\[1ex]
&-2\tau^2(f_{a,\tau}'(x_j^\e)-f_{a,\tau}' )B_{3,2}(f_{a,\tau},\chi_j^\e, f,\tau f,\tau f)\Big[\frac{\pi_j^\e f'h''}{(1+\tau^2 f'^2)^{3/2}}\Big]\\[1ex]
&+2\tau^2(f_{a,\tau}'(x_j^\e)-f_{a,\tau}' )B_{4,2}(f_{a,\tau},\chi_j^\e,f, f,\tau f,\tau f)\Big[\frac{\pi_j^\e h''}{(1+\tau^2 f'^2)^{3/2}}\Big] \\[1ex]
&+3\tau^2f_{a,\tau}'  \Big\{\pi_j^\e B_{1,1}(f_{a,\tau},\tau f)\Big[\frac{    f'f''h'' }{(1+\tau^2f'^2)^{5/2}}\Big]  
-B_{1,1}(f_{a,\tau},\tau f)\Big[\frac{ \pi_j^\e   f'f''h'' }{(1+\tau^2f'^2)^{5/2}}\Big] \Big\}\\[1ex]
&+  f_{a,\tau}'  \Big\{B_{1,1}(\pi_j^\e,\tau f)\Big[\frac{ f_{a,\tau}' h''}{(1+\tau^2 f'^2)^{3/2}}\Big] +B_{1,1}(f_{a,\tau},\tau f)\Big[\frac{ {\pi_j^\e}' h''}{(1+\tau^2 f'^2)^{3/2}}\Big] \Big\}\end{align*}
\begin{align*}
&-  2f_{a,\tau}'   B_{2,1}(\pi_j^\e,f_{a,\tau},\tau f)\Big[\frac{ h''}{(1+\tau^2 f'^2)^{3/2}}\Big]  -  2\tau^2 f_{a,\tau}'   B_{3,2}(\pi_j^\e,f_{a,\tau},f,\tau f,\tau f)\Big[\frac{ f'h''}{(1+\tau^2 f'^2)^{3/2}}\Big]  \\[1ex]
&+ 2\tau^2 f_{a,\tau}'   B_{4,2}(\pi_j^\e,f_{a,\tau},f,f,\tau f,\tau f)\Big[\frac{ h''}{(1+\tau^2 f'^2)^{3/2}}\Big] . 
\end{align*}
Using Remark \ref{R2}, the interpolation property \eqref{IP}, Young's inequality, and the H\"older  continuity of $f_{a,\tau}'$, it follows that
\begin{align}
 \|T_1[h]\|_2\leq&C \big[\|\chi_j^\e(f_{a,\tau}'-f_{a,\tau}'(x_j^\e))\|_\infty\|\pi_j^\e h'''\|_2+\|\pi_j^\e h''\|_\infty\big]+K\|h\|_{H^2}\nonumber\\[1ex]
 \leq &\frac{\mu}{3}\|\pi_j^\e h\|_{H^3}+K\|h\|_{H^2}\label{mm2}
\end{align}
provided that $\e$ is sufficiently small.

Furthermore, we  have
\begin{align*}
 T_2[h]=\frac{\tau^2f_{a,\tau}'(x_j^\e)}{ (1+\tau^2 f'^2(x_j^\e))^{3/2}}B_{1,1}(f_{a,\tau},\tau f)\Big[ Q(f'(x_j^\e)-f') \pi_j^\e h''' \Big],
\end{align*}
where
\[
Q:=\frac{(f'(x_j^\e)+f')\big[(1+\tau^2f'^2)^{2}+(1+\tau^2f'^2)(1+\tau^2 f'^2(x_j^\e))+(1+\tau^2 f'^2(x_j^\e))^2\big]}{(1+\tau^2f'^2)^{3/2}\big[(1+\tau^2f'^2)^{3/2}+(1+\tau^2f'^2(x_j^\e))^{3/2}\big]},
\]
and therewith
\begin{align}
 \|T_2[h]\|_2\leq C   \|\chi_j^\e(f_{a,\tau}'-f_{a,\tau}'(x_j^\e))\|_\infty\|\pi_j^\e h'''\|_2 \leq \frac{\mu}{3}\|\pi_j^\e h\|_{H^3}+K\|h\|_{H^2}\label{mm3}
\end{align}
if $\e$ is sufficiently small.

Finally, arguing as in {\em Step 3} of othe proof of Theorem \ref{T2}, we deduce that   for $\e$ sufficiently small we have
\begin{align}\label{mm4}
 \|T_3[h]\|_2\leq \frac{\mu}{3}\|\pi_j^\e h\|_{H^3}+K\|h\|_{H^2}.
\end{align}
Gathering \eqref{mm1}-\eqref{mm4}, we have established the desired estimate \eqref{DEK1}  for  $|j|\leq N-1.$\medskip

\noindent{\em The case $j=N$.\,\,}  For $j=N$  we write
\begin{align}\label{mmm1}
 \pi_N^\e K_{a}(\tau f)[h]-\bA^a_{N,\tau}[\pi_N^\e h]=:S_1[h]+S_2[h]+S_3[h]+S_4[h],
\end{align}
where 
\begin{align*}
S_1[h]:=& \pi_N^\e K_{a}(\tau f)[h]- a B_{1,1}(f_{a,\tau},\tau f)\Big[\frac{\pi_N^\e h'''}{(1+\tau^2 f'^2)^{3/2}}\Big], \\[1ex]
S_2[h]:=&aB_{1,1}(f_{a,\tau},\tau f)\Big[\frac{\pi_N^\e h'''}{(1+\tau^2 f'^2)^{3/2}}\Big]-a B_{1,1}(f_{a,\tau},\tau f)[ \pi_N^\e h''' ],\\[1ex]
S_3[h]:=&a B_{1,1}(f_{a,\tau},\tau f)[ \pi_N^\e h''' ]-a^2B_{0,1}(0)[ \pi_N^\e h''' ], \\[1ex]
S_4[h]:=&a^2B_{0,1}(0)[ \pi_N^\e h''' ]- \bA^a_{N,\tau}[\pi_N^\e h].
\end{align*}
Similarly as for $T_1[h],$ we derive the following identity 
\begin{align*} 
S_1[h] =&\tau(1-a)\chi_N^\e f' B_{1,1}(f_{a,\tau}, \tau f)\Big[\frac{\pi_N^\e h'''}{(1+\tau^2 f'^2)^{3/2}}\Big] \\[1ex]
&-\tau(1-a)f'(1-\chi_N^\e ) B_{1,1}(f_{a,\tau},\tau f)\Big[\frac{   {\pi_N^\e}'h''}{(1+\tau^2f'^2)^{3/2}}-3\tau^2 \frac{\pi_N^\e f'f''h''}{(1+\tau^2f'^2)^{5/2}}\Big] \\[1ex]
&-\tau(1-a)f' \Big\{B_{1,1}(\chi_N^\e,\tau f)\Big[\frac{\pi_N^\e f_{a,\tau}' h''}{(1+\tau^2 f'^2)^{3/2}}\Big] -2B_{2,1}(f_{a,\tau},\chi_N^\e,\tau f)\Big[\frac{\pi_N^\e h''}{(1+\tau^2 f'^2)^{3/2}}\Big]\Big\} \\[1ex]
&+2\tau^3(1-a)f'B_{3,2}(f_{a,\tau},\chi_N^\e, f,\tau f,\tau f)\Big[\frac{\pi_N^\e f'h''}{(1+\tau^2 f'^2)^{3/2}}\Big] \\[1ex]
&-2\tau^3(1-a)f'B_{4,2}(f_{a,\tau},\chi_N^\e,f, f,\tau f,\tau f)\Big[\frac{\pi_N^\e h''}{(1+\tau^2 f'^2)^{3/2}}\Big]\\[1ex]
&+3\tau^2f_{a,\tau}'  \Big\{\pi_N^\e B_{1,1}(f_{a,\tau},\tau f)\Big[\frac{    f'f''h'' }{(1+\tau^2f'^2)^{5/2}}\Big]  
-B_{1,1}(f_{a,\tau},\tau f)\Big[\frac{ \pi_N^\e   f'f''h'' }{(1+\tau^2f'^2)^{5/2}}\Big] \Big\}\\[1ex]
&+  f_{a,\tau}'  \Big\{B_{1,1}(\pi_N^\e,\tau f)\Big[\frac{ f_{a,\tau}' h''}{(1+\tau^2 f'^2)^{3/2}}\Big] +B_{1,1}(f_{a,\tau},\tau f)\Big[\frac{ {\pi_N^\e}' h''}{(1+\tau^2 f'^2)^{3/2}}\Big] \Big\}\\[1ex]
&-  2f_{a,\tau}'   B_{2,1}(\pi_N^\e,f_{a,\tau},\tau f)\Big[\frac{ h''}{(1+\tau^2 f'^2)^{3/2}}\Big] -  2\tau^2 f_{a,\tau}'   B_{3,2}(\pi_N^\e,f_{a,\tau},f,\tau f,\tau f)\Big[\frac{ f'h''}{(1+\tau^2 f'^2)^{3/2}}\Big]  \\[1ex]
&+ 2\tau^2 f_{a,\tau}'   B_{4,2}(\pi_N^\e,f_{a,\tau},f,f,\tau f,\tau f)\Big[\frac{ h''}{(1+\tau^2 f'^2)^{3/2}}\Big].  
\end{align*}
Recalling that $f'$ vanishes at infinity, we obtain  in virtue of Remark \ref{R2}, the interpolation property \eqref{IP}, and Young's inequality  that
\begin{align}
 \|S_1[h]\|_2\leq \frac{\mu}{3}\|\pi_j^\e h\|_{H^3}+K\|h\|_{H^2}\label{mmm2}
\end{align}
provided that $\e$ is sufficiently small.
Furthermore,   Remark \ref{R2} implies that  for $\e$ sufficiently small 
\begin{align}
 \|S_2[h]\|_2=&a\Big\|B_{11}(f_{a,\tau},\tau f)\Big[\frac{\pi_N^\e h'''}{(1+\tau^2 f'^2)^{3/2}}\big[1-(1+\tau^2 f'^2)^{3/2}\big]\Big]\Big\|_2\nonumber\\[1ex]
 \leq& C\|\pi_N^\e h'''\|_2\|\chi_N^\e\big[1-(1+\tau^2 f'^2)^{3/2}\big]\|_\infty\leq \frac{\mu}{3}\|\pi_j^\e h\|_{H^3}+K\|h\|_{H^2}\label{mmm3}.
\end{align}
Since $a(1-a)=0,$ we compute that
\[
S_3[h]=-a^2B_{2,1}(f,f,\tau f)[\pi_\e^N h'''],
\]
and the  arguments presented in {\em Step 4} of the proof of Theorem \ref{T1}  yield 
\begin{align}
 \|S_3[h]\|_2 \leq \frac{\mu}{3}\|\pi_j^\e h\|_{H^3}+K\|h\|_{H^2}\label{mmm4}
\end{align}
for   $\e$ sufficiently small.
 Finally,
\begin{align}
 \|S_4[h]\|_2=a^2\|B_{0,1}(0)[3(\pi_N^\e)' h''+3(\pi_N^\e)''h'+(\pi_N^\e )'''h] \|_2\leq K\|h\|_{H^2}\label{mmm5},
\end{align}
and combining \eqref{mmm1}-\eqref{mmm5} we obtain the estimate \eqref{DEK1} for $j=N$. This completes the proof.
\end{proof}

The Fourier multipliers defined by 
\eqref{FMK1}-\eqref{FMK2} are   generators of strongly continuous analytic semigroups in $\kL(L_2(\R))$ and they satisfy resolvent estimates which 
are uniform with respect to $x_j^\e\in\R $ and  $\tau\in [0, 1].$
More precisely, we have the following result.

\begin{prop}\label{PPK1} 
Let  $f\in H^2(\R)$  be fixed. Given $x_0\in \R$ and $\tau\in[0,1]$, let
\[
\bA_{x_0,\tau}:=-\frac{\pi}{(1+\tau^2 f'^2(x_0))^{3/2}}(\p_x^4)^{3/4}.
\]
  Then, there exists a constant  $\kappa_0\geq1$  such that  
  \begin{align}\label{13K}
&\lambda-\bA_{x_0,\tau}\in{\rm Isom}(H^3(\R), L_2(\R)),\\[1ex]
\label{14K}
& \kappa_0\|(\lambda-\bA_{x_0,\tau})[h]\|_{2}\geq  |\lambda|\cdot\|h\|_{2}+\|h\|_{H^3}
\end{align}
for all $x_0\in\R$, $\tau\in[0,1],$    $\lambda\in\C$ with $\re \lambda\geq 1$, and $h\in H^3(\R)$.
\end{prop}
\begin{proof}
 The proof is similar to that of Proposition \ref{PP1} and therefore we omit it.
\end{proof}

We now conclude with the  following generation result.

\begin{thm}\label{TK2}
 Let $f\in H^2(\R)$  be given. Then 
 \begin{align*} 
  -\Phi_\sigma(f)\in \kH(H^3(\R), L_2(\R)).
 \end{align*}
\end{thm}
\begin{proof}
As mentioned in the discussion preceding Theorem \ref{TK1}, we only need to prove the claim for the leading order term $\Phi_{\sigma,1}(f).$
 Let  $\kappa_0\geq1$ be the constant determined in Proposition \ref{PPK1} and let $\mu:=1/2\kappa_0$.
 In virtue of Theorem \ref{TK1} there exist   constants $\e\in(0,1) $    and $K=K(\e)>0$, a $\e$-localization family $\{\pi_j^\e\,:\, -N+1\leq j\leq N\}$ that satisfies \eqref{i}-\eqref{v}, 
 and for each  $ -N+1\leq j\leq N$ and $\tau\in[0,1]$
  operators $\bA_{j,\tau}\in\kL(H^3(\R), L_{2}(\R))$ such that
 \begin{equation}\label{DEK4}
  \|\pi_j^\e\Phi_{\sigma,1}(\tau f)[h]-\bA_{j,\tau}[\pi^\e_j h]\|_{2}\leq \frac{1}{2\kappa_0}\|\pi_j^\e h\|_{H^3}+K\|  h\|_{H^2}
 \end{equation}
 for all $-N+1\leq j\leq N$, $\tau\in[0,1],$ and  $h\in H^3(\R)$.
Furthermore, Proposition \ref{PPK1} implies
  \begin{equation}\label{DEK5}
    \kappa_0\|(\lambda-\bA_{j,\tau})[\pi^\e_jh]\|_{2}\geq |\lambda|\cdot\|\pi^\e_jh\|_{2}+ \|\pi^\e_j h\|_{H^3}
 \end{equation}
 for all $-N+1\leq j\leq N$, $\tau\in[0,1],$ $\lambda\in\C$ with $\re \lambda\geq 1$, and  $h\in H^3(\R)$.
 Combining \eqref{DEK4}-\eqref{DEK5}, we find
 \begin{align*}
   \kappa_0\|\pi_j^\e(\lambda-\Phi_{\sigma,1}(\tau f))[h]\|_{2}\geq& \kappa_0\|(\lambda-\bA_{j,\tau})[\pi^\e_jh]\|_{2}-\kappa_0\|\pi_j^\e\Phi_{\sigma,1}(\tau f)[h]-\bA_{j,\tau}[\pi^\e_j h]\|_{2}\\[1ex]
   \geq& |\lambda|\cdot\|\pi^\e_jh\|_{2}+ \frac{1}{2}\|\pi^\e_j h\|_{H^3}-\kappa_0K\|  h\|_{H^2}
 \end{align*}
for all $-N+1\leq j\leq N$, $\tau\in[0,1],$ $\lambda\in\C$ with $\re \lambda\geq 1$, and  $h\in H^3(\R)$.
 Summing up  over $j\in\{-N+1,\ldots, N\},$  we infer from Lemma \ref{L:EN2} that there exists a constant  $C\geq1$   with the property that
  \begin{align*}
   C\|  h\|_{H^2}+C\|(\lambda-\Phi(\tau f))[h]\|_{2}\geq |\lambda|\cdot\|h\|_{2}+ \| h\|_{H^3}
 \end{align*}
for all   $\tau\in[0,1],$ $\lambda\in\C$ with $\re \lambda\geq 1$, and  $h\in H^3(\R)$.
Using \eqref{IP} and Young's inequality, it follows that  there exist constants  $\kappa\geq1$ and $\omega> 0$ with the property that
 \begin{align}\label{KDED}
   \kappa\|(\lambda-\Phi_{\sigma,1}(\tau f))[h]\|_{2}\geq |\lambda|\cdot\|h\|_{2}+ \| h\|_{H^3}
 \end{align}
for  all  $\tau\in[0,1],$ $\lambda\in\C$ with $\re \lambda\geq \omega$,    $h\in H^3(\R)$.
Since $(\omega-\Phi_{\sigma,1}(\tau f))\big|_{\tau=0}\in {\rm Isom}(H^3(\R), L_2(\R)),$
 the method of continuity together with \eqref{KDED} yields that 
\begin{align}\label{DEDK2}
   \omega-\Phi_{\sigma,1}(f)\in {\rm Isom}(H^3(\R), L_2(\R)).
 \end{align}
 The claim follows from  \eqref{KDED} (with $\tau=1$),  \eqref{DEDK2}, and \cite[Proposition 2.4.1 and Corollary 2.1.3]{L95}.
\end{proof}\medskip

We are now in a position to prove the well-posedness for the Muskat problem with surface tension.
\begin{proof}[Proof of Theorem \ref{MT1K}]
 Let $s\in(2,3)$, $\ov s=2$, $1>\alpha:=s/3>\beta:=2/3>0$.
Combining \eqref{reg} and Theorem \ref{TK2}, it follows that
\[
-\Phi_\sigma\in C^{\omega}(H^{2}(\R),\kH(H^3(\R), L_2(\R))).
\]
Since
 \[
  H^{2}(\R)=[L_2(\R), H^3(\R)]_{\beta} \qquad\text{and}\qquad H^s(\R)=[L_2(\R), H^3(\R)]_{\alpha},
 \]
we  now infer from Theorem \ref{T:A}
that \eqref{P} (or equivalently \eqref{Pest1}) possesses a maximally defined solution
 \begin{equation*} 
 f:=f(\cdot; f_0)\in C([0,T_+(f_0)),H^s(\R))\cap C((0,T_+(f_0)), H^3(\R))\cap C^1((0,T_+(f_0)), L_2(\R))
 \end{equation*}
with \begin{equation*} 
 f\in C^{(s-2)/3}([0,T],H^{2}(\R)) \qquad\text{for all $T<T_+(f_0)$}.
 \end{equation*}

 Concerning  the uniqueness of solutions, we next show  that any classical solution 
 \begin{equation*} 
 \wt f\in C([0,\wt  T),H^s(\R))\cap C((0,\wt T), H^3(\R))\cap C^1((0,\wt T)), L_2(\R)),\qquad \wt T\in(0,\infty],
 \end{equation*}  
to \eqref{Pest1}   satisfies 
  \begin{equation} \label{T:EEE2}
 \wt f\in C^\eta([0,T],H^{2}(\R))\qquad \text{for all $T\in(0,\wt T)$,}
 \end{equation}
where $\eta:=(s-2)/(s+1).$ 
To this end, we recall that
\begin{align}\label{PPP}
 \Phi_\sigma(f)[f]=f'B_{1,1}(f,f)[(\kappa(f))']+B_{0,1}(f)[(\kappa (f))']+\Theta A_{0,0}(f)[f']\qquad\text{for $f\in H^3(\R)$.}
\end{align}
Let $T\in(0,\wt T)$ be fixed. 
 Lemma \ref{L21} $(i)$ implies that
 \begin{align}\label{L:BV1}
 \sup_{ [0,T]} \|A_{0,0}(\wt f)[\wt f']\|_{2}\leq C.
 \end{align}
 
We now consider the highest order terms in \eqref{PPP}.  
Arguing as in Lemma \ref{L31}, it follows  form Remark \ref{R2} that $B_{0,1}(f)[\kappa (f)],\, B_{1,1}(f,f)[\kappa(f)]\in H^1(\R)$ for all  $f\in H^3(\R),$ with
 \begin{align*}
  (B_{0,1}(f)[\kappa (f)])'=&B_{0,1}(f)[(\kappa (f))']-2B_{2,2}(f',f,f,f)[\kappa(f)]\\[1ex]
  (B_{1,1}(f,f)[\kappa (f)])'=&B_{1,1}(f,f)[(\kappa (f))']+B_{1,1}(f',f)[\kappa (f)]-2B_{3,2}(f',f,f,f,f)[\kappa(f)].
 \end{align*}
Furthermore, given $t\in(0,T]$ and $\varphi\in H^1(\R)$, integration by parts together with $\wt f\in C([0,T], H^s(\R))$ leads us to
\begin{align*}
\Big| \int_\R \wt f'(t)(B_{1,1}(\wt f(t),\wt f(t))[\kappa (\wt f(t))])'\varphi\, dx\Big|=&\Big| \int_\R \wt f''(t)B_{1,1}(\wt f(t),\wt f(t))[\kappa (\wt f(t))]\varphi\, dx\Big|\\[1ex]
&+\Big| \int_\R \wt f'(t) B_{1,1}(\wt f(t),\wt f(t))[\kappa (\wt f(t))] \varphi'\, dx\Big|\\[1ex]
\leq& C\|\varphi\|_{H^1},
\end{align*}
so  that 
\begin{align}\label{L:BV3}
 \sup_{ (0,T]} \|\wt f'(B_{1,1}(\wt f,\wt f)[\kappa (\wt f)])'\|_{H^{-1}}\leq C,
 \end{align}
and similarly
\begin{align}\label{L:BV4}
 \sup_{ (0,T]} \|(B_{0,1}(\wt f)[\kappa (\wt f)])'\|_{H^{-1}}\leq C.
 \end{align}
 
We now estimate the term $ f'B_{1,1}(f',f)[\kappa (f)]$ with $f\in H^3(\R)$ in the $H^{-1}(\R)$-norm.
To this end, we rely on the formula
\[
B_{1,1}(f',f)[\kappa (f)]=T_1(f)-T_2(f)-T_3(f), 
\]
 where
 \begin{align*}
  T_1(f)(x)&:=\int_0^\infty\frac{\kappa(f)(x-y)-\kappa(f)(x+y)}{y}\frac{f'(x)-f'(x-y)}{y}\frac{1}{1+\big(\delta_{[x,y]}f/y\big)^2}\, dy,\\[1ex]
  T_2(f)(x)&:=\int_0^\infty\frac{\kappa(f)(x+y)}{y}\frac{f'(x+y)-2f'(x)+f'(x-y)}{y}\frac{1}{1+\big(\delta_{[x,y]}f/y\big)^2}\, dy,\\[1ex]
  T_3(f)(x)&:=\int_0^\infty\frac{\kappa(f)(x+y)}{y}\frac{f'(x)-f'(x+y)}{y}\frac{1}{\big[1+\big(\delta_{[x,y]}f/y\big)^2\big]\big[1+\big(\delta_{[x,-y]}f/y\big)^2\big]}\\[1ex]
  &\hspace{1.5cm}\times\frac{f(x+y)-f(x-y)}{y}\frac{f(x+y)-2f(x)+f(x-y)}{y}\, dy.
 \end{align*}
We estimate the terms $\wt f'T_i(\wt f),$ $1\leq i\leq 3$, separately.
Given $t\in(0,T]$ and $\varphi\in H^{1}(\R)$, we compute
\begin{align*}
 &\Big| \int_\R \wt f'(t)T_1(\wt f(t))\varphi\, dx\Big|\\[1ex]
 &\hspace{1cm}\leq C\|\varphi\|_\infty\int_0^\infty\int_\R\frac{|\kappa(\wt f(t))(x-y)-\kappa(\wt f(t))(x+y)|}{y}\frac{|\wt f'(t,x)-\wt f'(t,x-y)|}{y}\, dx\, dy\\[1ex]
 &\hspace{1cm}\leq C\|\varphi\|_\infty\int_0^\infty \frac{1}{y^{2}}\Big(\int_\R |\kappa(\wt f(t))(x-y)-\kappa(\wt f(t))(x+y)|^2 \, dx\Big)^{1/2}\\[1ex]
 &\hspace{4.25cm}\times\Big(\int_\R |\wt f'(t,x)-\wt f'(t,x-y)|^2 \, dx\Big)^{1/2}\, dy\\[1ex]
 &\hspace{1cm}= C\|\varphi\|_\infty\int_0^\infty \frac{1}{y^{2}}\Big(\int_\R |\kF (\kappa(\wt f(t)))|^2(\xi)|e^{i2\xi y}-1|^2 \, d\xi\Big)^{1/2}\Big(\int_\R |\kF(\wt f'(t))|^2(\xi)|e^{iy\xi}-1|^2 \, d\xi\Big)^{1/2}\, dy,
\end{align*}
and since 
\[\begin{aligned}
&|e^{iy\xi}-1|^2\leq  C(1+|\xi|^2)\big[y^{2} {\bf 1}_{(0,1)}(y)+ {\bf 1}_{[y\geq1]}(y)\big],\\[1ex]
&|e^{i2y\xi}-1|^2\leq  C(1+|\xi|^2)^{s-2}\big[y^{2(s-2)} {\bf 1}_{(0,1)}(y)+ {\bf 1}_{[y\geq1]}(y)\big],
\end{aligned}\qquad y>0,\,\xi\in\R,\\[1ex]
\]
it follows that 
\begin{align}
 \Big| \int_\R \wt f'(t)T_1(\wt f(t))\varphi\, dx\Big|\leq& C\|\varphi\|_\infty \|\kappa(\wt f(t))|_{H^{s-2}}\|\wt f(t)\|_{H^1}
 \int_0^\infty y^{s-3} {\bf 1}_{(0,1)}(y)+ y^{-2}{\bf 1}_{[y\geq1]}(y)\, dy\nonumber\\[1ex]
 \leq& C\|\varphi\|_{H^1}.\label{EHH1}
\end{align}
To bound the curvature term in the $H^{s-2}(\R)$-norm we have use the inequality
 \[
 \|\kappa (f)\|_{H^{s-2}}\leq C\big\|(1+f'^2)^{-3/2}\big\|_{BC^{s-3/2}}\|f\|_{H^s}\qquad\text{for all $f\in H^s(\R)$.}
 \]
 
 Similarly we have 
 \begin{align}
 &\Big| \int_\R \wt f'(t)T_2(\wt f(t))\varphi\, dx\Big|\nonumber\\[1ex]
 &\hspace{1cm}\leq C\|\varphi\|_\infty\int_0^\infty \frac{1}{y^{2}}\Big(\int_\R |\kappa(\wt f(t))(x+y)|^2 \, dx\Big)^{1/2}\nonumber\\[1ex]
 &\hspace{4.25cm}\times\Big(\int_\R |\wt f'(t,x+y)-2\wt f'(t,x)+\wt f'(t,x-y)|^2 \, dx\Big)^{1/2}\, dy\nonumber\\[1ex]
 &\hspace{1cm}\leq C\|\varphi\|_\infty\int_0^\infty \frac{1}{y^{2}} \Big(\int_\R |\kF(\wt f'(t))|^2(\xi)|e^{iy\xi}-2+e^{-iy\xi}|^2 \, d\xi\Big)^{1/2}\, dy\nonumber\\[1ex]
 &\hspace{1cm}\leq C\|\varphi\|_\infty\int_0^\infty y^{ s-3} {\bf 1}_{(0,1)}(y)+ y^{-2}{\bf 1}_{[y\geq1]}(y)\, dy\nonumber\\[1ex]
 &\hspace{1cm}\leq C\|\varphi\|_{H^1}\label{EHH2}
\end{align}
by virtue of
 \[
 |e^{iy\xi}-2+e^{-iy\xi}|^2\leq  C(1+|\xi|^2)^{s-1}\big[y^{2(s-1)} {\bf 1}_{(0,1)}(y)+ {\bf 1}_{[y\geq1]}(y)\big],\qquad y>0,\,\xi\in\R.
\]
Lastly, since $H^{s-1}(\R)\hookrightarrow {\rm BC}^{s-3/2}(\R)$ for $s\neq 5/2$ (the estimate \eqref{EHH3} holds though also for $s=5/2$) and $H^{s}(\R)\hookrightarrow {\rm BC}^{1}(\R)$  the inequality
 \[
 |e^{iy\xi}-2+e^{-iy\xi}|^2\leq  C(1+|\xi|^2)^{2}\big[y^{4} {\bf 1}_{(0,1)}(y)+ {\bf 1}_{[y\geq1]}(y)\big],\qquad y>0,\,\xi\in\R,
\]
leads us to
\begin{align}
 &\Big| \int_\R \wt f'(t)T_3(\wt f(t))\varphi\, dx\Big|\nonumber\\[1ex]
 &\hspace{1cm}\leq C\|\varphi\|_\infty\int_0^\infty \frac{y^{\min\{1,s-3/2\}}}{y^{3}} \Big(\int_\R |\kF(\wt f(t))|^2(\xi)|e^{iy\xi}-2+e^{-iy\xi}|^2 \, d\xi\Big)^{1/2}\, dy\nonumber\\[1ex]
 &\hspace{1cm}\leq C\|\varphi\|_\infty\int_0^\infty  y^{\min\{0,s-5/2\}}  {\bf 1}_{(0,1)}(y)+ y^{-2}{\bf 1}_{[y\geq1]}(y)\, dy\nonumber\\[1ex]
 &\hspace{1cm}\leq C\|\varphi\|_{H^1}.\label{EHH3}
\end{align}
 Gathering \eqref{EHH1}-\eqref{EHH3}, we conclude that
 \begin{align}\label{L:BV5}
 \sup_{ (0,T]} \|\wt f'B_{1,1}(\wt f',\wt f)[\kappa (\wt f)]\|_{H^{-1}}\leq C,
 \end{align}
 and similarly we obtain
 \begin{align}\label{L:BV6}
 \sup_{ (0,T]} \big[\|\wt f'B_{3,2} (\wt f',\wt f, \wt f,\wt f,\wt f)[\kappa (\wt f)]\|_{H^{-1}}+\|B_{2,2} (\wt f', \wt f,\wt f,\wt f)[\kappa (\wt f)]\|_{H^{-1}}\big]\leq C.
 \end{align}
Combining \eqref{L:BV1}-\eqref{L:BV4}, \eqref{L:BV5}, and \eqref{L:BV6}, it follows that $\wt f\in {\rm BC}^1((0,T], H^{-1}(\R)).$
Recalling that $\eta=(s-2)/(s+1)$, \eqref{IP} together with the mean value theorem yields 
 \begin{align*}
  \|\wt f(t)-\wt f(s)\|_{H^{2}}\leq  \|\wt f(t)-\wt f(s)\|_{H^{-1}}^{\eta}\|\wt f(t)-\wt f(s)\|_{H^s}^{1-\eta}\leq C|t-s|^{\eta},\qquad t,s\in[0,T],
 \end{align*}
which proves \eqref{T:EEE2} and the uniqueness claim in Theorem \ref{MT1K}. 

Finally, let us assume that $T_+(f_0)<\infty$ and that
\[
\sup_{[0,T_+(f_0))}\|f(t)\|_{H^s}<\infty,
\]
Arguing as above, we find that 
\begin{align*}
  \|f(t)-f(s)\|_{H^{(s+2)/2}}\leq C|t-s|^{(s-2)/(2s+2)},\qquad t,s\in[0,T_+(f_0)).
 \end{align*} 
The criterion for global existence in Theorem \ref{T:A} applied for $ \alpha:=(s+2)/6$ and $ \beta:=2/3$ implies that
the solution can be continued on an interval $[0,\tau)$ with $\tau>T_+(f_0)$ and that
\[
 f\in C^{(s-2)/6}([0,T],H^{2}(\R))\qquad \text{for all $T\in(0,\tau)$.}
\]
The uniqueness claim in Theorem \ref{T:A} leads us to a contradiction. 
Hence our assumption was false and $T_+(f_0)=\infty$.

\end{proof}

\appendix
\section{Some technical results}\label{S:A}

 The following lemma is used in the proof of Theorem \ref{T1}.
\begin{lemma}\label{LA1}
 Given $f\in H^s(\R),$ $ s\in(3/2,2),$ and $\tau\in[0,1]$, let 
 $a_{\tau}:\R\to\R$ be defined by
  \begin{align*}
a_{\tau}(x) :=&\PV\int_\R \frac{y}{y^2+\tau^2(\delta_{[x,y]}f )^2}\, dy,\qquad x\in\R.
\end{align*}
Let further $\alpha:=s/2-3/4\in(0,1).$
Then, $a_\tau\in{\rm BC}^{\alpha}(\R)\cap C_0(\R)$, 
\begin{align}\label{UET1}
 \sup_{\tau\in[0,1]}\|a_\tau \|_{{\rm BC}^{\alpha}}<\infty,
 \end{align}
and, given $\e_0>0$, there exists $\eta>0$ such that 
\begin{align}\label{UET2}
 \sup_{\tau\in[0,1]}\sup_{|x|\geq\eta}|a_\tau(x)|\leq \e_0.
\end{align}
\end{lemma}

\begin{proof}
It holds that
  \begin{align*}
a_{\tau}(x) 
=&\tau^2\lim_{\delta \to 0 }\int_\delta^{1/\delta} \frac{f(x+y)-2f(x)+f(x-y)}{y^{2}}\frac{f(x+y)-f(x-y)}{y}\frac{y^4}{\Pi(x,y)}\, dy,\qquad x\in\R,
\end{align*}
 with
\[\Pi(x,y):= [y^2+\tau^2(\delta_{[x,-y]}f)^2][y^2+\tau^2(\delta_{[x,y]}f)^2].\]
 Letting
\[
I(x,y):=\tau^2\frac{f(x+y)-2f(x)+f(x-y)}{y^{2}}\frac{f(x+y)-f(x-y)}{y}\frac{y^4}{\Pi(x,y)}, \qquad\text{$(x,y)\in\R\times (0,\infty)$},
\]
it follows that
  \begin{align}\label{SE1}
   |I(x,y)|\leq 8\Big(\|f\|_\infty ^2 \frac{1}{y^3}{\bf 1}_{[1,\infty)}(y) +\|f'\|_\infty[f']_{s-3/2}  \frac{1}{y^{5/2-s}}{\bf 1}_{(0,1)}(y)\Big),\qquad\text{$(x,y)\in\R\times (0,\infty)$.}
\end{align}
The latter estimate was obtained by using the fact that  $f\in {\rm BC}^{s-1/2}(\R)$, $s-1/2\in (1,2)$, together with \eqref{MM}.
Hence,
  \begin{align*}
a_\tau(x)= \int_0^{\infty} I(x,y)\, dy,\qquad x\in\R,
\end{align*}
and $\sup_{\tau\in[0,1]}\|a_\tau \|_{\infty}<\infty$.
To estimate the H\"older seminorm of $a_\tau$, we compute for $x,x'\in\R$ that
 \begin{align}\label{Tip0}
|a_\tau(x)-a_\tau(x')|\leq \int_0^{\infty} |I(x,y)-I(x',y)|\, dy\leq T_1+T_2+T_3,
\end{align}
where
\begin{align*}
 T_1:=&\int_0^\infty \frac{|f(x+y)-2f(x)+f(x-y)|}{y^{2}}\frac{|[f(x+y)-f(x-y)]-[f(x'+y)-f(x'-y)]|}{y}\frac{y^4}{\Pi(x,y)}\, dy,\\[1ex]
  T_2:=&\int_0^\infty \frac{|[f(x+y)-2f(x)+f(x-y)]-[f(x'+y)-2f(x')+f(x'-y)]|}{y^{2}}\\[1ex]
  &\hspace{1cm}\times\frac{| f(x'+y)-f(x'-y)|}{y}\frac{y^4}{\Pi(x,y)}\, dy,\\[1ex]
   T_3:=&\int_0^\infty \frac{|f(x'+y)-2f(x')+f(x'-y)|}{y^{2}}\frac{|f(x'+y)-f(x'-y)|}{y}\frac{\big|\Pi(x,y)-\Pi(x',y)\big|}{y^4}\, dy.
\end{align*}
Using the mean value theorem, we have
\begin{align*}
 \frac{|[f(x+y)-f(x-y)]-[f(x'+y)-f(x'-y)]|}{y}\leq&2\int_0^1|f'(x+(2\tau-1)y)-f'(x'+(2\tau-1)y)|\, d\tau\\[1ex]
 \leq &2[f']_{s-3/2}|x-x'|^{s-3/2},\qquad y>0,
\end{align*}
and, similarly as above, we find that
\begin{align}\label{Tip1}
|T_1|\leq C\|f\|_{H^s}^2|x-x'|^{2\alpha}.
\end{align}

To deal with the second term we appeal to the formula
\begin{align*}
 f(x+y)-2f(x)+f(x-y)=&y[f'(x+y)-f'(x-y)]+y\int_0^{1}f'(x+\tau y)-f'(x+y)\, d\tau\\[1ex]
 &-y\int_{0}^1f'(x-\tau y)-f'(x-y)\, d\tau\qquad\text{for $x,$ $y\in\R$,}
\end{align*}
and we get
\begin{align*}
 \frac{|[f(x+y)-2f(x)+f(x-y)]-[f(x'+y)-2f(x')+f(x'-y)]|}{y^{2}}\leq T_{2a}+T_{2b}+T_{2c},
\end{align*}
where
\begin{align*}
 T_{2a}:=&\frac{| [f'(x+y)-f'(x-y)]-[f'(x'+y)-f'(x'-y)]|}{y}\\[1ex]
 \leq& 2[f']_{2\alpha}\Big(\frac{1}{y}{\bf 1}_{[1,\infty)}(y)|x-x'|^{2\alpha}+2\frac{1}{y^{1-\alpha}}{\bf 1}_{(0,1)}(y)|x-x'|^\alpha\Big), \\[1ex]
 T_{2b}:=&\frac{1}{y} \int_0^{1}\big|[f'(x+\tau y)-f'(x+y)]-[f'(x'+\tau y)-f'(x'+y)]\big|\, d\tau \\[1ex]
 \leq & 2[f']_{2\alpha}\Big(\frac{1}{y}{\bf 1}_{[1,\infty)}(y)|x-x'|^{2\alpha}+\frac{1}{y^{1-\alpha}}{\bf 1}_{(0,1)}(y)|x-x'|^\alpha\Big),\end{align*}
\begin{align*}
 T_{2c}:=&\frac{1}{y} \int_0^{1}\big|[f'(x-\tau y)-f'(x-y)]-[f'(x'-\tau y)-f'(x'-y)]\big|\, d\tau  \\[1ex]
 \leq & 2[f']_{2\alpha}\Big(\frac{1}{y}{\bf 1}_{[1,\infty)}(y)|x-x'|^{2\alpha}+\frac{1}{y^{1-\alpha}}{\bf 1}_{(0,1)}(y)|x-x'|^\alpha\Big),
\end{align*}
and therewith
\begin{align}\label{Tip2}
|T_2|\leq C\|f\|_{H^s}^2(|x-x'|^\alpha+|x-x'|^{2\alpha}).
\end{align}
Finally, since
\begin{align*}
  \frac{\big|\Pi(x,y)-\Pi(x',y)\big|}{y^4}\leq 4\|f'\|_\infty(1+\|f'\|_\infty^2)[f']_{2\alpha}|x-x'|^{2\alpha},
\end{align*}
we infer from \eqref{SE1} that 
\begin{align}\label{Tip3}
|T_3|\leq C\|f\|_{H^s}^4(1+\|f\|_{H^s}^2)(|x-x'|^\alpha+|x-x'|^{2\alpha}).
\end{align}
The relation \eqref{UET1} is a simple consequence of \eqref{Tip0}-\eqref{Tip3} and of     $\sup_{\tau\in[0,1]}\|a_\tau \|_{\infty}<\infty$.

To prove that $a_\tau$ vanishes at infinity let $\e_0>0$ be arbitrary. 
We write
\[
a_\tau(x)=\int_0^MI(x,y)\, dy+\int_M^\infty I(x,y)\, dy,\qquad x\in\R, 
\]
for some $M>1$ with
\[
\frac{4\|f\|_\infty^2}{M^2}\leq \frac{ \e_0}{2}.
\]
Recalling \eqref{SE1},  it follows that for all $x\in\R$ we have 
\[
\int_M^\infty |I(x,y)|\, dy\leq 8\|f\|_\infty^2\int_M^\infty\frac{1}{y^3}\, dy= \frac{4\|f\|_\infty^2}{M^2}\leq\frac{\e_0}{2}. 
\]
Let $\beta\in(0,1)$ be chosen such that $3/2+\beta<s$.  Since $f\in C_0(\R),$ there exists $\eta>M$ with  
$$|f(y)|\leq \Big[\frac{(s-3/2-\beta)\e_0 M^{3/2+\beta-s}}{32([f']_{s-3/2}\|f'\|_\infty^{1-\beta}+1)}\Big]^{1/\beta} \qquad\text{for all $|y|\geq \eta-M.$}$$ 
Using this inequality, it follows   that for all $|x|\geq \eta$ we have
\begin{align*}
|a_\tau(x)|\leq \int_0^M|I(x,y)|\, dy+ \frac{\e_0}{2}\leq\frac{\e_0}{2} +8[f']_{s-3/2}\|f'\|_\infty^{1-\beta}\int_0^M\frac{|f(x+y)-f(x-y)|^{\beta}}{y^{5/2+\beta-s}}\, dy\leq \e_0,
\end{align*}
hence $a_\tau\in C_0(\R) $ and \eqref{UET2} holds true. 
\end{proof}

The next result is used in Proposition \ref{PE}.

\begin{lemma}\label{L:A2} Given $f\in H^5(\R)$ and $\ov\omega\in H^2(\R)$, set
 \begin{equation}\label{V1'}
 \wt v(x,y):=\frac{1}{2\pi} \int_\R\frac{(-(y-f(s),x-s)}{ (x-s)^2+(y-f(s))^2 }\ov\omega(s)\, ds\qquad \text{in $\R^2\setminus[y=f(x)]$.}
\end{equation}
Let further  $\0_-:=[y<f(x)]$, $\0_+:=[y>f(x)]$, and $\wt v_\pm:=\wt v\big|_{\0_\pm}$. Then,     $\wt v_\pm\in C (\ov{\0_\pm})\cap C^1({\0_\pm})$ and
 \begin{equation}\label{LA2}
\wt v_\pm(x,y)\to0 \qquad\text{for \, $|(x,y)|\to\infty$}.
\end{equation} 
\end{lemma}
\begin{proof}
 It is easy to see that $\wt v_\pm  \in C^1({\0_\pm}).$ Plemelj's formula further shows that $\wt v_\pm \in  C (\ov{\0_\pm})$ and
\begin{align*}
\wt v_\pm(x,f(x))=\frac{1}{2\pi}\PV\int_\R\frac{(-(f(x)-f(x-s)),s)}{ s^2+(f(x)-f(x-s))^2 }\ov\omega(x-s)\, ds \mp\frac{1}{2}\frac{(1,f'(x))\ov\omega(x)}{1+{f'}^2(x)},\qquad x\in\R,
\end{align*}
 or equivalently, with the notation in Remark \ref{R2}, 
 \[
 \wt v_\pm|_{[y=f(x)]}=\pm\frac{1}{2}\frac{(1,f' )\ov\omega }{1+{f'}^2 } -\frac{1}{2\pi}B_{1,1}(f,f)[\ov\omega]+\frac{i}{2\pi}B_{0,1}(f)[\ov\omega].
 \]
 
 Recalling that  $f\in H^5(\R)$ and $\ov\omega\in H^2(\R)$, the arguments in the proof of Lemma \ref{L31} show that $B_{1,1}(f,f)[\ov\omega] $ and  $ B_{0,1}(f)[\ov\omega]$ belong to $ H^1(\R)$,
 thus    
 \begin{equation}\label{s1}
 \wt v_\pm(x,f(x))\to0\qquad\text{for $|x|\to\infty.$}
\end{equation} 
 Furthermore,  since $f$ and $\ov\omega$ vanish at infinity, we find, similarly as in the proof of \eqref{UET2},  that
 \begin{equation}\label{s2}
 \sup_{[ y\geq n]}|\wt v_+| +\sup_{[ y\leq -n]}|\wt v_-|\to 0
 \qquad\text{for  $n\to\infty$}
\end{equation} 
and that, for arbitrary $0<a<b$, 
\begin{equation}\label{s3}
 \sup_ {[a\leq y\leq b]\cap  [|x|\geq n]}| \wt v_+|+\sup_{[-b\leq y\leq -a]\cap [ |x|\geq n]}| \wt v_-|\to 0
\qquad\text{for  $n\to \infty$}.
\end{equation}

Finally, arguing along the lines of the proof of Privalov's theorem, cf. \cite[Theorem 3.1.1]{JKL93} (the lengthy details, which for $\ov\omega\in W^1_\infty(\R)$ are simpler than in \cite{JKL93}, are left to the interesting reader), 
it follows that there exists a constant $C$, which depends only on $f$ and $\ov\omega,$ such 
that for each $z=(x,y)\in  \R^2\setminus[y=f(x)]$  with 
$y\in [-\|f\|_\infty-1,\|f\|_\infty+1]$ the following inequalities hold
\begin{equation}\label{s4}
\begin{aligned}
 & |\wt v_+(z)-\wt v_+(x,f(x))|\leq C|y-f(x)|^{1/2} \qquad\text{if $y>f(x)$,}\\[1ex]
 & |\wt v_-(z)-\wt v_-(x,f(x))|\leq C|y-f(x)|^{1/2} \qquad\text{if $y<f(x)$.}
\end{aligned}
\end{equation}
The relation \eqref{LA2} is an obvious consequence of \eqref{s1}-\eqref{s4}.
\end{proof}


\section{The proof of Theorem \ref{T:A}}\label{S:B}
This section is dedicated to the proof of Theorem \ref{T:A}.
In the following $\E_0$ and $\E_1$ denote complex Banach spaces\footnote{The proof of Theorem \ref{T:A} in the  context
of real Banach spaces is identical.} and we assume that the embedding  $\E_1\hookrightarrow\E_0$ is dense.
In view of \cite[Theorem I.1.2.2]{Am95}, we may represent the set $\kH(\E_1,\E_0)$ of negative analytic generators
as  
\[
\kH(\E_1,\E_0) =\underset{\tiny\begin{array}{lll}\kappa\geq1\\[-0.1ex] \omega>0\end{array}}\bigcup\kH(\E_1,\E_0,\kappa,\omega),
\]
where, given $\kappa\geq 1$ and $\omega>0$,  the class  $\kH(\E_1,\E_0,\kappa,\omega)$ consists of the operators $\bA\in\kL(\E_1,\E_0)$ having the properties that
\begin{align*}
\bullet &\quad \text{$\omega+\bA\in{\rm Isom}(\E_1,\E_0),$ and}\\[1ex]
\bullet &\quad \text{$\kappa^{-1}\leq \frac{\|(\lambda+\bA)x\|_0}{|\lambda|\cdot\|x\|_0+\|x\|_1}\leq\kappa$\quad for all $0\neq x\in\E_1 $ and all $ \re\lambda\geq\omega.$}
\end{align*}
Given $\bA\in\kH(\E_1,\E_0,\kappa,\omega)$ and $r\in(0, \kappa^{-1})$, it follows from \cite[Theorem I.1.3.1 $(i)$]{Am95}  that 
  \begin{align}\label{OPEN}
  \bA+B\in\kH(\E_1,\E_0,\kappa/(1-\kappa r),\omega) \qquad\text{for all $  \|B\|_{\kL(\E_1,\E_0)}\leq r.$}
  \end{align}
This property shows in particular that    $\kH(\E_1,\E_0)$ is an open subset of $\kL(\E_1,\E_0)$.

The proof of Theorem \ref{T:A} uses to a large extent the powerful theory of parabolic evolution operators developed in \cite{Am95}.
The following result is a direct consequence of \cite[Theorem II.5.1.1, Lemma  II.5.1.3,  Lemma II.5.1.4]{Am95}.

\begin{prop}\label{T:Aa} Let $T>0$, $\rho\in(0,1) $, $L \geq0$, $\kappa\geq 1$, and $\omega>0 $ be given constants.
 Moreover, let $\mathcal{A}\subset  C^\rho([0,T], \kH(\E_1,\E_0))$  be a family satisfying  
\begin{align*}
\bullet &\quad \text{$[\bA]_{\rho,[0,T]}:=\sup_{t\neq s\in[0,T]}\frac{\|\bA(t)-\bA(s)\|}{|t-s|^\rho}\leq L $ for all $\bA\in\mathcal{A},$ and}\\[1ex]
\bullet &\quad \text{$\bA(t)\in \kH(\E_1,\E_0,\kappa,\omega)$ for all $\bA\in\mathcal{A}$ and $t\in[0,T]$.}
\end{align*}
Then, given  $\bA\in\A,$ there exists a unique parabolic evolution
operator\footnote{In the sense of \cite[Section II]{Am95}.} $U_\bA$  for $\bA$ possessing $\E_1$ as a regularity subspace.
Moreover, the following hold.
\begin{itemize}
\item[$(i)$] There exists a constant $C>0$ such that 
\begin{align}\label{EOP1}
\|U_\bA(t,s)\|_{\kL(\E_j)}+(t-s)\|U_\bA(t,s)\|_{\kL(\E_0,\E_1)}\leq C
\end{align}
for all $(t,s)\in \Delta_T^*:=\{(t,s)\in[0,T]^2\,:\, 0\leq s<t\leq T\}$, $j\in\{0,\,1\}$, and  all $\bA\in\A$.\\[-1ex]
\item[$(ii)$] Let $\Delta_T:=\{(t,s)\in[0,T]^2\,:\, 0\leq s\leq t\leq T\}$ and $0\leq \beta\leq \alpha\leq 1$. Then, given $x\in\E_\alpha$, it holds that
  $U_\bA(\,\cdot\,,\,\cdot\,)x\in C(\Delta_T, \E_\alpha) $. Moreover, $U_\bA\in C(\Delta_T^*,\kL(\E_\beta,\E_\alpha))$, and there exists a constant $C>0$ such that 
\begin{align}\label{EOP2}
 (t-s)^{\alpha-\beta}\|U_\bA(t,s)\|_{\kL(\E_\beta,\E_\alpha)}\leq C
\end{align}
for all $(t,s)\in \Delta_T^* $ and  all $\bA\in\A$.\\[-1ex]

\item[$(iii)$] Given $0\leq \beta< 1$ and $0<\alpha\leq 1$, there exists a constant $C>0$ such that
\begin{align}\label{EOP3}
 (t-s)^{\beta-\alpha}\|(U_\bA-U_\bB)(t,s)\|_{\kL(\E_\alpha,\E_\beta)}\leq C \max_{\tau\in[s,t]} \|\bA(\tau)-\bB(\tau)\|_{\kL(\E_1,\E_0)}
\end{align}
for all $(t,s)\in \Delta_T^* $ and  all $\bA,\, \bB\in\A$.
\end{itemize}
\end{prop}

Let now $\A$ be a family as in Proposition \ref{T:Aa}.
Given $\bA\in\A$ and $x\in\E_0$, we   consider the linear problem
\begin{align}\label{LP0x}
\dot u+\bA(t)u=0,\quad t\in(0,T], \qquad u(0)=x. 
\end{align}
Using the fundamental properties of the parabolic evolution
operator $U_\bA$ associated to $\bA$, it follows from \cite[Remark II.2.1.2]{Am95} that \eqref{LP0x} has a unique classical solution $u:=u(\,\cdot\,; x,\bA)$, that is
$$u:=u(\,\cdot\,; x,\bA)\in C^1((0,T],\E_0)\cap C((0,T],\E_1)\cap C([0,T],\E_0) $$
and $u$ satisfies the equation of \eqref{LP0x} pointwise.
This solution is given by the expression
\[u(t)=U_\bA(t,0)x,\qquad t\in[0,T].\] 
If $x\in\E_\alpha$ for some $\alpha\in(0,1)$, we may use  the relations \eqref{EOP1}-\eqref{EOP3} to derive additional regularity properties for the solution, as stated below.

\begin{prop}\label{T:Ab}
Let $\A$ be a family as in Proposition \ref{T:Aa}. The following hold true.
\begin{itemize}
\item[$(i)$]  Let $0\leq\beta\leq\alpha<1$ and $x\in\E_\alpha$. Then $u\in C^{\alpha-\beta}([0,T],\E_\beta) $ and  there exists $C>0$ such that
\begin{align}\label{EOP4}
 \|u(t)-u(s)\|_{\beta}\leq C (t-s)^{\alpha-\beta}\|x\|_\alpha
\end{align}
for all $(t,s)\in \Delta_T $, $x\in\E_\alpha$, and  $\bA \in\A$.\\[-1ex]
\item[$(ii)$] Let $0\leq \beta<\alpha\leq 1.$ Then, there exists $C>0$ such that
\begin{align}\label{EOP5}
 \|u(t;x,\bA)-u(t;x,\bB)\|_{\beta}\leq C  t^{\alpha-\beta}\max_{\tau\in[0,t]} \|\bA(\tau)-\bB(\tau)\|_{\kL(\E_1,\E_0)}\|x\|_\alpha 
\end{align}
for all  $t\in [0,T] $, $x\in\E_\alpha$, and  $\bA,\,\bB \in\A$.
\end{itemize}
\end{prop}
\begin{proof}
The claim $(i)$ follows from \cite[Theorem II.5.3.1]{Am95}, while $(ii)$ is a consequence of \cite[Theorem II.5.2.1]{Am95}.
\end{proof}

By means of a contraction  argument we   now obtain as a preliminary result the following (uniform) local existence theorem which stays at the basis of Theorem \ref{T:A}.

\begin{prop}\label{T:B1} 
Let the assumptions of Theorem \ref{T:A} be satisfied and let $\ov f\in\cO_\alpha:=\cO_\beta\cap\E_\alpha.$
Then, there exist constants $\delta=\delta(\ov f)>0$ and $r=r(\ov f)>0$ with the property that for all $f_0\in \cO_\alpha$ with $\|f_0-\ov f\|_\alpha\leq r$
the problem
\begin{equation}\label{QP'}
  \dot f=\Phi(f)[f],\quad t>0,\qquad f(0)=f_0.\tag{QP}\\[-0.1ex]
\end{equation}
possesses a classical solution 
 \[
 f\in {  C}([0,\delta],\cO_\alpha)\cap{  C}((0,\delta],\E_1)\cap {  C}^1((0,\delta],\E_0)\cap    {  C}^{ \alpha-\beta}([0,\delta],\E_\beta).
 \]
Moreover, if   $h$  is a further solution to \eqref{QP'} with
 \[
 h\in {  C}((0,\delta],\E_1)\cap {  C}^1((0,\delta],\E_0)\cap {  C}^\eta([0,\delta],\cO_\beta)\qquad\text{for some $\eta\in(0,\alpha-\beta]$,}
 \]
then $f\equiv h$.
\end{prop}
\begin{proof}
{\em Existence.} We first note that  $\cO_\alpha$   is an open subset of $\E_\alpha$ (see e.g. \cite[Section I.2.11]{Am95}).
Since by assumption  $-\Phi\in C^{1-}(\cO_\beta,\kH(\E_1,\E_0))$, it follows from \eqref{OPEN} there exist constant $R>0$,  $L>0$, $\kappa\geq 1$, and $\omega>0$ such that 
\begin{align}
&\text{$\|\Phi(f)-\Phi(g)\|_{\kL(\E_1,\E_0)}\leq L\|f-g\|_\beta$\quad for all $f,\, g\in \overline{\bB}_{\E_\beta}(\ov f,R)\subset\cO_\beta,$}\label{aa1}\\[1ex]
&\text{$-\Phi(f)\in\kH(\E_1,\E_0,\kappa,\omega)$ \quad for all  $ f\in \overline{\bB}_{\E_\beta}(\ov f, R)$}.\label{aa2}
\end{align}
Let $\rho\in(0,\alpha-\beta)$ be fixed. 
If $r>0$ is sufficiently small, it holds that   
\begin{align}\label{aa3}
  \overline{\bB}_{\E_\alpha}(\ov f, r)\subset  \overline{\bB}_{\E_\beta}(\ov f,R)\cap \cO_\alpha. 
\end{align}
Given $\delta>0$, $r>0$ such that \eqref{aa3} holds ($r$ and $\delta$ will be fixed later on) and $f_0\in   \overline{\bB}_{\E_\alpha}(\ov f, r)$, we define the set
 \begin{align*}
  \bM:=\big\{f\in   C \big([0,\delta], \overline{\bB}_{\E_\beta}(\ov f,R)\big)\,:\, \text{$f(0)=f_0$ and $\|f(t)-f(s)\|_\beta\leq |t-s|^\rho$ $\forall\, t,\,s\in[0,\delta]$}\big\}.
 \end{align*}
Since $\bM$ is a closed subset of $ C ([0,\delta],\E_\beta)$ it is also a (nonempty) complete metric space.
Given $f\in\bM,$ we define 
$$\bA_f(t):= -\Phi(f(t)),\qquad t\in[0,\delta].$$
As a direct consequence of \eqref{aa1} and of the definition of $\bM$ it follows that
\[
\|\bA_f(t)-\bA_f(s)\|_{\kL(\E_1,\E_0)}\leq L\|f(t)-f(s)\|_\beta\leq L|t-s|^\rho, \qquad t,\,s\in[0,\delta], 
\]
and \eqref{aa2} yields that $\bA_f(t)\in\kH(\E_1,\E_0,\kappa,\omega)$  for all $f\in\bM$ and all $t\in[0,\delta].$
Proposition \ref{T:Aa} ensures  the existence of a parabolic evolution operator $U_{\bA_f}$ for $\bA_f$.
Given $f\in\bM$, it is  natural to consider the linear evolution problem
\begin{align}\label{LP0f}
\dot g+\bA_f(t)g=0,\quad t\in(0,\delta], \qquad g(0)=f_0, 
\end{align}
which has, in view of Proposition \ref{T:Ab} a unique classical solution
$$g:=\Gamma(f):=U_{\bA_f}(\,\cdot\,,0)f_0\in C^{\alpha-\beta}([0,\delta],\E_\beta)\cap C([0,\delta],\E_\alpha).$$
The existence part of Proposition \ref{T:Aa} reduces  to proving that $\Gamma:\bM\to\bM$ is a strict contraction for suitable $r$ and $\delta$.
Clearly $\Gamma(f)(0)=f_0$. Moreover,  \eqref{EOP4} yields 
\[
\|\Gamma(f)(t)-\Gamma(f)(s)\|_\beta\leq C|t-s|^{\alpha-\beta}\|f_0\|_\alpha \leq  C\delta^{\alpha-\beta-\rho}(\|\ov f\|_\alpha+r)|t-s|^\rho\leq |t-s|^\rho,\quad \forall t,\, s\in[0,\delta],
\]
 provided that
\begin{equation}\label{cond1}
C\delta^{\alpha-\beta-\rho}(\|\ov f\|_\alpha+r)\leq 1.
\end{equation}
The latter estimate (with $s=0$) yields
\begin{align*}
\|\Gamma(f)(t)-\ov f\|_\beta\leq \|\Gamma(f)(t)-\Gamma(f)(0)\|_\beta+\|f_0-\ov f\|_\beta\leq \delta^\rho+r\|i_{\E_\alpha\hookrightarrow\E_\beta}\|_{\kL(\E_\alpha,\E_\beta)}\leq R,
\quad \forall t\in[0,\delta],
\end{align*}
if we  additionally require that
\begin{equation}\label{cond2}
\delta^\rho+r\|i_{\E_\alpha\hookrightarrow\E_\beta}\|_{\kL(\E_\alpha,\E_\beta)}\leq R.
\end{equation}
We now assume that $r$ and $\delta$ are chosen such that \eqref{cond1}-\eqref{cond2} hold true.  It then follows that $\Gamma:\bM\to\bM$ is a well-defined map.
Furthermore,  given $f,\, h\in\bM$, the estimate \eqref{EOP5} together with~\eqref{aa1}  yields 
\begin{align*}
\|\Gamma(f)(t)-\Gamma(h)(t)\|_\beta&=\|U_{\bA_f}(t,0)f_0-U_{\bA_h}(t,0)f_0\|_\beta\leq Ct^{\alpha-\beta}\max_{\tau\in[0,t]} \|\bA_f(\tau)-\bA_h(\tau)\|_{\kL(\E_1,\E_0)}\|f_0\|_\alpha\\[1ex]
&\leq C\delta^{\alpha-\beta}L(\|\ov f\|_\alpha+r)\max_{t\in[0,\delta]} \|f(t)-h(t)\|_{\beta}\\[1ex]
&\leq \frac{1}{2}\max_{t\in[0,\delta]} \|f(t)-h(t)\|_{\beta}, \quad \forall t\in[0,\delta],
\end{align*}
provided that
 \begin{equation}\label{cond3}
C\delta^{\alpha-\beta}L(\|\ov f\|_\alpha+r)\leq \frac{1}{2}.
\end{equation}
 Hence, if $r$ and $\delta$ are chosen such that also \eqref{cond3} is satisfied, then $\Gamma$ is a strict contraction and Banach's fixed point theorem ensures that $\Gamma$ has a fixed point.
 This proves the existence part.\medskip
 
 \noindent{\em Uniqueness.} Let $f$ be  a solution  to \eqref{QP'} as found above and let  $h\not\equiv f$ be a further classical solution such that
 $h\in C^\eta([0,\delta],\E_\beta) $ for some $\eta\in(0,\alpha-\beta]$. 
 The real number
 \[
t_0:=\max\{t\in[0,\delta]\,:\, f|_{[0,t]}=h|_{[0,t]}\} 
 \]
 satisfies $0\leq t_0<\delta$ and $f =h$ on $[0,t_0].$
 Since $f(t_0)\in\cO_\alpha$, there exist   $R>0$,  $L>0$, $\kappa\geq 1$, and $\omega>0$ such that 
\begin{align*}
&\text{$\|\Phi(u)-\Phi(v)\|_{\kL(\E_1,\E_0)}\leq L\|u-v\|_\beta$\quad for all $u,\, v\in \overline{\bB}_{\E_\beta}(f(t_0),R)\subset\cO_\beta,$} \\[1ex]
&\text{$-\Phi(u)\in\kH(\E_1,\E_0,\kappa,\omega)$ \quad for all  $ u\in \overline{\bB}_{\E_\beta}( f(t_0), R)$}. 
\end{align*}
   Given $ \delta_0\in(t_0,\delta]$, the set 
  \begin{align*}
  \bM_0:=\Big\{h\in   C \big([0,\delta_0-t_0], \overline{\bB}_{\E_\beta}(f(t_0),R)\big)\,:\, \text{$h(0)=f(t_0)$, $\frac{\|h(t)-h(s)\|_\beta}{|t-s|^{\eta/2}}\leq1 $ $\forall\,
   t\neq s\in[0,\delta_0-t_0]$}\Big\}
 \end{align*}
  is   a (nonempty) complete metric space.
  Letting $\bA_h(t):=-\Phi(h(t)) $ for $h\in\bM_0$ and $t \in[0,\delta_0-t_0]$, we may argue as in the existence part of this proof to conclude that the linear problem
  \begin{align*}
\dot u+\bA_h(t)u=0,\quad t\in(0,\delta_0-t_0], \qquad h(0)=f(t_0)
\end{align*}
has a unique classical  solution $\Gamma_0(h)\in C^{\alpha-\beta}([0,\delta_0-t_0],\E_\beta)\cap C([0,\delta_0-t_0],\E_\alpha)$. 
Furthermore,  $\Gamma_0:\bM_0\to\bM_0 $ is a $1/2$-contraction provided that $\delta_0$ is sufficiently close to $t_0$, hence $\Gamma_0$ has a unique fixed point.
But, if $\delta_0-t_0$ is sufficiently small, then it can be easily seen that  $f(\,\cdot\,+t_0)|_{[0,\delta_0-t_0]}$ and $h(\,\cdot\,+t_0)|_{[0,\delta_0-t_0]}$ both belong to $\bM_0$ and these functions are therefore fixed points of $\Gamma_0$. This implies that $f=h$ on $[0,\delta_0]$ for some $\delta_0>t_0$, in  contradiction with  the definition of $t_0.$ 
This proves the uniqueness  claim.
\end{proof}

We are now in a position to prove Theorem \ref{T:A}.
\begin{proof}[Proof of Theorem \ref{T:A}]
Let $f_0\in\cO_\alpha$ be given.  According to Proposition \ref{T:B1} (with $\ov f:=f_0$), there exists $\delta>0$ and a classical solution 
\[
 f\in {  C}([0,\delta],\cO_\alpha)\cap{  C}((0,\delta],\E_1)\cap {  C}^1((0,\delta],\E_0)\cap    {  C}^{ \alpha-\beta}([0,\delta],\E_\beta) 
 \]
 to \eqref{QP'}. This solution can be continued as follows. 
 Applying Proposition \ref{T:B1} (with  $\ov f:=f(\delta)$), we find $r>0$ and $\delta_1>0$ such that 
  \begin{align}\label{CPas} 
\dot h =\Phi(h)[h],\quad t\in(0,\delta_1], \qquad h(0)=f_1
\end{align}  
has a classical solution $h\in {  C}([0,\delta_1],\cO_\alpha)\cap { C}((0,\delta_1],\E_1)\cap {  C}^1((0,\delta_1],\E_0)\cap  {  C}^{ \alpha-\beta}([0,\delta_1],\E_\beta)$ for each
$f_1\in\cO_\alpha$ with $\|f_1-f(t_0)\|_\alpha\leq r.$
Let $t_0\in(0,\delta)$ be such that 
\[t_0+\delta_1>\delta \qquad \text{and}\qquad \|f(t_0)-\ov f\|_\alpha\leq r_.\]
Hence it is possible to chose $f_1:=f(t_0)$ as an initial value in \eqref{CPas}. 
Since $f(\,\cdot\,+t_0):[0,\delta-t_0]\to\cO_\alpha$
and     $h:[0,\delta-t_0]\to\cO_\alpha$ are both classical solutions to
 \begin{align*} 
\dot h =\Phi(h)[h],\quad t\in(0,\delta-t_0], \qquad h(0)=f(t_0),
\end{align*}
 by Proposition \ref{T:B1} they must coincide. 
Consequently, the function $F:[0,t_0+\delta_1]\to \cO_\alpha$ defined by
\[
F(t):=
\left\{
\begin{array}{lll}
f(t)&,& t\in[0,\delta],\\[1ex]
h(t-t_0)&,& t\in[\delta, t_0+\delta_1]
\end{array}
\right.
\]
is a classical solution to \eqref{QP'} which extends $f$.
The maximal solution $f=f(\,\cdot\,; f_0):I(f_0)\to\cO_\alpha$  in Theorem \ref{T:A} is defined by setting
\begin{align*}
 &I(f_0):=\bigcup \big\{[0,\delta]\,:\, \text{\eqref{QP'} has a classical solution $f_\delta $ on $[0,\delta]$ with $f_\delta\in {C}^{\alpha-\beta}([0,\delta],\E_\beta)$}\big\}\\[1ex] 
 &f(t):=f_\delta(t)\qquad\text{for $t\in[0,\delta]$}.
\end{align*}
The construction  above  shows that $f$ is well-defined and that $I(f_0)=[0,T_+(f_0))$ with $T_+(f_0)\leq\infty$.
 This proves the existence claim in Theorem \ref{T:A}.
The uniqueness assertion  is an immediate consequence of Proposition \ref{T:B1}. \medskip

  We now prove the criterion for global existence. Hence, let us assume that the unique classical maximal solution  $f=f(\,\cdot\,;f_0):[0,T_+(f_0))\to\cO_\alpha$  to \eqref{QP'}
  is uniformly continuous when restricted to each interval $[0,T]\cap [0,T_+(f_0)),$ with $T>0$ arbitrary.
  We further assume that $\tau:=T_+(f_0)<\infty$, otherwise we are done. 
  Then, since $f$ is uniformly continuous on $[0,\tau)$, it is straightforward to see that the limit 
  $$f(\tau):=\underset{t\nearrow \tau}\lim f(t)$$ exists in $\ov\cO_\alpha$.
  If ${\rm dist}(f(t),\p\cO_\alpha)\not\to 0$ for $t\to \tau$, it must hold that $f(\tau)\in \cO_\alpha$.
  Proceeding as above, we may extend in view of Proposition \ref{T:B1} this maximal solution to an interval $[0,\tau+\delta_1)$ for some $\delta_1>0$, which is a contradiction and we are done.
  
 Finally, the semiflow property of the solution map $[(t,f_0)\mapsto f(t;f_0)]$ stated at the end of  Theorem \ref{T:A} 
 is proven in detail in \cite[Theorem 8.1]{Am88}. Furthermore, if $\Phi$ is additionally smooth, then proceeding as in  \cite[Theorem 11.3]{Am88} one may show that the semiflow 
 map is also smooth. For real-analytic $\Phi$ the real-analyticity of  $[(t,f_0)\mapsto f(t;f_0)]$ follows by  estimating the Fr\'echet derivatives of the flow map, 
 which is a rather tedious and lengthy procedure which we refrain from presenting here.
\end{proof}

\vspace{0.5cm}
\hspace{-0.5cm}{ \bf Acknowledgements} 
The author would like to thank Christoph Thiele  for the discussion on an issue related to the analysis in Section \ref{S2}.

\bibliographystyle{abbrv}
\bibliography{B}
\end{document}